\documentclass{amsart}

\usepackage{epsfig}
\usepackage{amscd}
\usepackage{amsmath, amssymb, amsthm}
\usepackage{graphics}
\usepackage{color}
\usepackage{dsfont}
\newtheorem*{main-theorem}{Main Theorem}
\newtheorem*{main-theorem'}{Main Theorem'}
\newtheorem{proposition}{Proposition}[section]
\newtheorem{theorem}{Theorem}
\newtheorem*{old-thm}{Theorem}
\newtheorem{lemma}[proposition]{Lemma}
\newtheorem{corollary}[proposition]{Corollary}

\theoremstyle{definition}

\newtheorem{remark}[proposition]{Remark}
\newtheorem{claim}[proposition]{Claim}

\numberwithin{equation}{section}

\def\ZZ{{\mathbb Z}}
\def\NN{{\mathbb N}}
\def\reals{{\mathbb R}}
\def\cx{{\mathbb C}}

\def\Ci{{\mathcal C}^\infty}
\def\Re{\,\mathrm{Re}\,}
\def\Im{\,\mathrm{Im}\,}

\def\WF{\mathrm{WF}_h}

\def\supp{\mathrm{supp}\,}

\def\id{\,\mathrm{id}\,}

\def\O{{\mathcal O}}
\def\SS{{\mathbb S}}
\def\s{{\mathcal S}}

\def\cchar{\mathrm{char}\,}
\def\nbhd{\mathrm{neigh}\,}
\def\neigh{\mathrm{neigh}\,}
\def\Op{\mathrm{Op}\,}

\def\esssupp{\text{ess-supp}\,}
\def\ad{\mathrm{ad}\,}

\def\hsc{{\left( \tilde{h}/h \right)}}

\def\csh{{\left( h/\tilde{h} \right)}}
\def\phi{\varphi}
\def\half{{\frac{1}{2}}}
\def\dist{\text{dist}\,}
\def\diag{\text{diag}\,}

\def\be{\begin{eqnarray*}}
\def\ee{\end{eqnarray*}}
\def\ben{\begin{eqnarray}}
\def\een{\end{eqnarray}}
\def\ker{\text{ker}}
\def\km{\text{ker}_{m_0(z)}(P-z)}
\def\kmm{\text{ker}_{m_1(z)}(P-z)}
\def\MM{\mathcal{M}}
\def\lll{\left\langle}
\def\rrr{\right\rangle}
\def\rqf{\right\rangle_{QF}}
\def\PP{\mathcal{P}}
\def\DD{\mathcal{D}}
\def\EE{\mathcal{E}}
\def\L2R{L_{\text{Rest}}^2}
\def\contraction{\lrcorner}
\def\11{\mathds{1}}

\def\tM{\widetilde{M}}
\def\tz{\tilde{z}}

\def\L2c{L^2_{\text{comp}}}

\def\tkappa{\tilde{\kappa}}
\def\hyp{\text{hyp}}
\def\th{\tilde{h}}
\def\el{\text{ell}}
\def\tDelta{\widetilde{\Delta}}
\def\tV{\widetilde{V}}
\def\sech{\text{sech}}
\def\Diff{\text{Diff}}

\begin{document}
\title[Quantum Monodromy]{Quantum Monodromy and Non-concentration near
  a Closed Semi-Hyperbolic Orbit}
\author{Hans Christianson}
\address{Department of Mathematics, Massachusetts Institute of
Technology, 77 Mass. Ave., Cambridge, MA 02139, USA}
\email{hans@math.mit.edu}

\begin{abstract}
For a large class of semiclassical operators $P(h)-z$
which 
includes Schr\"odinger operators on manifolds with boundary,
we construct the Quantum Monodromy operator $M(z)$ associated to a
periodic orbit $\gamma$ of
the classical flow.  Using estimates relating $M(z)$ and $P(h)-z$, we
prove semiclassical estimates for small complex perturbations of $P(h)
-z$ in the case $\gamma$ is semi-hyperbolic.  As our main
application, we give logarithmic lower bounds on the mass of
eigenfunctions away from semi-hyperbolic orbits of the associated
classical flow.

As a second application of the Monodromy Operator construction, we prove if
$\gamma$ is an elliptic orbit, then $P(h)$ admits quasimodes which are
well-localized near $\gamma$.
\end{abstract}

\maketitle

\addcontentsline{toc}{section}{Table of Contents}

\tableofcontents



\section{Introduction}
\numberwithin{equation}{section}
\label{introduction}
\subsection{Statement of Results}
To motivate our general results, we first present a few applications.
Suppose $(X,g)$ is a compact 
Riemannian manifold with or without boundary.  Let $-\Delta_g$ be the
Laplace-Beltrami operator on $X$ and assume 
$u$ solves the eigenvalue problem
\be
- \Delta_g u = \lambda^2 u, \,\,\, \| u \|_{L^2(X)} = 1.
\ee
Assume $\gamma$ is a closed semi-hyperbolic geodesic satisfying
either $\gamma \cap \partial X = \emptyset$, or the reflection at the
boundary is transversal.  Then if $U$ is a
sufficiently small neighbourhood of $\gamma$, we prove
\ben
\label{non-conc-lambda}
\int_{X \setminus U} |u|^2 dx \geq \frac{C}{\log |\lambda|}, \,\,\,
|\lambda| \to \infty.
\een

From \cite{Ch}, we have an application to exponential decay of $L^2$
energy for the damped wave equation: suppose $a(x)$ is positive outside of
$U$ and $u$ satisfies
\begin{eqnarray*}
\left\{ \begin{array}{l}
\left( \partial_t^2 - \Delta + 2a(x) \partial_t \right) u(x,t) = 0, \quad (x,t) \in X \times (0, \infty) \\
u(x,0) = 0, \quad \partial_t u(x,0) = f(x).
\end{array} \right.
\end{eqnarray*}
Then 
\be
\| \partial_t u \|_{L^2(X)}^2 + \left\| \nabla u
\right\|_{L^2(X)}^2 \leq C e^{-t/C} \|f\|_{H^\epsilon(X)}^2,
\ee
for each $\epsilon>0$.  

In addition, we have two dispersive type estimates from
\cite{Ch-str}.  The first is a local smoothing estimate for the
Schr\"odinger equation.  Suppose $X$ is
a non-compact manifold which is asymptotically Euclidean, and there is a hyperbolic closed geodesic $\gamma
\subset U \Subset X$.  Then for every
$\epsilon >0$
\be
\int_0^T \left\| \rho_s e^{it (\Delta_g - V(x)) } u_0 \right\|_{H^{1/2 -
    \epsilon}(X)}^2 dt  \leq C \| u_0 \|_{L^2(X)}^2,
\ee
where $\rho_s \in \Ci(M)$ satisfies
\be
\rho_s(x) \equiv \lll d_g(x, x_0 )\rrr^{-s}
\ee
for $x_0$ fixed and $x$ outside a compact set, and $V \in \Ci(M)$, $0 \leq V \leq C$
satisfies 
\be
|\nabla V| \leq C \lll \dist(x,x_0) \rrr^{-1-\delta}
\ee
for some $\delta>0$.

The second dispersive estimate is a sub-exponential local energy decay rate for
solutions to the wave equation in odd dimensions $n \geq 3$.  Suppose
$X$ is a non-compact Riemannian manifold which is Euclidean outside a
compact set, and suppose $u$ solves
\be
\left\{ \begin{array}{l} (-D_t^2 - \Delta_g + V(x) ) u(x,t) = 0, \,\,\, X
    \times [0,\infty) \\
u(x,0) = u_0 \in H^1(X), \,\, D_t u(x,0) = u_1 \in L^2(X) \end{array}
\right.
\ee
for $u_0$ and $u_1$ smooth, compactly supported, where $V\in \Ci(M)$ satisfies 
\be
\exp ( -\dist_g (x,x_0)^{2}) V = o(1).
\ee 
Let $\psi \in
\Ci(X)$ satisfy
\be
\psi \equiv \exp ( -\dist_g (x,x_0)^{2})
\ee
for $x$ outside a compact set and $x_0$ fixed.  Then 
\be
\lefteqn{ \left\| \psi \partial_t u
\right\|_{L^2(X)}^2 + \left\| \psi u \right\|_{H^{1}(X)}^2} \\ &&  \leq C 
e^{-t^{1/2}/C} \left(   \|\partial_t u(x,0)\|_{H^\epsilon(X)}^2 + \|u(x,0)\|_{H^{1+\epsilon}(X)}^2
\right).
\ee

For the general statement of results, let $X$ be a smooth, compact
manifold.  In this introduction, we state the Main Theorem only in the
case $\partial X = \emptyset$.  The case with boundary will be
considered in \S \ref{manifold-boundary-1}.  We take $P(h) \in \Psi_h^{k,0}$ for $k
\geq 1$ and assume $P(h)$ is of real principal type.  That is, if $p = \Ci (T^*X)$ is the principal symbol of
$P(h)$, then $p$ is real-valued, independent of $h$.  Assume $p^{-1}(E)$
is a smooth, compact hypersurface and 
$dp(x, \xi) \neq 0$ for energies $E$ near $0$.  We assume $p$ is classically elliptic
outside of a compact subset of $T^*X$: there exists $C>0$ such that 
\be
|\xi|>C \implies p(x,\xi) \geq \lll \xi \rrr^k/C.  
\ee
We refer the reader to \S
\ref{preliminaries} for definitions.

Let $\Phi_t = \exp tH_p$ be
the Hamiltonian flow of $p$, and suppose $\Phi_t$ has a closed,
semi-hyperbolic orbit $\gamma \subset \{p=0\}$ of period $T$.  The
assumption that $\gamma$ be semi-hyperbolic means if $N$ is a Poincar{\'e}
section for $\gamma$ and $S:N \to S(N)$ is the
Poincar{\'e} map, then the linearization of $S$, $dS(0)$, is
nondegenerate and has at least one eigenvalue off the unit circle.
For the eigenvalues of modulus $1$ we also require the following
nonresonance assumption for energies near $0$:
\ben
\label{non-res}
\left\{ \begin{array}{l}
\text{if } e^{\pm i \alpha_1}, e^{\pm i \alpha_2}, \ldots, e^{\pm i \alpha_k} \text{ are eigenvalues of
modulus 1, then } \\   \alpha_1 , \alpha_2, \ldots,
 \alpha_k 
\text{ are independent over } \pi \ZZ.
\end{array} \right.
\een

We prove for a family of eigenfunctions $u(h)$ for $P(h)$,
\be
P(h)u(h) = E(h) u(h), \,\,\, E(h) \to 0 \,\, \text{as} \,\, h \to 0,
\ee
$u(h)$ has its mass concentrated away from $\gamma$.  This is made
precise in the following theorem.
\begin{main-theorem}
Let $A \in \Psi_h^{0,0}(X)$ be a pseudodifferential operator whose
principal symbol is $1$ near $\gamma$ and $0$ away from $\gamma$.  There exist constants $h_0>0$ and $C>0$ such that 
\be
\| u \| \leq C \frac{ \sqrt{\log ( 1/h )}} h \| P ( h ) u \| + C
\sqrt{ \log (1/h ) } \| ( I - A ) u \|
\ee
uniformly in $0<h<h_0$, where the norms are $L^2$ norms on $X$.  In particular, if $u(h)$
satisfies 
\be
\left\{ \begin{array}{l} P(h) u(h) = \O(h^\infty); \\
\|u(h)\|_{L^2(X)} = 1, \end{array} \right.
\ee
\be
\left\| (I - A) u \right\|_{L^2(X)} \geq \frac{1}{C} 
\left(\log \left(1/h\right)\right)^{-\half} \,, \ \ 0 < h < h_0 .
\ee
\end{main-theorem}

\begin{remark}
In \S \ref{manifold-boundary-1} we assume $P(h) \in \Diff_h^2(X)$
is a differential operator on $X$ and that $\partial X$ is
noncharacteristic with respect to the principal symbol of $P(h)$.
Then a similar conclusion to the Main Theorem holds (see Main Theorem'
in that section).
\end{remark}

\begin{remark}
In \S \ref{quasi}, we give a partial converse to the Main
Theorem in Theorem \ref{main-theorem-1}.  That is, the techniques of the proof
of the Main Theorem are used to show if the periodic orbit $\gamma$
is elliptic, then $P(h)$ admits quasimodes which are well-localized to
$\gamma$.
\end{remark}

\begin{remark}
The estimates in this work are all microlocal in nature, hence we lose
nothing by assuming $X$ is compact.  In order to apply these estimates
in the case of non-compact manifolds, we assume $P$ is classically
elliptic and the geometry is non-trapping outside of a compact
submanifold and then apply our results there.  See \cite{Ch-str} for more on this.
\end{remark}

The Main Theorem is the similar to \cite[Main Theorem]{Ch} with three generalizations,
namely that we no longer assume the linearized Poincar\'e map
has no negative eigenvalues, we allow some eigenvalues of modulus
$1$, and in \S \ref{manifold-boundary-1} we allow $\gamma$ to reflect transversally off 
$\partial X$ with some extra assumptions on $P(h)$.  This allows study of, for example, billiard
problems in any dimension.  The problems encountered in \cite{Ch} 
with these cases come from attempting to put $p$ into a
normal form in a neighbourhood of $\gamma \subset T^*X$.  

The
motivation for the proof in this paper is to reduce the
problem of studying the resolvent $(P-z)^{-1}$ in a microlocal 
neighbourhood of $\gamma$ to studying a related operator on the
Poincar{\'e} section $N$.  

If we identify $N$ with $T_0^*N \simeq \reals^{2n-2}$ near $0$,
we are led to study operators acting on $L^2(V)$, where $V \subset
\reals^{n-1}$.  In the course of this work, we will see the relevent
object of study is the Quantum Monodromy operator $M(z):L^2(V)
\to L^2(V)$.  By setting up a Grushin problem in a
neighbourhood of 
\be
\gamma \times (0,0) \subset T^*X \times T^*\reals^{n-1},
\ee
and using the microlocal inverse constructed by Sj\"ostrand-Zworski in \cite{SjZw}, we will see it is 
sufficient to bound $\| I - M(z) \|_{L^2(V) \to L^2(V)}$ from
below.  This will result in the following theorem.
\begin{theorem}
\label{main-theorem-5}
Suppose $P \in \Psi_h^{k,0}$ is a semiclassical pseudodifferential
operator of real principal type satisfying all of the assumptions of
the introduction, and assume $\gamma \cap \partial X = \emptyset$.  Then there exist positive constants $C$, $c_0$,
$h_0$, $\epsilon_0$, and a positive integer $N$ such that for $0 < h < h_0$, $z \in
[-\epsilon_0, \epsilon_0] + i(-c_0h, c_0h)$, if $u \in L^2(X)$ has
$h$-wavefront set sufficiently close to $\gamma$, then
\ben
\label{main-theorem-5-est}
\left\| (P-z) u \right\|_{L^2(X)} \geq C^{-1} h^{N} \|u\|_{L^2(X)}.
\een
\end{theorem}

Theorem \ref{main-theorem-5} allows us
to add a complex absorption term of order $h$ supported away from
$\gamma$.  Let $a \in \Ci(T^*X)$ equal $0$ in a neighbourhood of $\gamma$ and $1$ away
from $\gamma$, and define
\ben
\label{Q(z)}
Q(z) u = P(h) - z - ih Ca^w
\een
for a constant $C>0$ to be chosen later.  Then a semiclassical adaptation of the ``three-lines'' theorem from complex analysis, 
will allow us to deduce the following estimate.
\begin{theorem}
\label{main-theorem-2}
Suppose $Q(z)$ is given by \eqref{Q(z)}, and $z \in [-\epsilon_0/2 ,
\epsilon_0/2] \Subset \reals$.  Then there 
is $h_0 >0$ and $0 < C < \infty$ such that for $0 < h < h_0 $,
\begin{eqnarray}
\label{main-theorem-2-est-1}
\left\| Q(z)^{-1} \right\|_{L^2(X) \to L^2(X)} \leq C \frac{\log (1/h) }{h}.
\end{eqnarray}
If $\phi \in \Ci_c(X)$ is supported away from $\gamma$, then
\begin{eqnarray}
\label{main-theorem-2-est-2}
\left\| Q(z)^{-1} \phi \right\|_{L^2(X) \to L^2(X)} \leq C \frac{\sqrt{ \log (1/h)}}{h}.
\end{eqnarray}
\end{theorem}

\subsection{Examples}


There are many examples in which the hypotheses of the theorem are
satisfied, the simplest of which is the case in which $p = |\xi|^2- E(h)$ for $E(h)>0$.
Then the Hamiltonian flow of $p$ is the geodesic flow, so if the
geodesic flow has a closed semi-hyperbolic orbit, there is
non-concentration of eigenfunctions, $u(h)$, for the equation
\be
-h^2 \Delta u(h) = E(h) u(h).
\ee
Another example of such a $p$ is the case $p = |\xi|^2 + V(x)$, where
$V(x)$ is a confining potential with two ``bumps'' or ``obstacles''
in the lowest energy level (see Figure \ref{fig:fig13}).  In the
appendix to \cite{Sjo2a} it is shown that for an interval of energies
$V(x) \sim 0$, there is a closed hyperblic orbit $\gamma$ of the Hamiltonian
flow which ``reflects'' off the bumps (see Figure \ref{fig:fig14}).
Complex hyperbolic orbits may be constructed by considering $3$-dimensional
hyperbolic billiard problems (see, for example, \cite[\S 2]{AuMa}).  In
addition, Proposition 4.3 from \cite{Ch} gives a somewhat artificial means of constructing a manifold diffeomorphic to a neighbourhood in $T^*\SS_{(t, \tau)}^1
\times T^* \reals_{(x, \xi)}^{n-1}$ which contains a hyperbolic orbit $\gamma$
by starting with the Poincar\'{e} map $\gamma$ is to have.

\begin{figure}
\hfill
\begin{minipage}[t]{.45\textwidth}
\centerline{\epsfig{figure=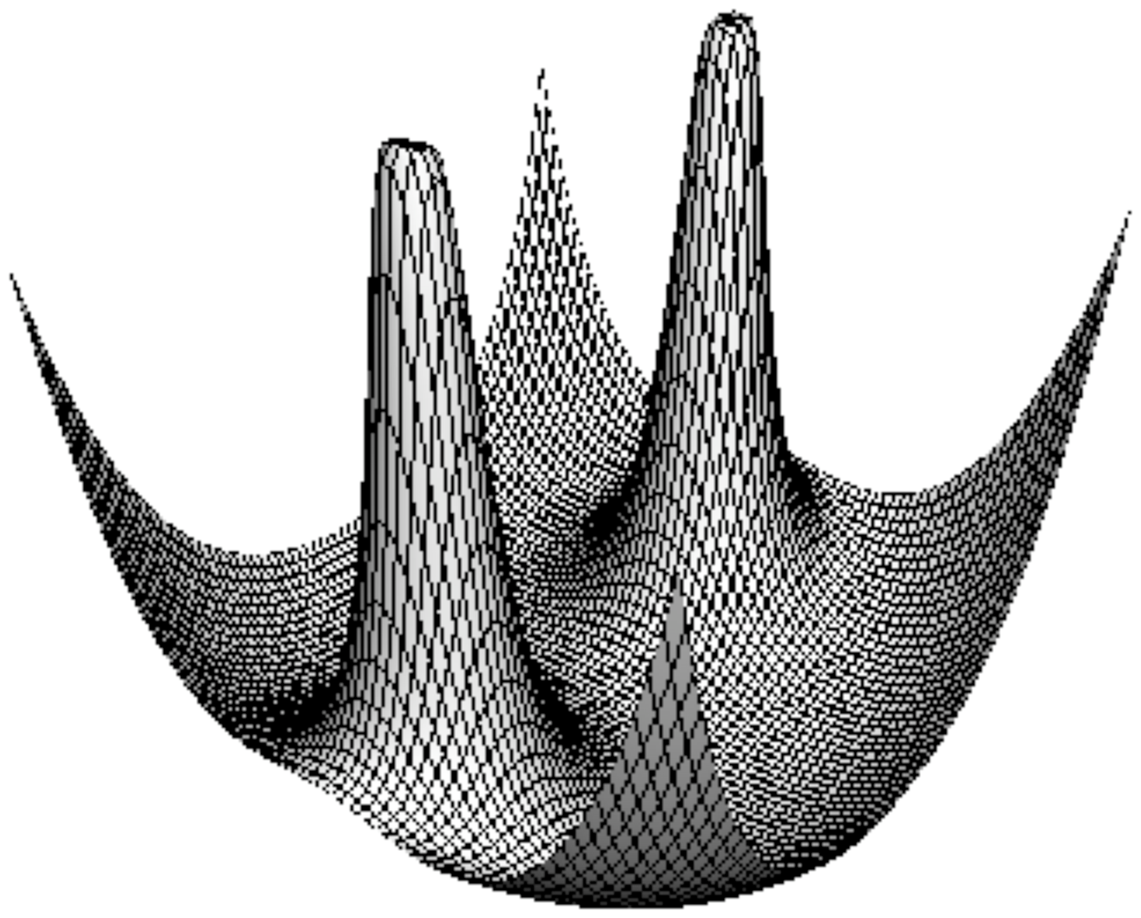, height=2in, width=3in}}
\caption{\label{fig:fig13} A confining potential $V(x)$ with two
  bumps at the lowest energy level $E <0$.}
\end{minipage}
\hfill
\begin{minipage}[t]{.45\textwidth}
\centerline{
\input{fig14a}}
\caption{\label{fig:fig14} The level set $V(x) = 0$ and the closed
  hyperbolic orbit $\gamma$ reflecting off the ``soft'' boundary.}
\end{minipage}
\hfill
\end{figure}

In the appendix, we examine the Riemannian manifold
\be
M = \reals_x / \ZZ \times \reals_y \times \reals_z
\ee
equipped with the metric
\be
ds^2 = \cosh^2 y (2z^4 -z^2 +1)^2 dx^2 + dy^2 + dz^2,
\ee
which has a semi-hyperbolic closed geodesic at $y=z=0$ and two
hyperbolic closed geodesics 
at $y=0, z=\pm 1/2$.  Restricting $y$ and
$z$ to compact intervals yields a compact manifold satisfying the
hypotheses of the Main Theorem, while the non-compact manifold
provides 
a model for possibly extending the dispersive-type estimates to the semi-hyperbolic case.

\subsection{Organization}

This work is organized as follows.  In \S \ref{preliminaries} we
recall 
basic facts from the calculus of semiclassical pseudodifferential
operators on manifolds and \S \ref{FIO} contains some results from the
theory of $h$-Fourier Integral Operators ($h$-FIOs).  \S
\ref{manifold-boundary-1} reviews the 
classical picture of a closed orbit reflecting transversally off the
boundary and contains the details of a propagation of singularities
result and the statement of the Main Theorem in the case of manifolds
with boundary.  \S \ref{mono} gives the definition and basic facts about the
Quantum Monodromy operator $M(z)$, while \S \ref{grushin} shows how
$M(z)$ arises naturally in the context of a Grushin problem.  \S
\ref{model-chapter} presents the main ideas of the proof of Theorem
\ref{main-theorem-5} by considering a
model.  In \S \ref{linearization}-\ref{theorem-proof-chapter} the
proof of Theorem \ref{main-theorem-5} is
presented, while the proof of Theorem \ref{main-theorem-2} and the Main Theorem is reserved for \S
\ref{main-theorem-proof-chapter}.  Finally, in \S \ref{quasi}, we show how the
Monodromy operator construction can be used to construct well-localized quasimodes
if $\gamma$ is elliptic.  In the appendix, we  
review some facts from the theory of ODEs, and present some tools from
symplectic geometry, and provide a concrete example of a
semi-hyperbolic orbit.

\subsection{Acknowledgements}

The author would like to thank Maciej Zworski for much help
  and support during the writing of this work, as well as Nicolas
  Burq for suggesting study of the monodromy operator as a means of
  tackling the boundary problem.  He would also like to thank Herbert Koch for
  suggesting the generalization to semi-hyperbolic orbits, and Michael
  Hitrik for much help in working out the model case for Theorem
  \ref{main-theorem-1}.  The majority of this work was conducted while the author was a graduate student in the Mathematics
  Department at UC-Berkeley and he is very grateful for the support
  received 
  while there.



\section{Preliminaries}
\label{preliminaries}

\numberwithin{equation}{section}

This section comes almost directly from \cite[\S 2]{Ch} and the references
cited therein, but we include it
here for completeness.

\subsection{$h$-Pseudodifferential Operators on Manifolds}
 We will be operating on half-densities,
\begin{eqnarray*}
u(x)|dx|^{\frac{1}{2}} \in \Ci\left(X, \Omega_X^{\frac{1}{2}}\right),
\end{eqnarray*}
with the informal change of variables formula
\begin{eqnarray*}
u(x)|dx|^{\frac{1}{2}} = v(y)|dy|^\half, \,\, \text{for}\,\, y = \kappa(x) 
\Leftrightarrow v(\kappa(x))|\kappa'(x)|^\half = u(x).
\end{eqnarray*}
By symbols on $X$ we mean
\begin{eqnarray*}
\lefteqn{\s^{k,m} \left(T^*X, \Omega_{T^*X}^\half \right):= } \\
& = &\left\{ a \in \Ci(T^*X \times (0,1], \Omega_{T^*X}^\half): \left| \partial_x^\alpha 
\partial_\xi^\beta a(x, \xi; h) \right| \leq C_{\alpha \beta}h^{-m} \langle \xi \rangle^{k - |\beta|} \right\}.
\end{eqnarray*}
We remark that the symbols we consider are half-densities on the
cotangent bundle with the natural symplectic structure, and in the course of this work we will use only
symplectic changes of variables on $T^*X$.  Consequently the change of
variables formula is invariant on symbols: if $\kappa: T^*X \to T^*X$ is
symplectic and $a \in \s^{k,m}\left(T^*X, \Omega_{T^*X}^\half \right)$,
\be
\kappa^* \left[ a(x, \xi) |d \xi \wedge dx |^{\half} \right] & = & a( \kappa(x, \xi))
| \kappa^* d \xi \wedge dx |^\half \\
& = &  a( \kappa(x, \xi)) | d \xi \wedge dx |^\half.
\ee
Hence we don't keep track of the $| d \xi \wedge dx|^\half$ except
where confusion may arise.

There is a corresponding class of pseudodifferential operators $\Psi_h^{k,m}(X, \Omega_X^\half)$ 
acting on half-densities defined by the local formula (Weyl calculus) in $\reals^n$:
\begin{eqnarray*}
\Op_h^w(a)u(x) = \frac{1}{(2 \pi h)^n} \int \int a \left( \frac{x + y}{2}, \xi; h \right) 
e^{i \langle x-y, \xi \rangle / h }u(y) dy d\xi.
\end{eqnarray*}
We will occasionally use the shorthand notations $a^w := \Op_h^w(a)$ and $A:=\Op_h^w(a)$ when 
there is no ambiguity in doing so.  We also use the notation $P \in
\Diff_h^k$ when $P$ is a semiclassical differential operator.

We have the principal symbol map
\begin{eqnarray*}
\sigma_h : \Psi_h^{k,m} \left( X, \Omega_X^\half \right) \to \s^{k,m} \left/ \s^{k, m-1} 
\left(T^*X, \Omega_{T^*X}^\half \right) \right.,
\end{eqnarray*}
which gives the left inverse of $\Op_h^w$ in the sense that 
\begin{eqnarray*}
\sigma_h \circ \Op_h^w: \s^{k,m} \to \s^{k,m}/\s^{k, m-1} 
\end{eqnarray*}
is the natural projection.  Acting on half-densities in the Weyl calculus, the principal 
symbol is actually well-defined in $\s^{k,m} / \s^{k, m-2}$, that is, up to $\O(h^2)$ in 
$h$ (see, for example \cite[Theorem D.3]{EvZw}).  

We will use the notion of wave front sets for pseudodifferential operators on manifolds.  
If $a \in \s^{k,m}(T^*X, \Omega_{T^*X}^\half)$, we define the singular support or essential support for $a$:
\begin{eqnarray*}
\esssupp_h a \subset T^*X \bigsqcup \SS^*X,
\end{eqnarray*}
where $\SS^*X = (T^*X \setminus \{0\}) / \reals_+$ is the cosphere bundle (quotient taken with 
respect to the usual multiplication in the fibers), and the union is disjoint.  $\esssupp_h a$ is defined using complements:
\begin{eqnarray*}
\lefteqn{\esssupp_h a := } \\
& = & \complement \left\{ (x, \xi) \in T^*X : \exists \epsilon >0, \,\,\, \partial_x^\alpha 
\partial_\xi^\beta a(x', \xi') = \O(h^\infty), \,\,\, d(x, x') + |\xi - \xi'| < \epsilon \right\} \\
&& \bigcup \complement \{ (x, \xi) \in T^*X \setminus 0 : \exists \epsilon > 0, \,\,\, \partial_x^\alpha 
\partial_\xi^\beta a(x', \xi') = \O (h^\infty \langle \xi \rangle^{-\infty}),  \\
&& \quad \quad \quad d(x, x') + 1 / |\xi'| + | \xi/ |\xi| - \xi' /
|\xi'| | < \epsilon \} / \reals_+.
\end{eqnarray*}
We then define the wave front set of a pseudodifferential operator $A \in \Psi_h^{k,m}( X, \Omega_X^\half )$:
\begin{eqnarray*}
\WF(A) : = \esssupp_h(a), \,\,\, \text{for} \,\,\, A = \Op_h^w(a).
\end{eqnarray*}
Finally for distributional half-densities $u \in \Ci( (0, 1]_h, \mathcal{D}'(X, \Omega_X^\half))$ 
such that there is $N_0$ so that $h^{N_0}u$ is bounded in $\mathcal{D}'(X, \Omega_X^\half)$, we can 
define the semiclassical wave front set of $u$, again by complement:
\begin{eqnarray*}
\lefteqn{\WF (u) := } \\
&= & \complement \{(x, \xi) : \exists A \in \Psi_h^{0,0}, \,\, \text{with} \,\, \sigma_h(A)(x,\xi) \neq 0, \,\,  \\ 
&& \quad \text{and} \,\,Au \in h^\infty \Ci((0,1]_h, \Ci(X, \Omega_X^\half)) \}.
\end{eqnarray*}
For $A = \Op_h^w(a)$ and $B = \Op_h^w(b)$, $a \in \s^{k,m}$, $b \in \s^{k',m'}$ we have the composition 
formula 
\begin{eqnarray}
\label{Weyl-comp}
A \circ B = \Op_h^w \left( a \# b \right),
\end{eqnarray}
where
\begin{eqnarray}
\label{a-pound-b}
\s^{k + k', m+m'} \ni a \# b (x, \xi) := \left. e^{\frac{ih}{2} \omega(D_x, D_\xi; D_y, D_\eta)} 
\left( a(x, \xi) b(y, \eta) \right) \right|_{{x = y} \atop {\xi = \eta}} ,
\end{eqnarray}
with $\omega$ the standard symplectic form.

We will need the definition of microlocal equivalence of operators.  Suppose 
$T: \Ci(X, \Omega_X^\half) \to \Ci(X, \Omega_X^\half)$ and that for any seminorm $\| \cdot \|_1$ on 
$\Ci(X, \Omega_X^\half)$ there is a second seminorm $\| \cdot \|_2$ on $\Ci(X, \Omega_X^\half)$ such that 
\begin{eqnarray*}
\| Tu\|_1 = \O(h^{-M_0})\|u \|_2
\end{eqnarray*}
for some $M_0$ fixed.  Then we say $T$ is {\it semiclassically tempered}.  We assume for the rest of 
this work that all operators satisfy this condition.  Let $U,V \subset T^*X$ be open precompact sets.  
We think of operators defined microlocally near $V \times U$ as equivalence classes of tempered operators.  
The equivalence relation is
\begin{eqnarray*}
T \sim T' \Longleftrightarrow A(T-T')B = \O(h^\infty): \mathcal{D}'\left( X, \Omega_X^\half \right) \to \Ci 
\left(X, \Omega_X^\half \right)
\end{eqnarray*}
for any $A,B \in \Psi_h^{0,0}(X, \Omega_X^\half)$ such that 
\begin{eqnarray*}
&& \WF (A) \subset \widetilde{V}, \quad \WF (B) \subset \widetilde{U}, \,\, \text{with} \,\, \widetilde{V}, 
\widetilde{U} \,\, \text{open and } \\
&& \quad \quad \overline{V} \Subset \widetilde{V} \Subset T^*X, \quad \overline{U} \Subset 
\widetilde{U} \Subset T^*X.
\end{eqnarray*}
In the course of this work, when we say $P=Q$ {\it microlocally} near $V \times U$, we mean for any $A$, $B$ as above,
\begin{eqnarray*}
APB - AQB = \O_{L^2 \to L^2}\left( h^\infty \right),
\end{eqnarray*}
or in any other norm by the assumed precompactness of $U$ and $V$.  Similarly, we say $B = T^{-1}$ on $V \times U$ if 
$BT = I$ microlocally near $V \times V$ and $TB = I$ microlocally near
$U \times U$. 

We will need the following semiclassical version of Beals's Theorem:  
Recall for operators $A$ and $B$, the notation $\ad_BA$ is defined as 
\begin{eqnarray*}
\ad_B A = \left[B, A \right].
\end{eqnarray*}
\begin{old-thm}[Beals's Theorem]
Let $A: \s \to \s'$ be a continuous linear operator.  Then $A = \Op_h^w(a)$ for a symbol $a \in \s^{0,0}$ if and only 
if for all $N \in \mathbb{N}$ and all linear symbols $l_1, \ldots l_N$, 
\begin{eqnarray*}
\ad_{\Op_h^w(l_1)} \circ \ad_{\Op_h^w(l_2)} \circ \cdots \circ \ad_{\Op_h^w(l_N)} A = \O(h^N)_{L^2 \to L^2}.
\end{eqnarray*}
\end{old-thm}

\subsection{Symbols with $2$ Parameters}

We will use the following results on symbols with two parameters.  We will only use symbol spaces with two
parameters in the context of microlocal estimates, in which case we
may assume we are working in an open subset of $\reals^{2n}$.  We define the 
following spaces of symbols with two parameters:
\begin{eqnarray*}
\lefteqn{\s^{k,m, \widetilde{m}} \left( \reals^{2n} \right) := } \\
& = &\Big\{ a \in \Ci \left( \reals^{2n} \times (0,1]^2 \right):   \\ 
&& \quad \quad  \left| \partial_x^\alpha \partial_\xi^\beta a(x, \xi; h, \tilde{h}) \right| 
\leq C_{\alpha \beta}h^{-m}\tilde{h}^{-\widetilde{m}} \langle \xi \rangle^{k - |\beta|} \Big\}.
\end{eqnarray*}
For the applications in this work, we assume $\tilde{h} >h$ and
define the scaled spaces:
\begin{eqnarray*}
\lefteqn{\s_{\delta}^{k,m, \widetilde{m}} \left(\reals^{2n} \right):= } \\
& = & \Bigg\{ a \in \Ci \left(\reals^{2n} \times (0,1]^2 \right):  \\
 && \quad \quad  \left| \partial_x^\alpha \partial_\xi^\beta a(x, \xi; h, \tilde{h}) \right| 
\leq C_{\alpha \beta}h^{-m}\tilde{h}^{-\widetilde{m}} \left( \frac{\tilde{h}}{h} \right)^{\delta(|\alpha| + |\beta|)} 
\langle \xi \rangle^{k - |\beta|} \Bigg\}.
\end{eqnarray*}
As before, we have the corresponding spaces of semiclassical pseudodifferential operators $\Psi^{k, m, \widetilde{m}}$ 
and $\Psi_{\delta}^{k,m, \widetilde{m}}$, where we will usually add a subscript of $h$ or $\tilde{h}$ to indicate which 
parameter is used in the quantization.  The relationship between $\Psi_h$ and 
$\Psi_{\tilde{h}}$ is given in the following lemma.  
\begin{lemma}
\label{U-hsc-lemma}
Let $a \in \s_{0}^{k,m,\tilde{m}}$, and set
\be
b(X, \Xi) = a\left(\csh^{\half}X, \csh^{\half} \Xi
\right) \in \s_{-\half}^{k,m,\tilde{m}}.
\ee
There is a linear operator $T_{h, \tilde{h}}$, unitary on $L^2$,  
such that
\be
\Op_{\tilde{h}}^w(b) T_{h, \tilde{h}} u = T_{h, \tilde{h}} \Op_h^w(a) u.
\ee
\end{lemma}
\begin{proof}
For $u \in L^2(\reals^{n})$, define $T_{h, \tilde{h}}$ by
\begin{eqnarray}
\label{U-hsc}
T_{h, \tilde{h}} u (X) := \csh^{\frac{n}{4}}u\left( \csh^{\half} X \right).
\end{eqnarray}
We see immediately that $T_{h, \tilde{h}}$ conjugates operators $a^w(x, hD_x)$ and $b^w(X, \tilde{h} D_X)$.
\end{proof}
We have the following microlocal commutator lemma.
\begin{lemma}
\label{2-param-lemma}
Suppose $a \in \s_{0}^{-\infty,0,0}$, $b \in
\s_{-\half}^{-\infty,m,\widetilde{m}}$, and $\tilde{h}>h$.  

(a) If $A = \Op_{\tilde{h}}^w (a)$ and $B = \Op_{\tilde{h}}^w(b)$,
\be
[A,B] &=& h^{-m} \tilde{h}^{-\tilde{m}} \left(\frac{\tilde{h}}{i}\Op_{\tilde{h}}^w (\{a,b\}) + \O\left(
h^{3/2} \tilde{h}^{3/2} \right)\right).
\ee

(b) More generally, for each $l>1$,
\be
 \ad_A^l B =  h^{-m} \tilde{h}^{-\tilde{m}} \O_{L^2 \to L^2} \left( h \tilde{h}^{l-1} \right) .
\ee

(c) If $a \in \s_{0}^{-\infty,0,0}$, $b \in
\s_{\half}^{-\infty,m,\widetilde{m}}$, $\tilde{h}>h$, $A= \Op_h^w(a)$,
and $B= \Op_h^w(b)$, then
\be
[A,B] =  h^{-m}\tilde{h}^{-\tilde m} \left( \frac{h}{i}\Op_{h}^w (\{a,b\}) + \O\left(
h^{3/2} \tilde{h}^{3/2} \right) \right).
\ee
\end{lemma}
\begin{proof}
Without loss of generality, $m = \widetilde{m} = 0$, so for (a) we have from the Weyl
calculus:
\be
[A,B] = \frac{\tilde{h}}{i} \Op_{\tilde{h}}^w (\{a,b\}) + \tilde{h}^3 \O
\left( \sum_{|\alpha| = |\beta| = 3} \partial^\alpha a \partial^\beta b \right),
\ee
since the second order term vanishes in the Weyl expansion of the
commutator.  Note $\partial^\alpha a$
is bounded for all $\alpha$, and observe for $|\beta|=3$,
\be
\tilde{h}^3 \partial^\beta b & = & \tilde{h}^3 \O \left(h^{3/2}
\tilde{h}^{-3/2} \right).
\ee
For part (b) we again assume $m = \widetilde{m} = 0$, and we observe that for $l>1$ we no longer have the
same gain in powers of $h$ as in part (a).  This follows from the fact
that the $\tilde{h}$-principal symbol for the commutator $[A,[A,B]]$, $-i
\tilde{h} \{a, -i \tilde{h}\{a,b\} \}$, satisfies 
\ben
-i \tilde{h}\{a, -i \tilde{h}\{a,b\}\} & = & -\tilde{h}^2 \Big( \partial_\Xi a
\partial_X \left(\partial_\Xi a \partial_X b - \partial_X a
\partial_\Xi b \right) \label{comm-calc-1} \\
&& \quad - \partial_X a \partial_\Xi \left( \partial_\Xi a \partial_X b - \partial_X a
\partial_\Xi b \right) \Big) \nonumber \\
& \in & \s_{0}^{-\infty,-1 ,-1}, \label{comm-calc-2}
\een
since $\{a,b\}$ involves products of derivatives of both $a$ and $b$.

For general $l>1$, assume 
\be
\sigma_{\tilde{h}} \left( \ad_A^l B \right) \in \s_{0}^{0, -1, 1-l}
\ee
and a calculation similar to (\ref{comm-calc-1}-\ref{comm-calc-2})
finishes the induction.

Finally, part (c) follows from the proof of part (a) with $-1/2$
replaced by $1/2$ and the composition formula \eqref{a-pound-b}.
\end{proof}

\subsection{A Lemma of Bony-Chemin}
The following lemma (given more generally in \cite{BoCh}) will be used
in the proof of Theorem \ref{main-theorem-5}.  We include a sketch of the proof from
\cite[Proposition 3.7]{SjZw2} here for completeness.  It is easiest to phrase in terms of 
order functions.  
A smooth function $m \in \Ci(T^*X; \reals)$ is called an order function if it satisfies
\begin{eqnarray*}
m(x, \xi) \leq C m(y, \eta) \left\langle \dist (x-y) + |\xi - \eta| \right\rangle^N
\end{eqnarray*}
for some $N \in \mathbb{N}$.  We say $a \in \s^l(m)$ if 
\begin{eqnarray*}
\left| \partial^\alpha  a \right| \leq C_{\alpha }h^{-l} m.
\end{eqnarray*}
If $l = 0$, we write $\s(m):= \s^0(m)$.
\begin{lemma}
\label{etG-lemma}
Let $m$ be an order function, and suppose $G \in \Ci (T^*X ; \reals)$ satisfies
\begin{eqnarray}
\label{G-cond-1}
G(x, \xi) - \log \left( m(x, \xi) \right) = \O (1),
\end{eqnarray}
and 
\begin{eqnarray}
\label{G-cond-2}
\partial_x^\alpha \partial_\xi^\beta G(x,\xi) = \O(1) \,\,\, \text{for} \,\,\, (\alpha, \beta) \neq (0,0).  
\end{eqnarray}
Then for $G^w = \Op_h^w(G)$ and $|t|$ sufficiently small,
\begin{eqnarray*}
\exp (tG^w) = \Op_h^w(b_t)
\end{eqnarray*}
for $b_t \in \s(m^t)$.  Here $e^{tG^w}$ is defined as the unique
solution to the evolution equation 
\be
\left\{ \begin{array}{l} \partial_t \left( U(t) \right) - G^w
  U(t) = 0 \\
U(0) = \id. \end{array} \right.
\ee
\end{lemma}

\begin{proof}[Sketch of Proof]
The conditions on $G$ \eqref{G-cond-1} and \eqref{G-cond-2} are equivalent to saying $e^{tG} \in \s(m^t)$.  We will 
compare $\exp tG^w$ and $\Op_h^w (\exp t G)$.
\begin{claim}
\label{U-inv-claim}
Set $U(t):= \Op_h^w (e^{tG}): \s \to \s$.  For $|t| < \epsilon_0$, $U(t)$ is invertible and $U(t)^{-1} = \Op_h^w (b_t)$ for 
$b_t \in \s(m^{-t})$, where $\epsilon_0$ depends only on $G$.
\end{claim}
\begin{proof}[Proof of Claim]
Using the composition law, we see $U(-t) U(t) = \id + \Op_h^w(E_t)$, with $E_t = \O(t)$.  Hence $\id + \Op_h^w(E_t)$ is 
invertible and using Beals's Theorem, we get $(\id + \Op_h^w(E_t))^{-1} = \Op_h^w(c_t)$ for $c_t \in \s(1)$.  
Thus $\Op_h^w(c_t)U(-t)U(t) = \id$, so
\begin{eqnarray*}
U(t)^{-1} = \Op_h^w \left( c_t \# \exp (-tG) \right),
\end{eqnarray*}
and subsequently $b_t \in \s(m^{-t})$.
\end{proof}
Now observe that
\begin{eqnarray*}
\frac{d}{dt} U(-t) = - \Op_h^w \left( G \exp(-tG) \right),  \,\,\, \text{and} \,\,\, U(-t) G^w = \Op_h^w \left( e^{-tG} \# G \right),
\end{eqnarray*}
so that
\begin{eqnarray}
\lefteqn{\frac{d}{dt} \left( U(-t) e^{tG^w} \right) =} \label{Op(A_t)}
\\ & = & -\Op_h^w \left( G \exp(-tG) \right) e^{tG^w} + 
\Op_h^w \left( e^{-tG} \# G \right) e^{tG^w} \nonumber \\
& = & \Op_h^w (A_t) e^{tG^w}, \nonumber
\end{eqnarray}
for $A_t \in \s(m^{-t})$.  To see \eqref{Op(A_t)}, recall that by the composition law,
\begin{eqnarray*}
e^{-tG} \# G = e^{-tG} G + \left( \text{terms with }\,G \,\,\, \text{derivatives} \right).
\end{eqnarray*}
Then the first terms in \eqref{Op(A_t)} will cancel and the remaining terms will all involve at least one 
derivative of $G$, which is then bounded by \eqref{G-cond-2}.  

Set $C(t):= -\Op_h^w(A_t)U(-t)^{-1}$.  Claim \ref{U-inv-claim} implies $C(t) = \Op_h^w (c_t)$ for a family 
$c_t \in \s(1)$.  The composition law implies $c_t$ depends smoothly on $t$.  Then
\begin{eqnarray*}
\left( \frac{\partial}{\partial t} + C(t) \right) \left( U(-t) e^{tG^w} \right) = \Op_h^w (A_t) e^{tG^w} - 
\Op_h^w (A_t) e^{tG^w} = 0,
\end{eqnarray*}
so we have reduced the problem to proving the following claim.
\begin{claim}
Suppose $C(t) = \Op_h^w (c_t)$ with $c_t \in \s(1)$ depending smoothly on $t \in (-\epsilon_0, \epsilon_0)$.  If $Q(t)$ solves
\begin{eqnarray*}
\left\{ \begin{array}{c}
\left( \frac{\displaystyle \partial}{\displaystyle \partial t} + C(t) \right) Q(t) = 0, \\
Q(0) = \Op_h^w(q), \,\,\, \text{with} \,\,\, q \in \s(1),
\end{array} \right.
\end{eqnarray*}
then $Q(t) = \Op_h^w (q_t)$ with $q_t \in \s(1)$ depending smoothly on $t \in (-\epsilon_0, \epsilon_0)$.
\end{claim}
\begin{proof}[Proof of Claim]
The Picard existence theorem for ODEs implies $Q(t)$ exists and is bounded on $L^2$.  We want to use Beals's 
Theorem to show $Q(t)$ is actually a quantized family of symbols.  Let $l_1, \ldots, l_N$ be linear symbols.  
We will use induction to show that for any $N$ and any choice of the $l_j$, $\ad_{\Op_h^w (l_1)} \circ \cdots \circ 
\ad_{\Op_h^w(l_N)}Q(t) = \O(h^N)_{L^2\to L^2}$.  Since we are dealing with linear symbols, we take $h = 1$ for 
convenience.  First note
\begin{eqnarray*}
&&\frac{d}{dt} \ad_{\Op_h^w (l_1)} \circ \cdots \circ \ad_{\Op_h^w(l_N)} Q(t) +  \ad_{\Op_h^w (l_1)} \circ \cdots \circ 
\ad_{\Op_h^w(l_N)}\\
&& \quad \quad \quad \quad \quad \quad \quad \quad \quad \cdot  \left( C(t) Q(t) \right) = 0
\end{eqnarray*}
For the induction step, assume $ \ad_{\Op_h^w (l_1)} \circ \cdots \circ \ad_{\Op_h^w(l_k)}Q(t) = \O(1)$ is known 
for $k<N$ and observe
\begin{eqnarray*}
\lefteqn{ \ad_{\Op_h^w (l_1)} \circ \cdots \circ \ad_{\Op_h^w(l_N)} \left( C(t) Q(t) \right) = } \\
& = & C(t)  \ad_{\Op_h^w (l_1)} \circ \cdots \circ \ad_{\Op_h^w(l_N)} Q(t) + R(t),
\end{eqnarray*}
where $R(t)$ is a sum of terms of the form $A_k(t) \ad_{\Op_h^w (l_1)} \circ \cdots \circ \ad_{\Op_h^w(l_k)}Q(t)$ for 
each $k<N$ and $A_k(t) = \Op_h^w (a_k(t))$ with $a_k(t) \in \s(1)$.  Set $\tilde{Q}(t) =  \ad_{\Op_h^w (l_1)} \circ 
\cdots \circ \ad_{\Op_h^w(l_N)}Q(t)$, and note that $\tilde{Q}$ solves
\begin{eqnarray*}
\left\{ \begin{array}{c}
\left( \frac{\displaystyle \partial}{\displaystyle \partial t} + C(t) \right) \tilde{Q}(t) = -R(t), \\
\tilde{Q}(0) = \O (1)_{L^2 \to L^2}.
\end{array} \right.
\end{eqnarray*}
Since $R(t) = \O(1)_{L^2 \to L^2}$ by the induction hypothesis, Picard's theorem implies $\tilde{Q}(t):L^2 \to L^2$ 
as desired.
\end{proof}
\end{proof}



\section{$h$-Fourier Integral Operators}
\numberwithin{equation}{section}
\label{FIO}
In this section we review some facts about $h$-Fourier Integral
Operators ($h$-FIOs).  For this work, we are only interested in a
special class of $h$-FIOs, namely those associated to local symplectomorphisms.  
In order to motivate this, suppose $f:X \to Y$ is a diffeomorphism.  Then we write
\begin{eqnarray*}
f^*u(x) = u(f(x))= \frac{1}{(2 \pi h)^n} \int e^{i \langle f(x) -y, \xi \rangle /h} u(y) dy d\xi,
\end{eqnarray*}
and $f^*: \Ci(Y) \to \Ci(X)$ is an $h$-FIO associated to the nondegenerate phase function $\phi = \langle f(x) -y, 
\xi \rangle$.  We recall the standard notation: if $A:\Ci_c(Y) \to \mathcal{D}'(X)$ is a continuous mapping with
distributional kernel $K_A \in \mathcal{D}'(X \times Y)$, 
\begin{eqnarray*}
\WF'(A) & = & \{((x, \xi),(y, \eta)) \in (T^*X \times T^*Y) \setminus 0 :
\\
&& \quad \quad (x, y ; \xi, - \eta) \in \WF (K_A) \}.
\end{eqnarray*}
In this notation, we note 
\begin{eqnarray*}
\WF' f^* \subset  \left\{ ((x, \xi),(y, \eta)): y = f(x), \,\, \xi = \,^t D_x f \cdot \eta \right\},
\end{eqnarray*}
which is the graph of the induced symplectomorphism
\begin{eqnarray*}
\kappa (x, \xi) = (f(x), (\,^tD_x f)^{-1}(\xi)).
\end{eqnarray*}

Now let $A(t)$ be a smooth family of semiclassical pseudodifferential operators: 
$A(t) = \Op_h^w(a(t))$ with 
\begin{eqnarray*}
a(t) \in \Ci \left( [-1,1]_t; \s^{-\infty,0} \left( T^*X \right) \right),
\end{eqnarray*}
such that for each $t$, $\WF(A(t)) \Subset T^*X$.  Let $U(t): L^2(X) \to L^2(X)$ be defined by
\begin{eqnarray}
\left\{ \begin{array}{c}
hD_tU(t) + U(t)A(t) = 0, \\
U(0) = U_0 \in \Psi_h^{0,0}(X),
\end{array} \right. \label{U(t)}
\end{eqnarray}
where $D_t = -i \partial / \partial t$ as usual.  If we let $a_0(t)$ be the real-valued $h$-principal symbol of 
$A(t)$ and let $\kappa(t)$ be the family of symplectomorphisms defined by
\begin{eqnarray*}
\left\{ \begin{array}{c}
\frac{\displaystyle d}{\displaystyle dt} \kappa(t)(x, \xi) = \left( \kappa(t) \right)_* \left( H_{a_0(t)}(x, \xi) \right), \\
\kappa(0)(x, \xi) = (x, \xi),
\end{array} \right.
\end{eqnarray*}
for $(x, \xi) \in T^*X$, then $U(t)$ is a family of $h$-FIOs associated to $\kappa(t)$.  We have the following 
well-known theorem of Egorov.
\begin{old-thm}[Egorov's Theorem]
Suppose $B \in \Psi_h^{k,m}(X)$, and $U(t)$ defined as above.  Suppose further that $U_0$ in \eqref{U(t)} is 
elliptic ($\sigma_h(U_0) \geq c >0$).  Then there exists a smooth family of pseudodifferential operators $V(t)$ 
such that
\begin{eqnarray}
\left\{ \begin{array}{c}
\sigma_h \left( V(t) B U(t) \right) = \left( \kappa(t) \right)^* \sigma_h(B), \\
V(t)U(t) -I, \,\, U(t)V(t) -I \in \Psi_h^{-\infty, -\infty}(X). \label{egorov1}
\end{array} \right.
\end{eqnarray}
\end{old-thm}
\begin{proof} As $U_0$ is elliptic, there exists an approximate inverse $V_0$, such that $U_0 V_0 -I, \,\, V_0 
U_0 - I \in \Psi_h^{-\infty, -\infty}$.  Let $V(t)$ solve
\begin{eqnarray*}
\left\{ \begin{array}{c}
h D_t V(t) - A(t) V(t) = 0, \\
V(0) = V_0.
\end{array} \right.
\end{eqnarray*}
Write $B(t) = V(t)BU(t)$, so that
\begin{eqnarray*}
hD_t B(t) = A(t)V(t)BU(t) - V(t)BU(t)A(t) = [A(t), B(t)] 
\end{eqnarray*}
modulo $\Psi_h^{-\infty, -\infty}$.  But the principal symbol of $[A(t), B(t)]$ is 
\begin{eqnarray*}
\sigma_h \left( [A(t), B(t)] \right) = \frac{h}{i} \left\{ \sigma_h ( A(t)), \sigma_h (B(t)) \right\} =
 \frac{h}{i} H_{a_0(t)} \sigma_h(B(t)),
\end{eqnarray*}
so \eqref{egorov1} follows from the definition of $\kappa(t)$.
\end{proof}
Let $U:= U(1)$, and suppose the graph of $\kappa$ is denoted by $C$.  Then we introduce the standard notation
\begin{eqnarray*}
U \in I_h^0(X \times X; C'), \,\,\, \text{with}\,\,\, C' = \left\{ (x, \xi; y, -\eta) : (x, \xi) = 
\kappa(y, \eta) \right\},
\end{eqnarray*}
meaning $U$ is the $h$-FIO associated to the graph of $\kappa$.  The next few results when taken together 
will say that locally all $h$-FIOs associated to symplectic graphs are of the same form as $U(1)$.  First
 a well-known lemma.
\begin{lemma}
\label{deform-lemma}
Suppose $\kappa : \nbhd (0,0) \to \nbhd (0,0)$ is a symplectomorphism fixing $(0,0)$.  Then there exists 
a smooth family of symplectomorphisms $\kappa_t$ fixing $(0,0)$ such that $\kappa_0 = \id$ and $\kappa_1 
= \kappa$.  Further, there is a smooth family of functions $g_t$ such that 
\begin{eqnarray}
\label{kappa-deform-eq}
\frac{d}{dt} \kappa_t = (\kappa_t)_* H_{g_t}.
\end{eqnarray}
\end{lemma}
The proof of Lemma \ref{deform-lemma} is standard, but we include a
sketch here, as it will be used in the proof of Theorem
\ref{main-theorem-5}.
\begin{proof}[Sketch of Proof]
First suppose $K: \reals^{2n} \to \reals^{2n}$ is a linear symplectic
transformation.  Write the polar decomposition of $K$, $K = QP$ with
$Q$ orthogonal and $P$ positive definite.  It is standard that $K$
symplectic implies $Q$ and $P$ are both symplectic as well.  Identify
$\reals^{2n}$ with $\cx^n$ on which $Q$ is unitary.  Write $Q = \exp
-JB$ for $B$ Hermitian and $P = \exp A$ for $A$ real symmetric and $JA
+ AJ = 0$, where
\begin{eqnarray*}
J := \left( \begin{array}{cc} 0 & -I \\ I & 0 \end{array} \right)
\end{eqnarray*}
is the standard matrix of symplectic structure on $\reals^{2n}$.
  Then 
\be
K_t = \exp (-tJB) \exp(tA)
\ee
satisfies $K_0 = \id$ and
$K_1 = K$.

In the case $\kappa$ is nonlinear, set $K = \partial \kappa(0,0)$ and
choose $K_t$ such that $K_0 = \id$ and $K_\half = K$.  Then set 
\begin{eqnarray*}
\tilde{\kappa}_t(x, \xi) = \frac{1}{t} \kappa(t(x, \xi)),
\end{eqnarray*}
and note that $\tilde{\kappa}_t$ satisfies $\tilde{\kappa}_0 = K$,
$\tilde{\kappa}_1 = \kappa$.  Rescale $\tilde{\kappa}_t$ in $t$, so
that $\tilde{\kappa}_t \equiv K$ near $1/2$ and $\tilde{\kappa}_1 =
\kappa$.  Rescale $K_t$ so that $K_0 = \id$ and $K_t \equiv K$ near
$1/2$.  Then $\kappa_t$ is defined for $0 \leq t \leq 1$ by taking
$K_t$ for $0 \leq t \leq 1/2$ and $\tilde{\kappa}_t$ for $1/2 \leq t
\leq 1$.

To show $\frac{d}{dt} \kappa_t = (\kappa_t)_* H_{g_t}$, set $V_t =\frac{d}{dt}
\kappa_t$.  Cartan's formula then gives for $\omega$ the symplectic form
\begin{eqnarray*}
\mathcal{L}_{V_t} \omega = d\omega \contraction V_t + d( \omega \contraction
V_t),
\end{eqnarray*}
but $\mathcal{L}_{V_t} \omega = \frac{d}{dt}  \kappa_t^* \omega = 0$ since
$\kappa_t$ is symplectic for each $t$.  Hence $\omega \contraction V_t =
dg_t$ for some smooth function $g_t$ by the Poincar\'{e} lemma, in
other words, $V_t = (\kappa_t)_* H_{g_t}$.
\end{proof}

We have the following version of Egorov's theorem.
\begin{proposition}
\label{AF=FB}
Suppose $U$ is an open neighbourhood of $(0,0)$ and $\kappa: U \to \kappa( U)$ is a symplectomorphism fixing $(0,0)$.  
Then there is a unitary operator $F : L^2 \to L^2$ such that for all $A = \Op_h^w(a)$,
\begin{eqnarray*}
AF = FB \,\, \text{microlocally on}\,\, \kappa(U) \times U,
\end{eqnarray*}
where $B = \Op_h^w(b)$ for a Weyl symbol $b$ satisfying
\begin{eqnarray*}
b = \kappa^* a + \O(h^2).
\end{eqnarray*}
$F$ is microlocally invertible in $ U$ and $F^{-1} A F = B$ microlocally in $U \times U$.
\end{proposition}
Proposition \ref{AF=FB} is a standard result, however we include a
proof since we will use it in the sequel.
\begin{proof}
For $0 \leq t \leq 1$ let $\kappa_t$ be a smooth family of symplectomorphisms satisfying $\kappa_0 = \id$, 
$\kappa_1 = \kappa$, and let $g_t$ satisfy $\frac{d}{dt} \kappa_t =
(\kappa_t)_* H_{g_t}$.  Let $G_t = \Op_h^w(g_t)$, 
and solve the following equations 
\begin{eqnarray}
\label{F-eqn-1a}
&& \left\{ \begin{array}{c}
h D_tF(t) + F(t) G(t) = 0, \,\, (0 \leq t \leq 1) \\
F(0) = I,
\end{array} \right. \\
&& \left\{ \begin{array}{c}
hD_t \tilde{F}(t) - G(t) \tilde{F}(t) = 0,\,\, (0 \leq t \leq 1) \\
\tilde{F}(0) = I.
\end{array} \right. \nonumber
\end{eqnarray}
Then $F(t), \tilde{F}(t) = \O(1) : L^2 \to L^2$ and
\begin{eqnarray*}
hD_t\left( F(t) \tilde{F}(t) \right) = -F(t)G(t) \tilde{F}(t) + F(t) G(t) \tilde{F}(t) = 0,
\end{eqnarray*}
so $F(t) \tilde{F}(t) = I$ for $0 \leq t \leq 1$.  Similarly,
$E(t) = \tilde{F} F - I$ satisfies
\ben
\label{ode-eg}
hD_t E(t)  =  G(t) \tilde{F}(t) F(t) - \tilde{F}(t) F(t) G(t)  =  [G(t), E(t) ]
\een
with $E(0) = 0$.  But equation \eqref{ode-eg} has
unique solution $E(t) \equiv 0$ for the initial condition $E(0) = 0$.
Hence $\tilde{F}(t) F(t) =I$ microlocally. 

Now set $B(t) = \tilde{F}(t) A F(t)$.  We would like to show $B(t) = \Op_h^w(b_t)$, for $b_t = 
\kappa_t^*a + \O(h^2)$.  Set $\tilde{B}(t) = \Op_h^w (\kappa_t^*a)$.  Then
\begin{eqnarray*}
hD_t \tilde{B}(t) & = & \frac{h}{i} \Op_h^w \left( \frac{d}{dt}\kappa_t^*a \right) \\
& = & \frac{h}{i} \Op_h^w \left(\{ g_t, \kappa_t^*a \}\right) \\
& = & \left[ G(t), \tilde{B}(t) \right] + E_1(t),
\end{eqnarray*}
where $E_1(t) = \Op_h^w(e_1(t))$ for $e_1(t)$ a smooth family of symbols.  
Note if we take $g_t \# (\kappa_t^* a) - (\kappa_t^* a) \# g_t$, the composition formula \eqref{a-pound-b} 
implies the $h^2$ term vanishes for the Weyl calculus since $\omega^2$ is symmetric while 
\begin{eqnarray*}
g_t(x,\xi) \kappa_t^* a(y, \eta) - \kappa_t^* a(x, \xi) g_t(y, \eta)
\end{eqnarray*}
is antisymmetric.  Thus $E_1(t) \in \Psi_h^{0,-3}$, since we are working
microlocally.  We calculate
\begin{eqnarray}
\lefteqn{hD_t \left( F(t) \tilde{B}(t) \tilde{F}(t) \right) =} \label{B-de1}\\
& = & -F(t) G(t) \tilde{B}(t) \tilde{F}(t) + F(t) 
\left( \left[ G(t), \tilde{B}(t) \right] + E_1(t) \right) \tilde{F}(t) \label{B-de2}\\ 
&& \quad + F(t)\tilde{B}(t) G(t) \tilde{F}(t) \nonumber \\
& = & F(t) E_1(t) \tilde{F}(t) \label{B-de3}\\
& = & \O(h^3).\nonumber
\end{eqnarray}
Integrating in $t$ and dividing by $h$ we get
\begin{eqnarray}
\label{B-int}
F(t) \tilde{B}(t) \tilde{F}(t) = A +  \frac{i}{h} \int_0^t F(s)E_1(s)\tilde{F}(s) ds = A + \O(h^2),
\end{eqnarray}
so that $\tilde{B}(t) - B(t) = \O(h^2)$. 

We will construct families of pseudodifferential operators $B_k(t)$ so
that for each $m$
\ben
B(t) = \tilde{B}(t) + B_1(t) + \cdots + B_m(t) + \O(h^{m+2}). \label{B-ind}
\een
Let 
\be
\tilde{e}_1(t) = (\kappa_t)^* \int_0^t (\kappa_s^{-1})^* e_1(s) ds,
\ee
and set $\tilde{E}_1(t) = \Op_h^w( \tilde{e}_1(t))$.  Observe
\be
hD_t \tilde{E}_1 = \left[ G(t), \tilde{E}_1 \right] + \frac{h}{i}\left(
E_1(t) + E_2(t)\right) ,
\ee
where $E_2(t) \in \Psi_h^{0,-4}$ by the Weyl calculus, since
$[G,\tilde{E}_1] = \O(h^4)$.  Then as in (\ref{B-de1}-\ref{B-de3})
\be
hD_t \left( F(t)\tilde{E}_1(t) \tilde{F}(t) \right) & = & - F(t)
\left[ G(t), \tilde{E}_1(t) \right] \tilde{F}(t) + F(t) hD_t\left(
\tilde{E}_1(t) \right) \tilde{F}(t) \\
& = & \frac{h}{i}\left( F(t) E_1(t) \tilde{F}(t) + F(t) E_2(t)
\tilde{F}(t)\right) .
\ee
Integrating in $t$ gives
\be
F(t) \tilde{E}_1(t) \tilde{F}(t) = \int_0^t F(s)E_1(s) \tilde{F}(s) ds
+ \int_0^t F(s)E_2(s) \tilde{F}(s) ds,
\ee
and substituting in \eqref{B-int} gives
\be
\tilde{B}(t) - B(t) & = & \frac{i}{h} \tilde{E}_1(t) - \tilde{F}(t) \left(
\frac{i}{h} \int_0^t F(s)E_2(s)\tilde{F}(s)ds \right) F(t) \\
& = & \frac{i}{h} \tilde{E}_1(t) + \O(h^3).
\ee
Setting $B_1(t) = -i \tilde{E}_1(t)/h$ and continuing inductively gives
$B_k(t)$ satisfying \eqref{B-ind}. 

Let $l$ be a linear symbol, and $L = \Op_h^w(l)$.  Then
\begin{eqnarray*}
\ad_L (\tilde{B} - B) = \left[ \tilde{B} - B, L \right] = \O(h^2).
\end{eqnarray*}
Fix $N$.  From \eqref{B-ind} we can choose $B_1, \ldots , B_N$ so that
replacing $\tilde{B}$ with $\tilde{B}+ B_1 + \cdots + B_N$, we have for $l_1, \ldots, l_N$ linear symbols, $L_k = \Op_h^w(l_k)$,
\begin{eqnarray*}
\ad_{L_1} \circ \cdots \circ \ad_{L_N} (\tilde{B} - B) = \O(h^{N+2}),
\end{eqnarray*}
so Beals's Theorem implies $B(t) = \Op_h^w(b(t))$ for $b(t) = \kappa_t^* a + \O(h^2)$.
\end{proof}

The next proposition is essentially a converse to Proposition \ref{AF=FB}.
\begin{proposition}[\cite{EvZw}, Theorem 10.7]
\label{U-FIO}
Suppose $U = \O(1): L^2 \to L^2$ and for all pseudodifferential operators $A ,B \in \Psi_h^{0,0}(X)$ 
such that $\sigma_h(B) = \kappa^* \sigma_h(A)$, $AU = UB$ microlocally near $(\rho_0, \rho_0)$, 
where $\kappa : \nbhd (\rho_0) \to \nbhd (\rho_0 )$ is a symplectomorphism fixing 
$\rho_0$.  Then $U \in I_h^0(X \times X; C')$ microlocally near $(\rho_0, \rho_0)$.
\end{proposition}

\begin{proof}
Choose $\kappa_t$ a smooth family of symplectomorphisms such that $\kappa_0 = \id$, $\kappa_1 = \kappa$, 
and $\kappa_t( \rho_0) = \rho_0$.  Choose $a(t)$ a smooth family of functions satisfying $\frac{d}{dt} 
\kappa_t = (\kappa_t)_* H_{a(t)}$, and let $A(t) = \Op_h^w(a(t))$.  Let $U(t)$ be a solution to 
\begin{eqnarray*}
\left\{ \begin{array}{c} 
hD_t U(t) - U(t)A(t) = 0, \\
U(1) = U,
\end{array} \right.
\end{eqnarray*}
for $0 \leq t \leq 1$.  Next let $A$ and $B$ satisfy the assumptions of the proposition.  Since $AU = UB$, 
we can find $V(t)$ satisfying
\begin{eqnarray}
\label{V(t)-1}
\left\{ \begin{array}{c}
A U(t)V(t) = U(t)BV(t), \\
V(0) = \id.
\end{array} \right.
\end{eqnarray}
By Egorov's theorem, the right hand side of \eqref{V(t)-1} is equal to 
\begin{eqnarray*}
U(t) V(t) \left( V(t)^{-1} B V(t) \right) = U(t) V(t) A + \O(h).
\end{eqnarray*}
Setting $t = 0$, we see $[U(0), A ] = \O(h)$.  Applying the same argument to $[U(t), A]$ and another choice 
of $\tilde{A}, \tilde{B}$ satisfying the hypotheses of the proposition yields by induction,
\begin{eqnarray}
\label{beals-00}
\ad_{A_1} \circ \cdots \circ \ad_{A_N} U(0) = \O(h^N)
\end{eqnarray}
for any choice of $A_1, \ldots, A_N \in \Psi_h^{0,0}(X)$.  Since we are only interested in what $U(t)$ looks 
like microlocally, \eqref{beals-00} is sufficient to apply Beals's Theorem and conclude that $U(0) \in 
\Psi_h^{0,0}(X)$.  Thus $U(t)$ and hence $U(1) = U$ is in $I_h^0(X \times X; C')$ for the twisted graph
\begin{eqnarray*}
C' = \left\{ (x,\xi, y, -\eta): (y, \eta) = \kappa(x, \xi) \right\}.
\end{eqnarray*}
\end{proof}

The following Corollary says the particular choice of deformation to
identity does not change $F(1)$ at the expense of modifying the
initial condition.  The proof follows directly from the proof of
Proposition \ref{AF=FB} and Proposition \ref{U-FIO}.
\begin{corollary}
\label{F-cor}
Let $F \in I_h^0(X \times X; C')$ microlocally be a unitary $h$-FIO
associated to the graph of a local symplectomorphism $\kappa : U \to
\kappa(U)$, where $U$ is a neighbourhood of $(0,0)$ and $\kappa(0,0) =
(0,0)$.  Let $\kappa_t$ be a deformation to identity of $\kappa$
satisfying \eqref{kappa-deform-eq} for some Hamiltonian $g_t$.  Then
there exists $F_0 \in I_h(X \times X; C_1')$, for
\be
C_1' = \{ (x, \xi, x, - \xi ) \},
\ee
$F_0$ microlocally unitary, such that $F(t)$ satisfies \eqref{F-eqn-1a} with $F(0) =
F_0$ and $F(1) = F$ microlocally near $U$.
\end{corollary}

We will make use of the following well known proposition (see, for
example Theorem 10.17, \cite{EvZw}).
\begin{proposition}
\label{hDx-prop}
Let $P \in \Psi_h^{k,0}(X)$ be a semiclassical operator of real principal type ($p = \sigma_h(P)$ is real 
and independent of $h$), and assume $dp \neq 0 $ whenever $p=0$.  Then for any $\rho_0 \in \{ p^{-1}(0) \}$, 
there exists a symplectomorphism $\kappa : T^*X \to T^* \reals^n$ defined from a neighbourhood of $\rho_0$ 
to a neighbourhood of $(0,0)$ and an $h$-FIO $T$ associated to its
graph such that 

(i) $\kappa^* \xi_1 = p$, 

(ii) $TP = h D_{x_1} T$ microlocally near $(\rho_0; (0,0))$, 

(iii) $T^{-1}$ exists microlocally near $((0,0); \rho_0)$.
\end{proposition}

The next proposition is a refinement of Proposition \ref{AF=FB} in the
case of symbols with two parameters.

\begin{proposition}
\label{2-param-egorov}
Suppose $U$ is an open neighbourhood of $(0,0)$ and $\kappa: U \to \kappa( U)$
is a symplectomorphism fixing $(0,0)$ and 
\be
a \in \s_{\half}^{-\infty, 0,0}.
\ee
Then there is a unitary operator $F : L^2 \to L^2$ such that for $A = \Op_h^w(a)$,
\begin{eqnarray*}
AF = FB \,\, \text{microlocally on}\,\, \kappa (U) \times U,
\end{eqnarray*}
where $B = \Op_h^w(b)$ for a Weyl symbol $b \in \s_{\half}^{-\infty, 0,0} $ satisfying
\begin{eqnarray*}
b = \kappa^* a + \O(h^{1/2} \tilde{h}^{3/2}).
\end{eqnarray*}
$F$ is microlocally invertible in $ U$ and $F^{-1} A F = B$ microlocally in $U \times U$.
\end{proposition}
The proof of Proposition \ref{2-param-egorov} follows from the Proof
of Proposition \ref{AF=FB}, using Lemma \ref{2-param-lemma} to
estimate the commutators.


\section[Manifolds with Boundary]{Manifolds with Boundary and Propagation of Singularities}
\label{manifold-boundary-1}
\numberwithin{equation}{section}
In this section, $X$ is a smooth, compact, $n$-dimensional manifold
with boundary.  We assume $P \in \Diff_{h,db}^{2,0}$ is a second order differential
operator whose principal symbol $p$ is a quadratic form in $\xi$ and $\partial X$ is
noncharacteristic with respect to $p$.  We adopt a microlocal viewpoint in which $\partial X$
is identified locally with a noncharacteristic hypersurface $Y \subset \reals^n$.  Our local model for $X$ near $Y$ is
$X = {\reals^n}$ with $Y = \{x \in \reals^n : x_1 = 0\}$.  We study the boundary value problem 
\begin{eqnarray}
\label{BVP}
\left\{ \begin{array}{l} (P-z)u = f \,\, \text{in} \,\, X, \\
u = 0 \,\, \text{on} \,\, Y,
\end{array}
\right. 
\end{eqnarray}
in a neighbourhood of a closed bicharacteristic for the flow of $H_p$
reflecting transversally off $Y$, and for energies $z$ near $0$.  Our
final goal is to describe
propagation of singularities at the boundary.  First we will prove factorization lemmas and energy estimates near $Y$, and
then prove the main result of this section, which is that the
microlocal propagator of $P-z$ can be extended in a meaningful way
through the reflections at the boundary.  The Main Theorem has the
following analogue in the case $\gamma$ reflects transversally off
$\partial X$.

\begin{main-theorem'}
Suppose $P(h) \in \Diff_h^2(X)$ and $\partial X$ is noncharacteristic
with respect to the principal symbol of $P(h)$.  Assume $\gamma$ makes
only transversal reflections with $\partial X$.  Let $A \in \Psi_{h,db}^{0,0}(X)$ be a pseudodifferential operator whose
principal symbol is $1$ near $\gamma$ and $0$ away from $\gamma$.  There exist constants $h_0>0$ and $C>0$ such that 
\be
\| u \| \leq C \frac{ \sqrt{\log ( 1/h )}} h \| P ( h ) u \| + C
\sqrt{ \log (1/h ) } \| ( I - A ) u \|
\ee
uniformly in $0<h<h_0$, where the norms are $L^2$ norms on $X$.  In particular, if $u(h)$
satisfies 
\be
\left\{ \begin{array}{l} P(h) u(h) = \O(h^\infty); \\
\|u(h)\|_{L^2(X)} = 1, \end{array} \right.
\ee
\be
\left\| (I - A) u \right\|_{L^2(X)} \geq \frac{1}{C} \log
\left(\left(1/h\right)\right)^{-\half} \,, \ \ 0 < h < h_0 .
\ee
\end{main-theorem'}

\subsection{Normally Differential Operators}
In the case $X$ is a smooth manifold
with boundary, we define pseudodifferential operators which are
differential in the normal direction at the boundary
microlocally.  For a microlocal definition, it suffices to assume $X =
\{ x_1 \geq 0 \}$ and $\partial X = \{ x_1 = 0\} $.  Then the algebra
of pseudodifferential operators which are normally differential at the boundary
is defined by the following:
\be
\Psi_{h,db}^{k,m}(X, \Omega_X^\half) & = & \Big\{ A(x, hD_x) \in
\Psi_h^{k,m} : \\
&& \quad \quad A(x, hD_x) = \sum_{j=0}^k A_j(x, hD_{x'})(hD_{x_1})^j \Big\}.
\ee

Suppose $\phi \in \Ci(X, \Omega^\half_X)$, and $x_0 \in \partial X$.  Using local
coordinates at the boundary, we write $x_0 = (0, x_0') \in \{ x_1 \geq
0 \}$.  Then $\phi \in \Ci(X, \Omega^\half_X)$ means there is a smooth extension $\tilde
\phi$ to an open neighbourhood of $x_0 \in \reals^n$.  For a
distribution $u \in \mathcal{D}'(X, \Omega^\half_X)$, we extend the
notion of $\WF (u)$ to a neighbourhood of the boundary.  We say
$(x_0, \xi_0) = (0, x_0', \xi_0)$ is not in $\WF (u)$ if there is a product neighbourhood $(x_0, \xi_0) \in U \times V \subset \reals^{2n}$ and a 
normally differential operator $A \in \Psi^{0,0}_{h, db}(U, \Omega_{U}^\half)$ such that
$\sigma_h(A)(x_0, \xi_0) \neq 0$ and
\be
Au \in h^\infty \Ci ( (0, 1]_h ; \Ci ( U, \Omega_{U}^\half)).
\ee
Observe if $u$ is smooth,
\be
(\WF u)|_{\partial X} \subset \WF (u|_{\partial X}) \sqcup
(\supp ( u|_{\partial X}) \times N^*(\partial X)).
\ee
Similarly, using our identification of the
$h$-wavefront set of a pseudodifferential operator as the essential
support of its symbol, $A \in \Psi_{h,db}^{0,0}$ with $\sigma_h(A)
\neq 0$ at $(0,x_0', \xi_0)$ implies the $\xi_1$
direction is always contained in the $h$-wavefront set of $A$.  

We are
going to be interested in symbols which are compactly supported in
$T^*X$, so we will need a notion of microlocal equivalence near the
boundary which allows us to consider operators which are both
normally differential {\it and} compactly supported in phase space.
For this we return to our local coordinates at the boundary.  Let $x_0
\in \partial X$, $x_0 = (0, x_0')$, and let $U\times V$ be a product
neighbourhood of $(x_0, 0)$ in $\reals^{2n}$ such that $V$ is of the
form $V = [-\epsilon_0, \epsilon_0]_{\xi_1} \times V_{\xi'}$.  By using the rescaling
\be
(x_1, \xi_1) \mapsto ( x_1 / \lambda, \lambda \xi_1),
\ee
the ellipticity of $P$ outside a compact set implies $p \geq C^{-1}$
in 
\be
(\complement [- \epsilon_0, \epsilon_0]) \times V_{\xi'}.
\ee
Choose $\psi \in \Ci( \reals)$ satisfying
\ben
&& \psi(t) \equiv 1 \text{ for } |t| \leq
\epsilon_0, \label{psi-assump-100} \\
&& \psi(t) \equiv 0 \text{ for } |t| \geq 2 \epsilon_0. \label{psi-assump-101}
\een  
We say two semiclassically tempered operators $T$ and $T'$
are microlocally equivalent near $(U \times V )^2$ if for all $A, A'
\in \Psi_{h,db}^{0,0}$ satisfying
\be
\text{proj}_{(x, \xi')} ( \WF A ) \text{ is sufficiently close to } U
\times V_{\xi'}, 
\ee
and similarly for $A'$,
\be
\psi(P(h))A(T - T') \psi(P(h))A' = \O(h^\infty): \mathcal{D}'(X) \to
\Ci (X).
\ee
In particular, if $A \in \Psi_{h.db}^{0,0}$, we say $A$ is microlocally
equivalent to
\be
\psi(P(h)) A
\ee
and we will use this identification freely throughout.

\subsection{Propagation of Singularities}
This section is basically a semiclassical adaptation of some of the
propagation of singularities results at the boundary presented in
\cite[Chap. 23]{hormander3}.  According to \cite[App. C.5]{hormander3}, under the noncharacteristic
assumption we can find local symplectic coordinates near $Y$ so that $Y  = \{x_1= 0\}$ and (possibly after a sign change)
\begin{eqnarray}
\label{p-factorization}
p(x, \xi) = \xi_1^2 - r(x, \xi'), \,\,\, \xi' = (\xi_2, \ldots, \xi_n).
\end{eqnarray}
We define the hyperbolic set $H\subset T^*Y$:
\begin{eqnarray*}
H := \{(x', \xi'): r(0, x', \xi') >0 \},
\end{eqnarray*}
on which the characteristic equation has two roots $\{x_1 = 0, \xi_1 = \pm r(x, \xi')^\half\}$.  Thus the Hamiltonian vector field of $p$,
\begin{eqnarray*}
H_p = 2 \xi_1 \partial_{x_1} - \partial_{\xi'}r \partial_{x'} + \partial_{x}r \partial_\xi
\end{eqnarray*}
points from $Y$ into $\{x_1 >0\}$ or $\{x_1 <0\}$, respectively, depending on which root of $r$ we choose.
We call the corresponding bicharacteristic rays {\it outgoing} and
{\it incoming} and write $\gamma_+$ and $\gamma_-$ respectively (see
Figure \ref{fig:fig8}).  We have the following factorization of the
operator $P-z$ in our microlocal coordinates.
\begin{figure}
\centerline{
\input{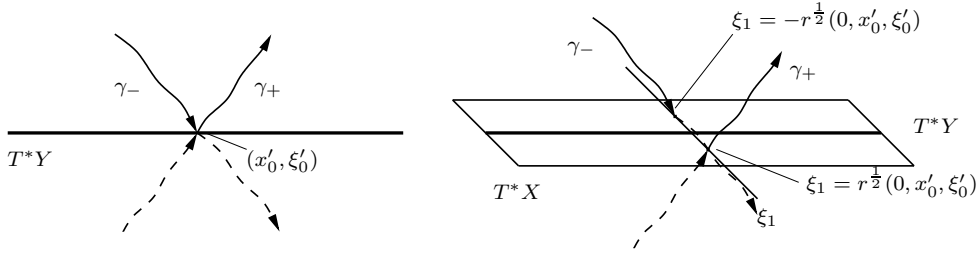}}
\caption{\label{fig:fig8} The incoming and outgoing bicharacteristics.
  Observe the {\it reflection} of $\gamma$ is continuous in a neighbourhood
  of $H$ in $T^*Y$ but not in $T^*X$, while the {\it transmission} of
  $\gamma$ is continuous in both.}
\end{figure}
\begin{lemma}
\label{p-fact-lemma}
There is a factorization of $P-z$ near $H$:
\begin{eqnarray*}
P-z = (hD_1 - A_-(x, hD'))(hD_1 - A_+(x, hD')) 
\end{eqnarray*}
with $A_\pm\in \Psi^{1,0}_{h,db}$ having principal symbol $\pm r^\half$.
\end{lemma}
\begin{remark}
We remark $r(x, \xi')$ and $A_\pm(x, hD')$ implicitly depend on the
energy $z$, although we don't explicitly note this dependence where no
ambiguity can arise.
\end{remark}
\begin{proof}
We follow the proof for the $h$ independent version of the lemma found in
\cite[Lemma 23.2.8]{hormander3}.  Using the coordinates above, the principal symbol of $P-z$ is \eqref{p-factorization}.  Set $A_\pm^1 = \Op( \pm r^\half )$ so that 
\begin{eqnarray*}
P-z - (hD_1 - A_-^1) (hD_1 - A_+^1) = R_1(x,hD) \,\,\, \text{microlocally}, 
\end{eqnarray*}
where $\sigma_h R_1 = \O(h)$ is independent of $\xi_1$.

Suppose now we have $A_\pm^j$ with principal symbols $\pm
r^\half$ such that
\begin{eqnarray*}
P -z- (hD_1 - A_-^j) (hD_1 - A_+^j) = R_j(x,hD) \,\,\, \text{microlocally}, 
\end{eqnarray*}
where $\sigma_h  R_j = \O(h^j)$ is independent of $\xi_1$.
Choose $a^j_-(x, \xi') = \O(h^{j})$ satisfying 
\begin{eqnarray*}
\sigma_h R_j(x, \xi') +2 a_-^j(x, \xi') r^\half(x, \xi') = 0,
\end{eqnarray*}
which we can do since $r^\half >0$ near $H$.  We will similarly add
$a_+^j(x, hD')$ to $A_+^j$, where $a_+^j = \O(h^j)$ is determined by the
following calculation:
\be
\lefteqn{P-z - (hD_1 - A_-^j - a_-^j(x, hD'))(hD_1 - A_+^j - a_+^j(x,
  hD')) =} \\
& = & R_j(x, hD) - a_-^j(x, hD')(hD_1 - A_+^j) - (hD_1 -
A_-^j)(a_+^j(x, hD')) \\
&& + a_-^j(x,hD') a_+^j(x,hD') \,\,\, \text{microlocally}.
\ee
On the level of principal symbol, this yields the requirement that
\be
\lefteqn{\sigma_h R_j(x, \xi') - a_-^j(x, \xi')( \xi_1 - r^\half(x,
  \xi')) - (\xi_1 + r^\half(x, \xi'))a_+^j(x, \xi') = } \\
& = & -a_-^j(x, \xi') ( \xi_1 + r^\half(x, \xi')) - (\xi_1 +
r^\half(x, \xi'))a_+^j(x, \xi') \\
& = & 0,
\ee
which gives $a_+^j(x, \xi') = -a_-^j(x, \xi')$.
By induction and Borel's Lemma the argument is complete.
\end{proof}

We have also a microlocal factorization $P-z = (hD_1 -
\tilde{A}_+)(hD_1 - \tilde{A}_-)$, where the principal symbols of
$\tilde{A}_\pm$ are $\pm r^\half$ as in the lemma.  Suppose the
$\gamma_\pm$ intersect $T^*Y$ at $(x_0', \xi_0')$.  On
$\gamma_-$ we have $\xi_1 = -r^\half$, so $(hD_1 - \tilde{A}_+)$ is
elliptic near $\gamma_-$.  Then to solve \eqref{BVP}, we need only solve $(hD_1 - \tilde{A}_-)u = (hD_1 - \tilde{A}_+)^{-1} f = \tilde{f}$.  
\begin{lemma}
\label{bdy-energy-est}
Suppose $u$ solves the following Cauchy problem in $\reals^n_+$:
\begin{eqnarray}
\label{BVP2}
\left\{ \begin{array}{l}
(hD_1 - \tilde{A}_-) u = \tilde{f}, \,\,\, x_1 >0\\
\left. u \right|_{x_1=0} = \phi(x').
\end{array} \right.
\end{eqnarray}
Then 
\begin{eqnarray}
\label{apriori-1}
\lefteqn{ \sup_{0 \leq y \leq T_0} \| u(y, \cdot ) \|_{L^2_{x'}(\reals^{n-1} \times \{x_1 = y\})} \leq } \\ 
& \leq & C \| \phi \|_{L_{x'}^2( \reals^{n-1} \times \{ x_1 = 0\})} + \frac{C_{T_0}}{h} \| \tilde{f} \|_{L^1([0, T_0], L^2(\reals^{n-1}))} . \nonumber
\end{eqnarray}
\end{lemma}
\begin{proof}
Consider
\begin{eqnarray*}
\frac{1}{2} \partial_{y} \| u(y, \cdot) \|_{L^2(\reals^{n-1} \times \{x_1=y\})}^2 & = & \langle \partial_{y} u, u \rangle_{x'} \\
& = & - \frac{ \Im }{h} \langle hD_1 u, u \rangle_{x'} \\
& \leq & \frac{1}{h} \|u(y, \cdot) \|_{L^2_{x'}} \| \tilde{f}(y, \cdot) \|_{L^2_{x'}} \\
& \leq & \frac{1}{4h^2} \|\tilde{f}(y, \cdot) \|_{L^2_{x'}}^2 + \|u(y, \cdot) \|_{L^2_{x'}}^2
\end{eqnarray*}
which by Gronwall's inequality gives the lemma.
\end{proof}
Recall the semiclassical Sobolev norms $\| \cdot \|_{H_h^k}$ are given
by 
\begin{eqnarray*}
\| u \|_{H_h^k(V)} = \left( \sum_{|\alpha| \leq k} \int_V |(hD)^\alpha u|^2 dx \right)^\half.
\end{eqnarray*}
We observe that since $P$ is elliptic outside a compact set and $(hD_1
- \tilde{A}_+)$ is elliptic, replacing $\tilde{f}$ with $f$ and conjugating $\tilde{A}_-$ above with the invertible
operators $(C+P)^{s/2}$ for sufficiently large $C>0$, we can estimate the $L^2$ norm of $v =
(C+P)^{s/2}u$, and we get the Sobolev estimate
\begin{eqnarray*}
\lefteqn{ \sup_{0 \leq y \leq T_0} \| u(y, \cdot ) \|_{(H_h^s)_{x'}(\reals^{n-1} \times \{x_1 = y\})} \leq } \\ 
& \leq & C \| \phi \|_{(H_h^s)_{x'}( \reals^{n-1} \times \{ x_1 = 0\})} + \frac{C_{T_0}}{h} \| f \|_{L^1([0, T_0], H_h^s(\reals^{n-1} ))} .
\end{eqnarray*}
We are interested in proving the existence of a microlocal solution
propagator, hence we assume the wavefront set of $f$ is contained in a
compact set $K$ in a neighbourhood of $\gamma_-$ near $Y$.  We assume as well
that $K$ is contained in a single coordinate chart $U$ on which the
assumptions of Proposition \ref{hDx-prop} hold.  Suppose $K \subset \{ T_1 < x_1 < T_2 \}$ and $U \subset \{ T_1' < x_1 < T_2'\}$.
\begin{proposition}
\label{ml-solution-1}
There are exactly two microlocal solutions $u_i$, $i = 1,2$ to
$(hD_1 - \tilde{A}_-)u= \tilde{f}$ microlocally near $\gamma_-$ satisfying
\begin{eqnarray}
\label{u_1}
u_1 & = & 0 \,\,\, \text{microlocally for} \,\,\, x_1 \leq T_1, \\
u_2 & = & 0 \,\,\, \text{microlocally for} \,\,\, x_1 \geq T_2. \label{u_2}
\end{eqnarray}
\end{proposition}
\begin{proof}
First we prove the proposition in a neighbourhood of $\WF \tilde{f}$.  Let
$\tilde{K}$ be the coordinate representation of $K$.  Apply
Proposition \ref{hDx-prop} to write $(hD_1 -\tilde{A}_-)$ as $hD_1$ in these coordinates.
We write in the $x$-projection of this coordinate patch,
\begin{eqnarray*}
u_1(x) & = & \frac{i}{h} \int_{-\infty}^{x_1} \tilde{f}(y, x') dy, \\
u_2(x) & = & - \frac{i}{h} \int_{x_1}^{+\infty} \tilde{f}(y, x') dy,
\end{eqnarray*}
which satisfies $(hD_1 - \tilde{A}_-)u=\tilde{f}$ with (\ref{u_1}-\ref{u_2}).  Back in the
original coordinates on our manifold, set $u_1 = 0$ for $x_1 \leq
T_1'$ and $u_2 = 0$ for $x_1 \geq T_2'$.  To continue, we will employ the energy estimates in Lemma
\ref{bdy-energy-est}.  Suppose $u$ is a solution to $(hD_1-\tilde{A}_-)u=\tilde{f}$.  Then
for $v \in \Ci_c ([0, T_0) \times \reals^{n-1} )$,
\begin{eqnarray*}
\int_0^{T_0} \left\langle u, (hD_1- \tilde{A}_-) v \right\rangle_{x'} dy = \int_0^{T_0} \langle \tilde{f}, v\rangle_{x'} dy + \frac{h}{i}\langle u(0, \cdot), v(0, \cdot) \rangle_{x'},
\end{eqnarray*}
since in the Weyl calculus real symbols are self adjoint.  But from
the proof of Lemma \ref{bdy-energy-est}, replacing $y$ with $T_0 - y$,
we have since $\lim_{y \to T_0} v(y,\cdot) = 0$, 
\begin{eqnarray*}
\sup_{0 \leq y \leq T_0} \|v(y, \cdot) \|_{L_{x'}^2} \leq \frac{C}{h} \int_0^{T_0} \|g\|_{L_{x'}^2} dy,
\end{eqnarray*}
with $g = (hD_1 - \tilde{A}_-) v$.  We then have
\begin{eqnarray*}
\left|  \int_0^{T_0} \langle \tilde{f}, v\rangle_{x'} dy + \frac{h}{i}\langle
  u(0, \cdot), v(0, \cdot) \rangle_{x'} \right| \leq \frac{C_{\tilde{f}, u_{0,
    \cdot}}}{h} \int_0^{T_0} \|g\|_{L_{x'}^2} dy.
\end{eqnarray*}
For $h >0$ we can extend to $g \in L^2$ the complex-conjugate linear form
\begin{eqnarray*}
g \mapsto  \int_0^{T_0} \langle \tilde{f}, v\rangle_{x'} dy + \frac{h}{i}\langle u(0, \cdot), v(0, \cdot) \rangle_{x'}
\end{eqnarray*}
by the Hahn-Banach Theorem.  Thus by the Riesz Representation Theorem,
for $\tilde{f} \in \Ci$ with sufficiently
small wavefront set, we can find $u \in \Ci ([0, T_0), L_{x'}^2)$
  satisfying $(hD_1 - \tilde{A}_-)u=\tilde{f}$. 

For the uniqueness given by the conditions (\ref{u_1}-\ref{u_2}), note that if $f= 0$ and $u(0, \cdot) = 0$ in \eqref{apriori-1}, $u$ is zero.  Replacing $x_1$ by $T_0- x_1$ we get the backwards uniqueness result.
\end{proof}
Since $u_1$ is supported in the forward direction along the
bicharacteristic $\gamma_-$, we refer to $u_1$ and $u_2$ as the {\it
  forward} and {\it backward} solutions respectively.  Let $u_- = u_1$
be the forward solution along the incoming bicharacteristic
$\gamma_-$.  So far we have proved the solution $u_-$ satisfies $(P-z)u_-
= f$ near $\gamma_-$, $u_- = 0$ microlocally for $x_1$ larger than the
support of $f$, and $u_-$ restricted to the boundary is controlled by
$h^{-1}$ in $L^2$ if the wavefront set of $f$ is sufficiently small.  

The same energy method techniques can be used to solve the problem
\begin{eqnarray}
\label{BVP3}
\left\{ \begin{array}{l}
(P-z)u_+ = 0, \,\,\, \text{in} \,\, X, \\
\left. u_+ \right|_{Y} = \left. u_- \right|_{Y}
\end{array} \right.
\end{eqnarray}
near $\gamma_+$ so that $u = u_- - u_+$ solves \eqref{BVP}.  
\begin{corollary}
If $f \in H_h^{\infty}$ has sufficiently small wavefront set and $u$
solves \eqref{BVP}, then 
\begin{eqnarray*}
u \in C^1 ([0, T_0], H_h^s( \reals^{n-1}))
\end{eqnarray*}
for every $s$.  In particular, $u(y, \cdot) \in \Ci(\reals^{n-1} \times \{x_1 = y\})$ for each fixed $y \in [0, T_0]$.
\end{corollary}
In order to describe propagation of singularities near the boundary,
we first need the following lemma.
\begin{lemma}
\label{comm-lemma}
Let $\gamma_+$ be an interval on the outgoing bicharacteristic with
one endpoint at $(0,x_0',r(0,x_0', \xi_0')^\half, \xi_0')$.  Then
there is a pseudodifferential operator $Q(x,hD') \in \Psi^{0,0}_{h, db}$ which satisfies 

(i) $\sigma_h(Q)=0$ microlocally outside a neighbourhood of 
\begin{eqnarray*}
\{(x, \xi') : (x, r(x, \xi')^\half, \xi') \in \gamma_+ \},
\end{eqnarray*}

(ii) $Q$ is noncharacteristic at $(x_0', \xi_0')$, and 

(iii) $[Q(x, hD'), hD_1 - A_+(x, hD')] = 0$ microlocally near $\gamma_+$.
\end{lemma}
\begin{proof}
The principal symbol of the commutator $[Q(x, hD'), hD_1 - A_+(x,
hD')]$ is 
\be
-ih\{\sigma_h Q(x, \xi'), \xi_1 - r^\half(x, \xi')\} =
ih(\partial_{x_1}- H_{r^\half}) \sigma_h Q.  
\ee
First we solve the Cauchy problem
\begin{eqnarray}
\label{cauchy-q}
\left\{ \begin{array}{l}
\left( \partial_{x_1} - H_{r^\half} \right) Q_0=0, \,\,\, (x, \xi) \in T^*X \\
Q_0 = q_0, \,\,\, x_1 = 0
\end{array} \right.
\end{eqnarray}
so that $Q_0$ is constant on orbits of the Hamiltonian system
\begin{eqnarray}
\label{ham-sys}
\left\{ \begin{array}{l}
\dot{x} = -\partial_{\xi'} b(x, \xi'); \\
\dot{\xi} =  \partial_{x'} b(x, \xi'),
\end{array} \right.
\end{eqnarray}
where $(\,\,\dot{ }\,\,) := \partial_{x_1} $ and $b(x, \xi') :=
r^\half(x, \xi')$.  Let $\chi_{x_1}(y, \eta)$ be a solution to
\eqref{ham-sys} for initial conditions close to $(x_0', \xi_0')$ valid for $0 \leq x_1 \leq T$, say.  If $T>0$ is sufficiently small, then $\chi_{x_1}$ is invertible and $Q_0(x, \xi') = q_0( \chi_{x_1}^{-1}(x', \xi'))$ is the solution to \eqref{cauchy-q}.  Now if we select $q_0$ compactly supported and $q_0 = 1$ in a neighbourhood of $(x_0',\xi_0')$, we satisfy conditions (i)-(ii) with 
\begin{eqnarray*}
[Q_0(x, hD'), hD_1 - A_+(x, hD')] = R_1(x, hD'), \,\,\, \sigma_h R_1 = \O(h^2),
\end{eqnarray*}
since $\xi_1$ only appears as a monomial of first order in the principal symbol.  Now suppose we have $Q(x, hD')$ satisfying (i)-(ii) and 
\begin{eqnarray*}
[Q(x, hD'), hD_1 - A_+(x, hD')] = R_j(x, hD'), \,\,\, \sigma_h R_j = \O(h^{j+1}).
\end{eqnarray*}
We solve the inhomogeneous Cauchy problem
\begin{eqnarray*}
\left\{ \begin{array}{l}
ih \left( \partial_{x_1} - H_{b} \right) Q_j(x, \xi') = - \sigma_h R_j(x, \xi'),\,\,\, x \in T^*X \\
Q_j = q_j, \,\,\, x_1 = 0
\end{array} \right.
\end{eqnarray*}
for $Q_j = \O(h^j)$ and $q_j = \O(h^j)$, which we do by setting
\begin{eqnarray*}
Q_j(x_1, \chi_{x_1}(y, \eta)) = q_j(y, \eta) +(ih)^{-1} \int_0^{x_1}
\sigma_h R_j(s , \chi_s (y, \eta))ds.
\end{eqnarray*}
Then $\tilde{Q} = Q + Q_j$ satisfies (i)-(ii) and 
\begin{eqnarray*}
[\tilde{Q}(x, hD'), hD_1 - A_+(x, hD')] = R_{j+1}(x, hD'), \,\,\,
\sigma_h R_{j+1} = \O(h^{j+2}).
\end{eqnarray*}
By induction, the argument is finished, by setting $q_j = 0$ for $j>0$.
\end{proof}
Now suppose $\phi \in \Ci_c(\reals^{n-1})$, where we identify
$Y$ with $\reals^{n-1}$ near $(x_0')$.  Suppose further that
$(x_0', \xi_0') \in T^*(\reals^{n-1})   \setminus \WF \phi$ and
$\gamma_+ \cap (x_0', \xi_0') \neq \emptyset$ as before.  Choose $Q$ as in Lemma \ref{comm-lemma} so that 
\begin{eqnarray*}
[Q(x, hD'), hD_1 - A_+(x, hD')] = 0
\end{eqnarray*}
microlocally and if $q$ is the principal symbol of $Q$, then $q(0,
x_0', \xi_0') \neq 0$, but $q(0, \cdot, \cdot)$ vanishes outside a
small neighbourhood of $(x_0', \xi_0')$.  Thus $Q(x, hD') \phi =
0$ microlocally.  Suppose $u_+$ solves \eqref{BVP3} with $\phi$
replacing $u_- |_{Y}$.  Then $Qu_+$ satisfies
\begin{eqnarray*}
\left\{ \begin{array}{l} \left(hD_1 - A_+(x, hD')\right) Q(x, hD') u_+
  = 0, \,\,\, x_1 >0\\
Q(x, hD') u_+ = 0, \,\,\, x_1 = 0
\end{array} \right.
\end{eqnarray*}
microlocally.  Hence by the energy estimate \eqref{apriori-1} $Q(x,
hD')u_+ = 0$ microlocally.  We conclude $\WF u_+ \subset \cchar q$.

We have proved the following proposition.  
\begin{proposition}
\label{WF-propagation}
With the notation as in the preceding paragraphs, $\WF u_+ \subset \chi_{x_1} ( \WF \phi )$ and $\WF u_- \subset \chi_{-x_1} ( \WF \phi )$.  Further, as the $x_1$ direction is reversible, if at some $0 < x_1 < T_0$, $(y', \eta') \in \WF u_+ (x_1, \cdot)$, then $\chi_{-x_1}(y', \eta') \in \WF \phi$ and $\chi_{-x_1-y} (y', \eta') \in \WF u_-(y, \cdot)$ for $0 <y <T_0$.
\end{proposition}

In the special case $Y = \partial X$ we have the following lemma to
connect the notions of $h$-wavefront set near $\gamma_+$ and $\gamma_-$.

\begin{lemma}
\label{t-def-lemma}
Let $(x_0', \xi_0') \in H \cap \gamma$ be the reflection point at the
boundary, let $\chi_{x_1}$ be a solution to
\eqref{ham-sys} as above, and let $\phi_t = \exp(tH_p)$.  Then there
is an odd diffeomorphism $t = t(x_1)$ and a function $\xi_1 =
\xi_1(x_1)$ such that $(x_1, \xi_1; \chi_{x_1})$ lies on $\gamma_+$ and
$(x_1, \xi_1; \chi_{-x_1})$ lies on $\gamma_-$ for
$x_1> 0$ sufficiently small.  That is, $\chi_{x_1}(x_0',
\xi_0')$ coincides with the $(x', \xi')$ components of 
\be
\phi_t(0, x_0', r^\half(0, x_0', \xi_0'), \xi_0'),
\ee
\end{lemma}
\begin{proof}
Write $b(x, \xi') = r^\half(x, \xi')$ as in the proof of Lemma
\ref{comm-lemma}, and note $\chi_{-x_1}$ is the solution to
\eqref{ham-sys} with $b$ replaced with $-b$.  $\phi_t$ satisfies the
following differential equation on $\gamma_+$:
\be
\left\{ \begin{array}{l}
\partial_t x_1 = 2b \\
\partial_t x' = -2bb_{\xi'} \\
\partial_t \xi_1 = 2bb_{x_1} \\
\partial_t \xi' = 2bb_{x'} ; \\
x_1(0) = 0 \\
x'(0) = x_0' \\
\xi_1(0) = b(0,x_0', \xi_0') \\
\xi'(0) = \xi_0'.
\end{array} \right.
\ee
Set $(y'(x_1), \eta'(x_1)) = \chi_{x_1}(x_0', \xi_0')$, 
\ben
\label{t-def}
t(x_1): = \int_0^{x_1} (2b(y,y'(y), \eta'(y)))^{-1} dy,
\een
and calculate
\be
\frac{\partial}{\partial x_1} x' &= & \frac{\partial}{\partial t} x'
\frac{\partial t}{\partial x_1} \\
& = & -b_{\xi'}(x_1, x', \xi') \\
& = & \frac{\partial}{\partial x_1} y'.
\ee
As $x'$ and $y'$ have the same initial conditions, we conclude they
are equal for sufficiently small $x_1$.  For negative $t$,
we define $t=t(x_1)$ in the incoming bicharacteristic to be the
negative of that on the outgoing bicharacteristic, and a similar proof applies to
$\chi_{-x_1}$.
\end{proof}

\begin{remark}
With the addition of Lemma \ref{t-def-lemma} we could write
Proposition \ref{WF-propagation} in an equivalent form using
$\exp(tH_p)$ in place of $\chi_{x_1}$.  The important thing is that
the wavefront set does not depend on $\xi_1$.
\end{remark}

\begin{figure}
\centerline{
\input{fig3} }
\caption{\label{WF-fig} Proposition \ref{WF-propagation}}.
\end{figure}

\subsection{Microlocal Propagator at the Boundary}
We now return to the special case where $Y = \partial X$.  In the next section we will construct the Quantum Monodromy operator using the
microlocal propagator.  Suppose first $X$ is a manifold without
boundary.  We define the forward and backward microlocal propagators of $P-z$, $I_\pm^z(t) =
\exp(\mp it(P-z)/h)$, by the following evolution equation:
\be
\left\{ \begin{array}{l} hD_t I_\pm^z(t) \pm (P-z) I_\pm^z(t) = 0 \\
I_\pm^z(0) = \id_{L^2 \to L^2}. \end{array} \right.
\ee
In the case of a manifold without boundary, this is a well-defined
semigroup satisfying $[I_\pm^z(t), P-z] = 0$ and 
\be
\WF I_\pm^z(t) u \subset \exp(\pm tH_p)(\WF u).
\ee
We will show for $P \in \Diff_h^{2,0}$ with homogeneous principal symbol on a manifold with
boundary, the microlocal propagator can be extended in a meaningful
fashion as a discontinuous $h$-FIO which 
still carries the commutator and wavefront set properties above.

Suppose $\gamma$ reflects off $\partial X$ at the points
\be
m_\pm:= (0, x_0', \pm r^{\half}(0, x_0', \xi_0'), \xi_0'),
\ee
with the incoming and outgoing rays, $\gamma_\mp$, intersecting $\partial X$ at
$m_\mp$ respectively.  Since $p$ is assumed smooth up to the boundary,
we may extend $p$ and $\gamma_-$ to a neighbourhood of $m_-$ in $\{
x_1 \leq 0\}$.  We will show that functions $v(x')$ defined on
$\partial X$ can be identified with the microlocal kernel of $P-z$ in
a neighbourhood of $m_-$.  We factorize $P-z$ as in Lemma
\ref{p-fact-lemma}, $P-z = (hD_1 - \tilde A_+)(hD_1 - \tilde A_-)$
microlocally near $m_-$.  Near $\gamma_-$ the operator $(hD_1 - \tilde A_+)$ is elliptic.  Thus we want to be able to solve 
\begin{eqnarray}
\label{BVP4}
\left\{ \begin{array}{l}
(hD_1 - \tilde A_-) u = 0 , \,\,\, x_1 >0, \\
u(0,x') = v(x'), \,\,\, x_1 = 0
\end{array}
\right.
\end{eqnarray}
for any boundary condition $v$, microlocally near $m_-$.  The proof of
the following standard Proposition can be found in \cite[Theorem 10.9]{EvZw}.

\begin{proposition}
\label{u-osc-int}
There is a microlocal solution to \eqref{BVP4} given by the oscillatory integral
\begin{eqnarray}
\label{u-ansatz}
u(x) = \frac{1}{(2 \pi h)^{n-1}} \int e^{i/h (\phi(x, \xi') - \langle y', \xi' \rangle)} b(x, \xi') v(y') d\xi' dy',
\end{eqnarray}
where $b(0,x', \xi') = 1$ microlocally near $(x_0', \xi_0')$ and $\phi$ solves the eikonal equation
\begin{eqnarray}
\label{eikonal}
\left\{ \begin{array}{l}
\partial_{x_1} \phi(x, \xi') - a(x, \partial_{x'} \phi(x, \xi')) = 0, \,\,\, x_1 >0, \\
\phi(0, x', \xi') = \langle x', \xi' \rangle, \,\,\, x_1 = 0,
\end{array} \right.
\end{eqnarray}
with $a = \sigma_h(\tilde A_-)$.  Further, $u(x)$ is unique microlocally.
\end{proposition}

\begin{proposition}
\label{ml-prop-thm}
Let $X$ be a manifold with boundary, $P \in \Diff_{h,db}^{2,0}$ be a
differential operator with homogeneous principal symbol $p$, and assume
$\partial X$ is noncharacteristic with respect to $p$.  Let $U_\pm \subset T^*X$ be a
neighbourhood of $m_\pm \in T^*X$, and assume $P-z$
and $p-z$ are factorized near the boundary as in Lemma
\ref{p-fact-lemma} and equation \eqref{p-factorization} respectively.

(i) For each $m_0 \in \gamma_- \cap U_-$ sufficiently close to $m_-$, and $z \in [-\epsilon_0, \epsilon_0]$ for
$\epsilon_0>0$ sufficiently small there exist $h$-FIOs, $I_\pm^z(t)$, defined microlocally
near 
\be
\exp(\pm tH_p)(m_0) \times m_0
\ee
satisfying
\ben
\label{I-ode}
\left\{ \begin{array}{l}
hD_{t} I_\pm^z(t) \pm (P-z)(t) I_\pm^z(t) = 0 \\
I_\pm^z(0) = \id_{L^2 \to L^2}, \end{array} \right.
\een
for $t \neq t_1$, where $m_- = \exp(t_1 H_p)(m_0)$.

(ii) We have $[(P-z)(t), I_\pm^z(t)]=0$ for all $t \neq t_1$ sufficiently small, and if $u(x) \in L^2$ is a microlocal
solution to 
\ben
\label{BVP4'}
\left\{ \begin{array}{l} (P-z)u = f \in L^2, \,\,\, x \in \mathring{X}, \\
u = 0, \,\,\, x \in \partial X \end{array} \right.
\een
near $m_0$, then $I_\pm^z(t) u(x)$ is a microlocal solution to
\eqref{BVP4'} near $\exp(\pm tH_p) (m_0)$.

(iii) If $\WF u \subset K$, where $K$ is a compact neighbourhood of a
point $m_0$, 
\be
\WF I_\pm^z(t)u \subset \exp(\pm tH_p)(K).
\ee
\end{proposition}

\begin{proof}
Fix $m_0$.  According to Proposition \ref{hDx-prop}, $P-z$ may be conjugated to
$hD_{x_1}$ in a neighbourhood of $m_0$.  Then we use the proof of
Proposition \ref{ml-solution-1} to find a solution $u_{-,1}$ to
$(P-z)u=f$ near $m_0$.
Use the microlocal forward propagator defined for a neighbourhood of
$\gamma_-$ extended to a neighbourhood of $m_-$ to define $u_{-,1}$
along $\gamma_-$.  That is, $I_+^z(t) u_{-,1}$ satisfies
\be
(P-z) I_+^z(t) u_{-,1} = f
\ee
microlocally near $\exp(t H_p)(m_0)$, $0 \leq t \leq t_1$.  Let $v_-(x') = I(t_1) u_{-,1}
|_{\partial X}$, and use Proposition \ref{u-osc-int} to find a
function $u_{-,2}$ satisfying
\be
\left\{ \begin{array}{l} (P-z)u_{-,2} = 0 \\ u_{-,2} |_{\partial X} =
  v_-(x') \end{array} \right.
\ee
microlocally near $m_-$.  Let
\be
u_- = u_{-,1} - I_+^z(-t_1) u_{-,2},
\ee
so that $I_+^z(t) u_-$ satisfies \eqref{BVP4'} microlocally near
$\exp( t H_p)(m_0)$, $0 \leq t < t_1$.

Fix $m_2 = \exp( t_2H_p)(m_0) \in \gamma_+$ sufficiently close to
$m_+$ that we can similarly construct $I_+^z(t_2-t ) u_+$ satisfying
\eqref{BVP4'} microlocally near $\exp (-t H_p)(m_2)$ for $0 \leq t
< t_2 - t_1$.  We extend $I_+^z(t)$ to be discontinuous at $t_1$, so
that if $u$ solves \eqref{BVP4'} microlocally near $m_0$, 
\be
I_+^z(t_1) u = u_- + u_+
\ee
with 
\be
\WF u_\pm \subset U_\pm \cap \{ x_1 \geq 0 \}.
\ee

We need to verify this extension of $I_+^z(t)$ satisfies (i), (ii),
and (iii).  For $0 \leq t < t_1$ this is clear because $I_+^z(t)$ is
the usual semigroup.  At $t_1$, we have
\be
(P-z) I_+^z(t_1) u & = & (P -z) (I_+^z(t_1)u_- + I_+^z(t_2 - t_1) u_+)
\\
& = & f_- + f_+,
\ee
with $f_\pm = f$ microlocally near $m_\pm$.  Thus (ii) and (iii) are
clear.

For (i), let $A =
\Op_h^w(a)$ be a symbol defined microlocally in a neighbourhood of
$\exp(tH_p)(m_0)$.  Assume $m_\pm \notin \WF A$, and let $B =
\Op_h^w(\exp(tH_p)^* a)$.  Then 
\be
AI_+^z(t) u = I_+^z(t) B u
\ee
microlocally, and by Proposition \ref{U-FIO}, $I^z(t)$ satisfies
\ref{I-ode}.  

The proof for $I_-^z(t)$ is similar.
\end{proof}

\begin{corollary}
\label{ml-prop-cor}
Let $X$, $P$, and $p$ be as in Proposition \ref{ml-prop-thm}.  Suppose $\gamma(t)$ is
a periodic orbit for $\exp(tH_p)$ of period $T$ which has a finite
number of transversal reflections off $\partial X$.  Then for any
$m \in \gamma(t)$, $m \cap \partial X = \emptyset$, there exist $h$-FIOs,
$I_\pm^z(t)$, defined microlocally near $\exp(tH_p)(m)\times m$ for $0 \leq |t| \leq
T$ satisfying

(i)
\be
\left\{ \begin{array}{l}
hD_t I_\pm^z(t) \pm (P-z) I_\pm^z(t) = 0 \\
I_\pm^z(0) = \id_{L^2 \to L^2} \end{array} \right.
\ee
for almost every $t$.

(ii) $[P-z, I_\pm^z(t)]=0$, and if $u(x) \in L^2$ satisfies $(P-z)u =f \in
L^2$ microlocally near m, 
then $I_\pm^z(t) u(x)$ satisfies
\be
(P-z) I_\pm^z(t) u(x) = f(x)
\ee
microlocally near $\exp(\pm tH_p) (m)$.

(iii) If $\WF u \subset K$, where $K$ is a compact neighbourhood of a
point $m$, 
\be
\WF I_\pm^z(t) u \subset \exp(\pm tH_p)(K).
\ee
\end{corollary}
\begin{proof}
This follows immediately from Proposition \ref{ml-prop-thm} and uniqueness
of solutions to ordinary differential equations.
\end{proof}


\section{Quantum Monodromy Construction}
\label{mono}
\numberwithin{equation}{section}

In this section, we construct the {\it Quantum Monodromy} operator
\be
M(z): L^2( \reals^{n-1}) \to L^2( \reals^{n-1}) 
\ee
and prove some
basic properties.  Here we follow \cite{SjZw} and the somewhat
simplified presentation in \cite{SjZw3}.  It is
classical (see, for example, \cite{AbMa}) that the assumtions on $p$
imply there exists $\epsilon_0 >0$ such that for $-2 \epsilon_0 \leq E
\leq 2 \epsilon_0$ there is a closed semi-hyperbolic orbit in the level set
$\{p = E\}$.  Let
$z \in [- \epsilon_0, \epsilon_0] \subset \reals$.  Then $p - z$ is
the principal symbol of $P-z$, and $p-z$ admits a closed semi-hyperbolic
orbit in the level set $\{ p - z = 0  \}$, say, $\gamma(z)$ of period
$T(z)$.  

We work microlocally in a neighbourhood of 
\be
\Gamma : = \bigcup_{- \epsilon_0 \leq z \leq \epsilon_0} \gamma(z) \subset T^*X.  
\ee

Fix $m_0(z) \in \gamma(z)$, $m_0(z) \cap \partial X = \emptyset$, depending smoothly on $z$, and set $m_1(z)
= \exp (\half T(z) H_p)(m_0(z))$.  By perturbing $m_0(z)$ and
shrinking $\epsilon_0>0$ if necessary, we may assume $m_1(z) \cap \partial X = \emptyset$ as well.  Assume we are working with a fixed
$z$.  Define
\be
\km &= &\{ u \in L^2( \neigh (m_0(z))) : (P-z) u = 0 \\
&& \quad \quad \quad \quad \text{ microlocally near } m_0(z) \},
\ee
where ``$\neigh(m_0(z))$'' refers to a small arbitrary fixed neighbourhood
of $m_0(z)$ which is allowed to change from line to line.  Similarly we have
\be
\kmm &=&  \{ u \in L^2( \neigh (m_1(z))) : (P-z) u = 0 \\
&& \quad \quad \quad \quad \text{ microlocally near } m_1(z) \}.
\ee

We define the forward and backward microlocal propagators $I_\pm^z$ as in
Corollary \ref{ml-prop-cor}.  Then
\be
I_+^z(t) : \km \to \ker_{\exp(tH_p)(m_0(z))}(P-z),
\ee
and since
\be
\exp(T(z)H_p)(m_0(z)) = m_0(z)
\ee
we define the {\it Absolute Quantum Monodromy} operator 
\be
\MM(z): \km \to \km
\ee
by 
\ben
\label{MM-def}
\MM(z):= I_+(T(z)).
\een

It is convenient to introduce an inner product structure on
$\km$ (see \cite{HeSj}).  For this, let $\chi \in
\Ci(T^*X)$ be a microlocal cutoff supported near $\gamma(z)$
satisfying the following properties (see Figure \ref{fig:fig9}):
\ben
&& \chi \equiv 1 \text{ on } \exp (t H_p)(m_0(z)) \text{ for } 0 \leq t
\leq \half T(0) \label{chi-1}\\
&& \chi \equiv 0 \text{ on } \exp (t H_p)(m_0(z)) \text{ for }
\half T(0) + \delta \leq t \leq T(0) - \delta, \,\, \delta > 0. \label{chi-2}
\een
Let $[P, \chi]_+$ denote the part of the commutator supported near
$m_0(z)$ where we use $\chi$ to denote both the function and the
quantization whenever unambiguous, and for $u, v \in \km$, define the
{\it Quantum Flux product} as  
\be
\lll u, v \rqf: = \lll \frac{i}{h} [P, \chi]_+ u, v \rrr_{L^2( \neigh (m_0(z)))}.
\ee

According to Proposition \ref{hDx-prop}, there is a neighbourhood of
$m_0(z)$ and an $h$-Fourier integral operator $F$ defined microlocally near $m_0(0)$ such that
$F(P-z)F^{-1} = hD_{x_1}$ on $L^2(\widetilde{V})$, where 
$\widetilde{V} \subset \reals^{n}$ is an open neighbourhood of $0\in \reals^n$.  Then $\km$ can be
identified with $L^2(V)$, where $V \subset \reals^{n-1}$ is an open
neighbourhood of $0 \in \reals^{n-1}$.  Let 
\be
K(z): L^2(V) {\longleftrightarrow} \km
\ee
be the identification, and define the adjoint $K(z)^*$ with respect to
the $L^2$ inner product on $\km$.
Note
\be
K(z)^* : \km \longleftrightarrow  L^2(V)
\ee
is an identification as well.  The following two lemmas are from
\cite{SjZw}.
\begin{lemma}
\label{k-id}
The operator  
\be
U:= K(z)^* \frac{i}{h} [P, \chi]_+ K(z) : L^2(V) \to L^2(V)
\ee
is positive definite.  Setting $\widetilde{K}(z) = K(z) U^\half$, we have 
\be
\widetilde{K}(z)^* \frac{i}{h} [P, \chi]_+ \widetilde{K}(z) = \id : L^2(V) \to L^2(V).
\ee
\end{lemma}
\begin{proof}
Using Proposition \ref{hDx-prop}, we write
\be
\lll K(z)^* \frac{i}{h} [P, \chi]_+ K(z)v,v \rrr_{L^2(V)} & =&  \lll \partial_{x_1}
\chi K(z) v, K(z)v \rrr_{L^2(\neigh (m_0(z)))} \\
&  \geq &C^{-1} \|v\|^2.
\ee
\end{proof}
\begin{remark}
In light of Lemma \ref{k-id}, we replace $K(z)$ with $\widetilde{K}(z)$ and
write 
\be
K(z)^{-1} = K(z)^* \frac{i}{h}[P, \chi]_+.
\ee
\end{remark}
\begin{lemma}
The Quantum Flux product $\lll \cdot, \cdot \rqf$ does not depend on
the choice of $\chi$ satisfying (\ref{chi-1}-\ref{chi-2}).  In
addition, $\MM(z)$ is unitary on $\km$ with respect to this product.
\end{lemma}
\begin{proof}
Suppose $u,v \in L^2(V)$ and suppose $\widetilde{\chi}$ is another function satisfying
(\ref{chi-1}-\ref{chi-2}) which agrees with $\chi$ near $m_1(z)$.
Then $[P, \widetilde{\chi} - \chi]_+ = [P, \widetilde{\chi} - \chi]$,
$(P-z)K(z)u = 0$, and $K(z)^*(P-z) = ((P-z)K(z))^*$ imply
\be
\lll \frac{i}{h} [P, \widetilde{\chi} - \chi]_+ K(z)u, K(z)v \rrr = \lll
\frac{i}{h} ( \widetilde{\chi} - \chi ) K(z)u, (P-z)K(z)v \rrr = 0.
\ee
To see $\MM(z)$ is unitary, observe for $\tilde{u} \in \km$,
\be
\lefteqn{\lll \frac{i}{h} [P, \chi]_+ I_+^z(T(z)) \tilde{u}, I_+^z(T(z))
\tilde{u} \rrr =} \\
& = & \lll \frac{i}{h} \left[ P,  I_-^z(T(z)) \chi I_+^z(T(z)) \right]_+
\tilde{u} , \tilde{u}
\rrr \\
& = & \lll \frac{i}{h} \left[ P, \widetilde{\chi} \right]_+ \tilde{u},
\tilde{u} \rrr,
\ee
where $\widetilde{\chi} = \exp(TH_p)^*\chi$ satisfies
(\ref{chi-1}-\ref{chi-2}).
\end{proof}

Next we restrict our attention to $L^2(V)$ by defining the {\it Quantum
Monodromy operator} $M(z): L^2(V) \to L^2(V)$ by 
\be
M(z) = K(z)^{-1} \MM(z) K(z).
\ee
\begin{lemma}
$M(z): L^2(V) \to L^2(V)$ is unitary, and $M(z)$ is the quantization
  of the Poincar{\'e} map $S$.
\end{lemma}
\begin{proof}
Let $u \in L^2(V)$.  We calculate:
\be
\lefteqn{ \lll M(z)u, M(z)u \rrr_{L^2(V)}=} \\
 &=& \lll K(z)^{-1} \MM(z) K(z)u, K(z)^{-1}\MM(z)K(z)u
\rrr_{L^2(V)} \\
& = & \lll (K(z)^*)^{-1}K(z)^{-1} \MM(z) K(z) u, \MM(z) K(z) u \rrr_{L^2( \neigh(m_0(z)))} \\
& = & \lll \frac{i}{h} [P, \chi]_+ \MM(z) K(z) u, \MM(z) K(z) u
\rrr_{L^2( \neigh(m_0(z)))} \\
& = & \lll K(z) u, K(z) u \rrr_{L^2( \neigh(m_0(z)))} \\
& = &\lll u, u \rrr_{L^2(V)}.
\ee
In order to prove $M(z)$ is the quantization of the Poincar{\'e} map, we
will use Proposition \ref{U-FIO}.  We need to prove for
pseudodifferential operators $A,B \in \psi_h^{0,0}(V)$ such that
$\sigma_h(B) = S^* \sigma_h(A)$, we have $AM(z) = M(z)B$.  Without loss of
generality, we write $x' = (x_2, \ldots, x_n) \in V$ for the
variables in $V$ and $x = (x_1, x') \in \neigh ( \gamma(z))$ for the
variables in $X$ near $\gamma$.  Then for $v \in L^2(V)
\cap \Ci (V)$
\be
\lefteqn{ M(z) B(x', hD_{x'} )v(x')  = } \\
& = &  K(z)^{-1} \MM(z) K(z) B(x', hD_{x'}) v(x') \\
& = & K(z)^{-1} I_+^z(T(z)) B(x', hD_{x'}) I_-^z(T(z))  I_+^z(T(z)) 
K(z) v(x') \\
& = & K(z)^{-1} \Op \left( \left( \exp (TH_p) \right)^* \sigma_h(B)
\right)(x, hD_{x}) I_+^z(T(z)) K(z) v(x') \\
& = & A(x', hD_{x'}) M(z) v(x').
\ee

\end{proof}


\section{The Grushin Problem}
\label{grushin}

\numberwithin{equation}{section}

\subsection{Motivation of the Grushin Problem}
In this section we follow \cite{SjZw} and show how the Quantum
Monodromy operator arises naturally in the context of a {\it Grushin
  Problem} near $\gamma$.  This is a generalization of the linear
algebra Grushin problem: Suppose 
\be
A & : & H \to H, \\
B&:& H_- \to H, \\
C & : & H \to H_+,\,\, \text{and} \\
D &: & H_- \to H_+
\ee
are matrices acting on finite dimensional Hilbert spaces $H$, $H_-$, and $H_+$, and
\be
\left( \begin{array}{cc}
 \alpha & \beta \\ \sigma & \delta \end{array} \right) = \left(
\begin{array}{cc} A & B \\ C & D \end{array} \right)^{-1},
\ee
where
\be
\alpha & : & H \to H, \\
\beta & : & H_+ \to H, \\
\sigma & : & H \to H_-, \,\, \text{and} \\
\delta & : & H_+ \to H_-.
\ee
Then $A$ is invertible if and only if $\delta$ is invertible, in which
case
\be
A^{-1} = \alpha - \beta \delta^{-1} \sigma.
\ee

It appears counterintuitive at first that understanding the
invertibility of a larger matrix might be easier than understanding
the invertibility of a submatrix.  However, when the entries are
operators instead of matrices, the situation may change.  In the next
section we will see that introducing a matrix of operators will allow
us to understand microlocal invertibility of $P-z$ near a periodic
orbit by constructing a parametrix in the {\it
  double cover} of a neighbourhood of the orbit.  

\subsection{The Grushin Problem Reduction}
Throughout this section, we suppress the dependence on $z$ whenever unambiguous 
for ease in exposition.  We will build operators $R_+= R_+(z): \DD'(X) \to \DD'(V)$ and $R_-= R_-(z): \DD'(V)
\to \DD'(X)$ such that 
\be
\PP:= \left( \begin{array}{cc} \frac{i}{h} (P-z) & R_- \\ R_+ & 0
\end{array} \right) : \DD'(X) \times \DD'(V) \to \DD'(X) \times
\DD'(V)
\ee
has microlocal inverse 
\ben
\label{EE}
\EE = \left( \begin{array}{cc} E & E_+ \\ E_- & E_{-+} \end{array}
\right)
\een
near $\gamma \times (0,0)$, where $E$, $E_+$, and $E_-$ will be defined later, and 
\be
E_{-+} = I - M(z).
\ee
The following construction of the solution to the Grushin problem is
from \cite{SjZw}, with the addition here that we allow $\gamma(z)$ to
reflect transversally off the boundary of $\partial X$.  Recall $\chi
\in \Ci_c( T^*X)$ satisfies (\ref{chi-1}-\ref{chi-2}), and begin by setting 
\be
R_+ = K^* \frac{i}{h} \left[ P, \chi \right]_+.
\ee
Then if $u$ satisfies $(P-z)u = 0$ microlocally near $m_0(z)$, $R_+ u$ is
the microlocal Cauchy data.  That is, for $v \in L^2(V)$, $u = Kv$ is
a solution to the microlocal Cauchy problem
\ben
\label{Cauchy-1}
\left\{ \begin{array}{c}
(P-z)u = 0, \\ R_+ u = v \end{array} \right. 
\een
near $\gamma \times (0,0)$.  To construct a global solution, let
$K_f(t):= I_+(t) K$ and $K_b(t):= I_-(t) K$ be the forward and
backward (respectively) Cauchy problem solution operators.  Note for
$t \sim T/2$ we have
\ben
K_f(t) & = & I_+(t) K  \nonumber \\
& = & I_-(t) K K^{-1} \MM(z) K \nonumber \\
& = & K_b(t) M(z), \label{Kb-M}
\een
so microlocally near $m_1 \times (0,0)$ we have $K_f = K_bM(z)$.  Now
for $\Omega$ a neighbourhood of $\gamma$, we can solve
\eqref{Cauchy-1} in $\Omega \setminus \neigh(m_1)$.  To do this, set
\ben
\label{Eplus-def}
E_+ v = \chi K_fv + (1 - \chi) K_b v,
\een
so $E_+ v$ satisfies
\be
i) && E_+ v = K v \,\, \text{in a neighbourhood of }m_0(z) \\
ii) && R_+ E_+ = \id \,\, \text{microlocally near }(0,0) \times (0,0)
\in (T^*V)^2.
\ee
With $[\cdot, \cdot ]_-$ denoting the part of the commutator supported
near $m_1(z)$, we calculate:
\be
(P-z)E_+ v =  [P, \chi]_- K_f v - [P, \chi ]_- K_b v, 
\ee
since $K_f = K_b$ microlocally near $ m_0(z) \times (0,0)$ and $(P-z)
I_\pm(t)K v = 0$ microlocally near $ \exp
(tH_p)(m_0(z)) \times (0,0)$.  According to \eqref{Kb-M}, we can then write
\be
(P-z)E_+ v = [P, \chi]_- K_b (M(z) - \id) v.
\ee
For $v \in L^2(V)$, we set
\be
u& = & E_+ v, \\
u_-& = & E_{-+} v, \,\, \text{and} \\
R_- & = & \frac{i}{h} [P, \chi]_- K_b.
\ee
We have solved the following problem microlocally in $(\Omega
\setminus \neigh(m_1(z)))^2$ (see Figure \ref{fig:fig10}):
\ben
\label{g-2'}
\left\{ \begin{array}{r}
\frac{i}{h} (P-z) u + R_- u_- = 0 \\ R_+ u = v \end{array} \right. .
\een
\begin{figure}
\centerline{
\input{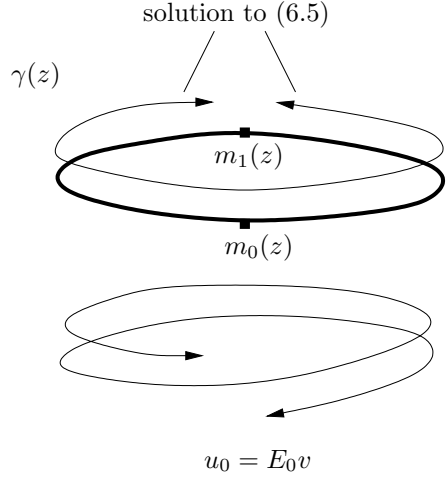}}
\caption{\label{fig:fig10} Microlocal solution to \eqref{g-2'} and
  construction of global solution to \eqref{global-sol}.}
\end{figure}

Thus if $\PP^{-1}$ exists, it is necessarily given by \eqref{EE},
where $E$ and $E_-$ have yet to be defined.  

For $\epsilon>0$ let 
\be
\left( \Omega \times_\epsilon \Omega \right)_\pm := \left\{ \left( \bigcup_{m \in \Omega} (
\exp \pm t H_p ) m, m \right) \bigcap \Omega \times \Omega : - \epsilon <
t < T - 2 \epsilon \right\}.
\ee
We will define $L_f$ and $L_b$, the forward and backward fundamental
solutions (respectively) of $i(P-z)/h$, which will be defined
microlocally on $\left( \Omega \times_\epsilon \Omega \right)_\pm$ respectively.
By Proposition \ref{hDx-prop}, we can conjugate $i (P-z)/h$ to
$\partial_{x_1}$ microlocally near the point $m_0(z)^2 \in (T^*X)^2$.
Then the local fundamental solutions $L_f^0$ and $L_b^0$ are given by 
\be
L_f^0 v(x)& = &\int_{-\infty}^{x_1} v(y, x') dy, \text{ and}\\
L_b^0 v(x) & = & - \int_{x_1}^{\infty} v(y,x') dy,
\ee
while $L_f$ and $L_b$ are now given microlocally near $ \exp (\pm t
H_p) m_0(z) \times m_0(z)$, respectively, by 
\be
L_f & = & I_+^z(t) L_f^0 \text{ and} \\
L_b & = & I_-^z(t) L_b^0.
\ee
It is convenient to introduce two new microlocal cutoffs $\chi_f$ and
$\chi_b$ satisfying (\ref{chi-1}-\ref{chi-2}) and in addition, 
\be
&& \chi \equiv 1 \text{ on } \supp \chi_f \cap W_+,  \\
&& \chi_b \equiv 1 \text{ on } \supp \chi \cap W_+,
\ee
where $W_+$ is a neighbourhood of $m_0(z)$ containing the support of
$[P, \chi]_+$ (see Figure \ref{fig:fig9}).  
\begin{figure}
\centerline{
\input{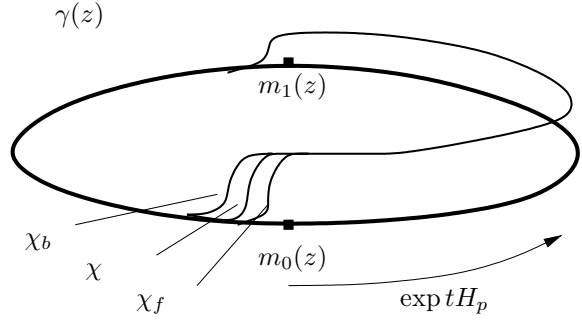}}
\caption{\label{fig:fig9} The cutoffs $\chi_b$, $\chi$, and $\chi_f$.}
\end{figure}
For $v \in L^2(\Omega)$, set
\be
\tilde{u} = L_f (I - \chi) v,
\ee
and observe $(P-z) \tilde{u} = 0$ past the support of $(I - \chi)$ in the
direction of the $H_p$ flow.  In particular, $(P-z) \tilde{u} = 0$ on
$\supp \chi_f$.  Then past $\supp(I - \chi)$ in the direction of the
$H_p$ flow, 
\be
\tilde{u} & =&  K K^* \frac{i}{h} \left[ P, \chi_f \right]_+ \tilde{u}
\\
& = & KK^* \frac{i}{h} \left[ P, \chi_f \right]_+ L_f ( I - \chi) v.
\ee
Let $\widetilde{I}_+(t)$ be the extension of $I_+(t)$ to $T \leq t
\leq 2T - \epsilon$, and let $\widetilde{K}_f = \widetilde{I}_+ K$.
Let $\widetilde{\Omega}$ denote the double covering space of
$\Omega$.  Then in $\widetilde{\Omega}$,
\be
\tilde{u} = \widetilde{K}_f K^* \frac{i}{h} \left[ P, \chi_f
  \right]_+ L_f (I - \chi) v.
\ee
We define $\hat{u} = L_b \chi v$ and $\widetilde{K}_b =
\widetilde{I}_- K$ similar to $\widetilde{K}_f$ so that 
\be
\hat{u} = \widetilde{K}_b K^* \frac{i}{h} \left[ P, \chi_b \right]_+
L_b \chi v.
\ee
We think of $\tilde{u}$ and $\hat{u}$ as being double-valued on $\Omega$
and write $L_{ff}v$ and $L_{bb}v$ to denote the second branches
respectively in a neighbourhood of $m_1(z)$.  Let $W_-$ denote a
neighbourhood of $m_1(z)$, and define (see Figure \ref{fig:fig10})
\be
u_0 = E_0 v := \left\{ \begin{array}{l}
L_b \chi v + L_f (I - \chi) v \,\,\, \text{outside } W_-, \\
L_b \chi v + (I - \chi) L_{bb} \chi v + L_f (I - \chi) v + \chi L_{ff}
(I - \chi) v \,\,\, \text{in } W_- \end{array} \right. .
\ee
Now we apply $i (P-z)/h$ to $E_0 v$ in $W_-$:
\be
\frac{i}{h} (P-z) E_0 v & = & v - \frac{i}{h} \left[ P, \chi \right]_-
L_{bb} \chi v + \frac{i}{h} \left[ P,\chi \right]_- L_{ff} (I - \chi)
v \\
& = & v - \frac{i}{h} \left[ P, \chi \right]_- K_b K^* \frac{i}{h}
\left[ P, \chi \right]_+ L_b \chi v \\
&& \quad \quad \quad + \frac{i}{h} \left[ P ,\chi \right]_- K_f K^*
\frac{i}{h} \left[ P, \chi \right]_+ L_f ( I - \chi ) v \\
& = & v - \frac{i}{h} \left[ P, \chi \right]_- K_b \Big( K^*
\frac{i}{h} \left[ P, \chi \right]_+ L_b \chi v \\
&& \quad \quad \quad - M(z) K^* \frac{i}{h} \left[ P, \chi \right]_+
L_f (I - \chi) v \Big),
\ee
where we have used $K_f = K_b M(z)$ in $W_-$ and dropped the tilde and
hat notation when thinking of second branches.  We have solved the
following problem:
\ben
\label{global-sol}
\frac{i}{h} (P-z) E_0 v + R_- E_{0,-} v = v,
\een
with 
\be
R_- = \frac{i}{h} \left[ P, \chi \right]_- K_b
\ee
as above, and
\be
E_{0,-}v := K^* \frac{i}{h} \left[ P, \chi \right]_+ L_b \chi v - M(z) K^*
\frac{i}{h} \left[ P, \chi \right]_+ (I - \chi) v.
\ee
Recalling the structure of $\EE$ and $\PP$, we calculate
\be
\PP \EE = \left( \begin{array}{ll} \frac{i}{h} (P-z)E + R_- E_- &
  \frac{i}{h}(P-z)E_+ + R_- E_{- +} \\ R_+ E & R_+ E_+ \end{array}
\right),
\ee
so that if $\EE$ is to be a microlocal right inverse of $\PP$ near $\gamma
\times (0,0)$, we require
\ben
 \frac{i}{h} (P-z)E + R_- E_- & = & \id : L^2( \Omega) \to L^2(\Omega), \label{g-1}\\
\frac{i}{h}(P-z)E_+ + R_- E_{- +} & = & 0 : L^2(V) \to L^2(\Omega), \label{g-2} \\
R_+ E & = & 0 : L^2(\Omega) \to L^2(V) \text{ and}, \label{g-3}\\
R_+ E_+ & = & \id : L^2(V) \to L^2(V) \label{g-4}
\een
microlocally.  Note \eqref{g-2} and \eqref{g-4} are satisfied according to
\eqref{g-2'}.  Owing to \eqref{g-4}, if we write $E = (I - E_+ R_+)
\tilde{E}$ for some $\tilde{E}$, then 
\be
R_+ E = R_+(I - E_+ R_+)\tilde{E} = (I - R_+ E_+) R_+\tilde{E} = 0,
\ee
and comparing with \eqref{g-1} we see $\tilde{E} = E_0$, 
\be
E = E_0 + E_+ R_+ E_0,
\ee
and 
\be
E_- = E_{0,-} - E_{- +} R_+ E_0.
\ee
Thus $\EE$ is a right inverse.  To see it is also a left inverse,
observe 
\be
R_+^* &= &\frac{i}{h} \left[ P, \chi \right]_+ K, \text{ and} \\
R_-^* & = & K_b^* \frac{i}{h} \left[ P, \chi \right]_-,
\ee
together with
\be
K_b^* \frac{i}{h} \left[ P, \chi \right]_- K_b = - \id 
\ee
implies
\be
K_b^* \frac{i}{h} \left[ P, (I - \chi) \right]_- K_b = \id.
\ee
In other words, after exchanging $\chi$ with $1 - \chi$, $W_+$ with
$W_-$, and $K$ with $K_b$, $R_+^*$ has the same form as $R_-$ and
$R_-^*$ has the same form as $R_+$.  Thus
\be
\PP^* = \left( \begin{array}{cc}\frac{i}{h} (P-z) & R_+^* \\ R_-^* & 0
\end{array} \right)
\ee
has the same form as $\PP$ and hence has a microlocal right inverse,
say 
\be
\mathcal{F}^* := \left( \begin{array}{cc} F & F_+ \\ F_- & F_{-+}
\end{array} \right)^*
\ee
Then $\PP^* \mathcal{F}^* = \id$ implies $\mathcal{F} \PP = \id$, so
\be
\mathcal{F} = \mathcal{F} \PP \EE = \EE
\ee
implies $\mathcal{F} = \EE$.


As every operator used in the preceding construction depends holomorphically 
on $z \in [-\epsilon_0, \epsilon_0] + i ( -c_0 h, c_0 h)$, we have
proved 
the following Proposition, which is from \cite{SjZw}:
\begin{proposition}
\label{SjZw-prop}
With $\PP$ and $\EE$ as above and $z \in [-\epsilon_0, \epsilon_0] + i
( -c_0 h, c_0 h)$, 
$\EE$ is a microlocal inverse for 
\be
\PP: L^2(\Omega) \times L^2(V) \to  L^2(\Omega) \times L^2(V)
\ee
 near $\gamma \subset T^*X$, and in addition, 
\be
\left\| \EE \right\|_{ L^2(\Omega) \times L^2(V) \to  L^2(\Omega) \times L^2(V) } \leq C.
\ee
\end{proposition}

\subsection{Comparing $P-z$ to $M(z)$}
As a consequence of Proposition \ref{SjZw-prop} and motivated by the
linear algebra Grushin problem, the following two theorems show
quantitatively that $P-z$ is invertible if and only if $I-M(z)$ is
invertible. 
\begin{theorem}
\label{main-theorem-4}
Let $M(z): L^2(V) \to L^2(V)$ be the Quantum Monodromy operator, $(P-z)$
and $R_+$ as above.  
Suppose $A \in \Psi_h^{0,0}(T^*X)$ is a microlocal cutoff with
wavefront set 
sufficiently close to $\gamma \subset T^*X$ and $B \in \Psi_h^{0,0}
(T^*V)$ is 
a microlocal cutoff with wavefront set sufficiently close to $(0,0)
\in T^*V$.  
Then there exists $\epsilon_0 >0$, $c_0>0$, and $h_0 >0$ such that,
with 
$z \in [ - \epsilon_0, \epsilon_0] + i(-c_0h, c_0h)$ and $0 < h < h_0$,
\ben
\lefteqn{ \|(P-z)u \|_{L^2(X)} \geq } \nonumber \\  && \geq C^{-1}h \left( \left\| B ( I - M(z)) R_+ u 
\right\|_{L^2(V)} - \left\| (I - A) u \right\|_{L^2(X)} \right) \label{inj-1a} \\
&& \quad \quad - \O( h^\infty) \| u \|_{L^2(X)}. \nonumber
\een
Further,
\ben
\lefteqn{  \| A u \|_{L^2(X)} \leq } \nonumber \\ && \leq C \left(  \left\| R_+ u \right\|_{L^2(V)} + 
h^{-1} \left\| (P-z) u \right\|_{L^2(X)} +  \left\| (I-A) u
\right\|_{L^2(X)}\right) \label{inj-2a} \\
&& \quad \quad +\O( h^\infty) \| u \|_{L^2(X)}. \nonumber
\een
\end{theorem}
\begin{proof}
That $\EE$ is a microlocal left inverse for $\PP$ means in particular 
that for $A$ and $B$ as in the statement of the theorem,
\ben
\label{inj-1}
\frac{i}{h} E_- (P-z) + E_{-+} R_+  = l_+ + \O(h^\infty)_{L^2(X) \to
  L^2(V)},
\een
where 
\be
B l_+ A = \O(h^\infty)_{L^2(X) \to L^2(V)}.
\ee
Since \eqref{inj-1a} is only concerned with injectivity, we note that 
by replacing $E_-$ and $E_{-+}$ with $\tilde{E}_- = BE_-$ and 
$\tilde{E}_{-+} = BE_{-+}$ respectively in \eqref{inj-1} doesn't 
change the fact that $\EE$ is a microlocal {\it left} inverse.  Thus
\ben
\label{tilde-Es}
\frac{i}{h} \tilde{E}_- (P-z) + \tilde{E}_{-+} R_+  = \tilde{l}_+ +
\O(h^\infty)_{L^2(X) \to L^2(V)},
\een
with
\be
\tilde{l}_+ A := B l_+ A =  \O(h^\infty)_{L^2(X) \to L^2(V)},
\ee
and for $u \in L^2(X)$,
\be
\tilde{E}_- (P-z) u + \frac{h}{i} \tilde{E}_{-+} R_+ u = \frac{h}{i} 
\tilde{l}_+ (I - A) u + \O(h^\infty)\| u \|_{L^2(X)},
\ee
hence \eqref{inj-1a}.

For \eqref{inj-2a}, we note $\EE \PP = \id$ microlocally gives also 
\ben
\label{inj-2}
\frac{i}{h} E(P-z) + E_+ R_+  = \id_{L^2(X)) 
\to L^2(X)} + l,
\een
where
\be
AlA = \O(h^\infty)_{L^2(X) \to L^2(X)}.
\ee
Similar to \eqref{tilde-Es}, we replace $E$ and $E_+$ with $\tilde{E} = AE$ and
$\tilde{E}_+ = AE_+$ without changing that $\EE$ is a microlocal
left inverse of $\PP$, and from \eqref{inj-2}, we get for $u \in L^2(X)$
\be
\frac{i}{h} \tilde{E} (P-z) u + \tilde{E}_+ R_+ u = A u + \tilde{l} u
+ \O(h^\infty)\| u \|_{L^2(X)},
\ee
Using $\tilde{l} A u := A l A u = \O(h^\infty) u$, we get
\be
C \left( \|( P-z) u \|_{L^2(X)} + h \| R_+ u \|_{L^2(V)} \right) \geq
h \| A u \|_{L^2(X)} - h \| (I - A) u \|_{L^2(X)},
\ee
which is \eqref{inj-2a}.
\end{proof}

Using that $\EE$ is a microlocal right inverse for $\PP$ we obtain the
following theorem, which completes the correspondence between $I -
M(z)$ and $P-z$.

\begin{theorem}
\label{ml-equiv-thm}
Suppose $A \in \Psi^{0,0}(X)$, $B \in \Psi^{0,0}(V)$ satisfy
\be
&& A \equiv 1 \text{ microlocally near } \gamma, \\
&& A \equiv 0 \text{ microlocally away from } \gamma, \\
&& B \equiv 1 \text{ microlocally near } (0,0), \\
&& B \equiv 0 \text{ microlocally away from } (0,0).
\ee
Suppose $u \in L^2(X)$ satisfies
\be
Au = u + \O(h^\infty) \| u \|_{L^2(X)},
\ee
and $v \in L^2(V)$ satisfies
\be
Bv = v + \O(h^\infty) \| v \|_{L^2(V) }.
\ee
Then we have
\ben
\label{ml-equiv-1}
A \frac{i}{h} (P-z) E u + A R_- E_- u & = & u + \O(h^\infty) \| u
\|_{L^2(X)}, \, \text{ and} \\
\label{ml-equiv-2}
A \frac{i}{h} (P-z) E_+ v + A R_- (I - M(z)) v & = & \O(h^\infty) \| v
\|_{L^2(V)}.
\een
\end{theorem}

\begin{remark}
The utility of \eqref{ml-equiv-2} is that in \S \ref{quasi},
where $\gamma$ will be assumed elliptic instead of semi-hyperbolic, we construct $v \in
L^2(V)$ concentrated near $(0,0)$ satisfying
\be
(I - M(z)) v = \O(h^N), \,\,\,\, \forall N
\ee
then $u := E_+ v$ satisfies
\be
(P-z) u = \O(h^{N+1}) \| u \|_{L^2(X)}
\ee
microlocally near $\gamma$.  This provides essentially a converse
to our Main Theorem.
\end{remark}

\begin{proof}[Proof of Theorem \ref{ml-equiv-thm}]
From Proposition \ref{SjZw-prop}, if we multiply $\mathcal{P}$ by
$\mathcal{E}$ on the right, we get
\ben
 \frac{i}{h} (P-z)  E + R_- E_- & = & \id_{L^2(X) \to L^2(X)} + r, \label{ml-equiv-1a}\\ 
\frac{i}{h} (P-z) E_+ + R_- (I-M(z)) & = & r_-, \label{ml-equiv-1b}\\
 R_+ E & = & r_+, \nonumber \\
 R_+ E_+ & = & \id_{L^2(V) \to L^2(V)} + r_{-+}, \nonumber
\een
where
\be
 ArA & = & \O(h^\infty)_{L^2(X) \to L^2(X)}, \\ 
 A r_- B & = & \O(h^\infty)_{L^2(V) \to L^2(X)}, \\
B r_+ A & = & \O(h^\infty)_{L^2(X) \to ^2(V)}, \text{ and}\\
 B r_{-+} B & = & \O(h^\infty)_{L^2(V) \to L^2(V)}.
\ee
Hence (\ref{ml-equiv-1a}-\ref{ml-equiv-1b}) imply for any $u \in
L^2(X)$, $v \in L^2(V)$,
\be
 A \frac{i}{h} (P-z) E u + A R_- E_- u = A u + A r (I-A) u + \O(h^\infty) \| u
\|_{L^2(X)} 
\ee
and
\be
A \frac{i}{h} (P-z) E_+ v + A R_- (I - M(z)) v = A r_- (I-B) v + \O(h^\infty) \| v
\|_{L^2(V)},
\ee
which is (\ref{ml-equiv-1}-\ref{ml-equiv-2}).
\end{proof}


\section{The Model Case}
\label{model-chapter}
\numberwithin{equation}{section}
In this section we indicate how Theorem \ref{main-theorem-4} can be
used to estimate $P-z$ in the model case.  Let $\dim X = 2$, and assume $t$ parametrizes
$\gamma = \gamma(0)$, and $\tau$ is the dual variable to $t$.  Then
our model for $p$ near $\gamma$ is the symbol 
\be
p = \tau + \lambda x \xi,
\ee
with $\lambda >0$.  We have 
\be
H_p = \partial_t + \lambda ( x \partial_x - \xi \partial_\xi ),
\ee
and the Poincar{\'e} map $S: \reals^2 \to \reals^2$ is given by
\be
S = \left( \begin{array}{cc} e^\lambda & 0 \\ 0 & e^{- \lambda}
\end{array} \right).
\ee
We want a deformation of the identity into $S$, that is a smooth
family of symplectomorphisms $\kappa_t$ such that $\kappa_0 = \id$ and
$\kappa_1 = S$.  This is clear in the model case:
\be
\kappa_t = \exp t\left( \begin{array}{cc} \lambda & 0 \\ 0 & - \lambda
\end{array} \right).
\ee
According to Lemma \ref{deform-lemma}, we can find a time-dependent
effective 
Hamiltonian $q_t= q_t(x, \xi)$ such that
\be
\frac{d}{dt} \kappa_t = (\kappa_t)_* H_{q_t}.
\ee
In the model case, this is again clear: $q_t = \lambda x \xi$,
independent of $t$.  

We know in general if $M(z)$ is the Quantum Monodromy operator it is an
$h$-FIO associated to the graph of $S$, which means our model is $M(z) = M^z(1)$ for
$M^z(t)$ a family of $h$-FIOs satisfying
\be
\left\{ \begin{array}{l} hD_t M^z(t) + Q(t) M^z(t) = 0, \\ M^z(0) = \id
\end{array} \right.
\ee
where $Q(t) = \Op (q_t)$ for the effective Hamiltonian $q_t$ as above.  In the
model case, $q$ does not depend on $t$ or $z$, so with $Q = \Op(q)$, $M^z(t)$ is just the
semigroup
\be
M(t) = \exp \left( -\frac{i}{h}tQ \right).
\ee

The basic idea is $M(t)$ is unitary, but $e^{-G^w}M(t) e^{G^w}$ is not
for $G$ with real principal symbol (if it exists).  Further, in the
model case, if $G$ is independent of $t$, 
\be
e^{-G^w}M(t) e^{G^w} = \exp \left( -\frac{i}{h} t e^{-G^w} Q e^{G^w}
\right),
\ee
and it will suffice to show $e^{-G^w} Q e^{G^w}$ has an imaginary part
of fixed size comparable to $h$.

We conjugate $M(t)$ to $M_1(t) = T_{h, \tilde{h}}M(t) T_{h,
  \tilde{h}}^{-1}$ where $T_{h, \tilde{h}}$ is the operator on
$L^2(\reals)$ as in \eqref{U-hsc}, and observe $M_1(t)$ satisfies the
evolution equation
\ben
hD_t M_1(t) & = & - T_{h, \tilde{h}} Q M(t) T_{h, \tilde{h}}^{-1}
\label{M-ode-1} \\
& = & - T_{h, \tilde{h}} Q T_{h, \tilde{h}}^{-1} T_{h, \tilde{h}} M(t) T_{h,
  \tilde{h}}^{-1} \label{M-ode-2} \\
& = & - Q_1 M_1(t), \label{M-ode-3}
\een
where 
\be
Q_1 = T_{h, \tilde{h}} Q T_{h, \tilde{h}}^{-1} \in \Psi_{-\half}^{-\infty,0,0}
\ee
microlocally.  We write $q_1(X, \Xi) = \sigma_{\tilde{h}}( Q_1)$, where
\be
q_1(X, \Xi) = \lambda \csh X \Xi + \O(h^2 + \tilde{h}^2),
\ee
as in Lemma \ref{U-hsc-lemma}.

Now we define the {\it escape function} 
\be
G(X, \Xi) = \half \log \left( \frac{1 + X^2}{1 + \Xi^2} \right),
\ee
and according to Lemma \ref{etG-lemma}, we can form the family of
operators 
\be
e^{s G^w},
\ee
where $G^w$ is the Weyl quantization of $G$ in the $\tilde{h}$ calculus
and $|s|$ is sufficiently small.  Let 
\be
\widetilde{M}(t) = e^{-sG^w} M_1(t) e^{sG^w},
\ee
whence
\be
hD_t \widetilde{M}(t) = - \widetilde{Q} \widetilde{M}(t)
\ee
for
\be
\widetilde{Q} = e^{-sG^w} Q_1 e^{sG^w}
\ee
by a similar argument to (\ref{M-ode-1}-\ref{M-ode-3}).  We write
\be
\widetilde{Q} = \exp( -s \ad_{G^w}) Q_1,
\ee
with
\be
\ad_{G^w}^k Q_1 = \O_{L^2 \to L^2}\left(h \tilde{h}^{k-1} \right),
\ee
and
\be
\left[ Q_1, G^w \right] = -i \tilde{h} \Op_{\tilde{h}}^w \left( H_{q_1} G
\right) + \O(h^{3/2}\tilde{h}^{3/2}).
\ee
We have
\be
H_{q_1} G & = & \lambda \csh\left( \frac{X^2}{1+X^2} + \frac{\Xi^2}{1+ \Xi^2}
\right) \\
& =: & \lambda \csh A,
\ee
so that 
\be
\widetilde{Q} = Q_1 - ish\Op_{\tilde{h}}^w \left( A\right) + s E_1^w + s^2 E_2^w,
\ee
with $E_1=\O(h^{3/2} \tilde{h}^{3/2})$ and $E_2 = \O(h \tilde{h})$.
Since $A$ is roughly the harmonic oscillator (see Lemma \ref{harm-osc}), 
\be
\langle \Op_{\tilde{h}}^w \left( A\right)U, U \rangle \geq
\frac{\tilde{h}}{C} \|U \|^2
\ee
independently of $h$, so that
\ben
\label{Q-tilde-est}
\Im \lll \widetilde{Q} U, U \rrr \leq - \frac{h \tilde{h}}{C} \| U \|^2.
\een
Thus with $\tilde{h}>0$ small but fixed,
\be
\widetilde{M}(1) = \exp \left( - \frac{i}{h} (\Re \widetilde{Q} + i
\Im \widetilde{Q}) \right)
\ee
and by \eqref{Q-tilde-est},
\ben
\label{M-tilde-est}
\left\| \widetilde{M}(1) \right\|_{L^2(\reals) \to L^2(\reals)} \leq r
<1.
\een
For $u \in L^2(\reals)$ and $U = T_{h, \tilde{h}} u$, we have by
\eqref{Q-tilde-est} and \eqref{M-tilde-est}
\be
\Re \lll (I - \widetilde{M}(1) ) U, U \rrr \geq C^{-1}\|U\|^2
\ee
for some $0<C<\infty$.  Define the operator $K^w$ by 
\ben
\label{Kw-def}
e^{sK^w} = T_{h, \tilde{h}}^{-1} e^{sG^w} T_{h, \tilde{h}}.
\een
We have shown that 
\be
\Re \lll e^{-sK^w} (I - M) e^{sK^w} u, u \rrr \geq C^{-1} \|u\|^2.
\ee
Since $\|\exp(\pm sK^w)\|_{L^2 \to L^2} = \O(h^{-N})$ for some $N$, we
have 
\be
\Re \lll (I - M) u, u \rrr \geq Ch^N \|u\|^2.
\ee


\section{The Linearization}
\label{linearization}
\numberwithin{equation}{section}
\subsection{Symplectic Linear Algebra and Matrix Logarithms}
In this section, we will show how to reduce the case of a general
Poincar{\'e} map with a fixed point to studying the quadratic Birkhoff
normal forms.  We assume as in the introduction that the
eigenvalues of modulus one obey the nonresonance assumption \eqref{non-res}.

We begin by tackling the problem
of negative real eigenvalues and eigenvalues of modulus $1$ of the linearized Poincar{\'e} map.  Let
$S: W_1 \to W_2$, $W_1, W_2 \subset \reals^{2n-2}$, be a local
symplectic map, $S(0,0) = (0,0)$, which we have identified with its coordinate
representation.  As in the proof of Lemma \ref{deform-lemma}, we
consider the polar decomposition of $dS(0,0)$:
\be
dS(0,0) = \exp (-J F) \exp (B),
\ee
with $F$ and $B$ real valued and $\exp(B)$ positive definite and
symplectic.  Specifically, $\exp(-JF)$ describes the action due to the
eigenvalues of modulus $1$ as well as the rotation inherent in the
negative real eigenvalues.  We
consider first $A = \exp(B)$.  We denote by $\{ \mu_j\}$ the
eigenvalues of $A$ and by $\{ \tilde{\mu}_j\}$ the eigenvalues of
$dS(0,0)$.  Let $\mu$ be an eigenvalue of $A$.
Then $A$ symplectic implies if $\mu>1$ is real $\mu^{-1}$ is also an
eigenvalue, and if $\mu$ is complex, $|\mu| >1$, $\mu^{-1}$, $\overline{\mu}$, and $\overline{\mu}^{-1}$
are also eigenvalues.  If $\tilde\mu$, $|\tilde\mu| = 1$ is an
eigenvalue of $dS(0,0)$, then $\overline{\tilde\mu} =
\mu^{-1}$ is also an eigenvalue $dS(0,0)$, but we will see neither of
these contributes to $A$.  If $\tilde{E}_\mu$
is the generalized complex eigenspace of $\mu$, then we can put $A$ into
complex Jordan form over $\tilde{E}_\mu$.  To keep the change of variables symplectic,
we observe that $\tilde{E}_{\mu^{-1}}$ is the dual eigenspace to $\tilde{E}_\mu$, so 
if 
\be
A_\mu:= \left. A \right|_{\tilde{E}_\mu} = \left( \begin{array}{ccccc} \mu & 1 & 0 &
  \multicolumn{2}{c}\dotfill \\
0 & \mu & 1 & 0 & \ldots \\
\vdots & \ddots & \ddots  & \multicolumn{2}{c} \dotfill \\
0 & \ldots & 0 & \mu & 1 \\
0 & \multicolumn{3}{c}\dotfill & \mu \end{array} \right),
\ee
then symplectically completing this basis in $\tilde{E}_\mu \oplus
\tilde{E}_{\mu^{-1}}$ gives 
\be
\left.A \right|_{\tilde{E}_{\mu^{-1}}} = \left( A_\mu^T \right)^{-1}.
\ee
As $A_\mu = \mu I + N_\mu$ with $N_\mu$ nilpotent, by expanding
$\left( A_\mu^T \right)^{-1}$ as a power series, we see 
\be
A_{\mu^{-1}}:= \left.A \right|_{\tilde{E}_{\mu^{-1}}} = \mu^{-1}I +
N_{\mu^{-1}} 
\ee
with $N_{\mu^{-1}}$ nilpotent.  We
choose a branch of logarithm so that
\be
\lambda(\mu) = \log (\mu)
\ee
satisfies
\ben
\lambda(\mu^{-1}) & = &- \lambda (\mu), \text{ and} \label{lambda-1}\\
\lambda(\overline{\mu}) & = & \overline{\lambda(\mu)}, \label{lambda-2}
\een
and observe for $N$ nilpotent,
\be
\log(I + N) = N - \frac{N^2}{2} + \frac{N^3}{3} + \ldots
\ee
is a finite series.  Then we can define
\be
\log(\mu I + N_\mu) = \lambda(\mu) + N_\lambda,
\ee
with $N_\lambda$ nilpotent.  

We apply this technique to each generalized eigenspace of $A$ to
obtain a complex matrix
\be
\widetilde{B} : = \log A.
\ee
We see $\widetilde{B}$ is block diagonal with diagonal elements of the
form $\lambda I + N_\lambda$ with $N_\lambda$ nilpotent.  We know $|\mu|
>1$ real gives $\Re \lambda(\mu) >0$.  For $\lambda$ satisfying $\Re
\lambda >0$, let $E_\lambda$ denote the generalized complex eigenspace
of $\lambda$ with respect to $\widetilde{B}$, and let
$\tilde{E}_{\tilde\mu}$ denote the generalized complex eigenspace of $\tilde\mu$
with respect to $dS(0,0)$.  There are $4$ cases to consider.

{\bf Case 1: $\tilde\mu >1$ is real, and an eigenvalue of $dS(0,0)$.}  Then $E_\lambda \oplus E_{-\lambda}$
is a real symplectic space which is equal to $\tilde{E}_{\tilde\mu} \oplus \tilde{E}_{\tilde\mu^{-1}}$.  If we put $\widetilde{B}$ into Jordan
form over $E_\lambda$, 
\be
\widetilde{B}_\lambda:= \left. \widetilde{B} \right|_{E_\lambda} = \left( \begin{array}{ccccc} \lambda & 1 & 0 &
  \multicolumn{2}{c}\dotfill \\
0 & \lambda & 1 & 0 & \ldots \\
\vdots & \ddots & \ddots  & \multicolumn{2}{c} \dotfill \\
0 & \ldots & 0 & \lambda & 1 \\
0 & \multicolumn{3}{c}\dotfill & \lambda \end{array} \right),
\ee
completing the basis symplectically over $E_\lambda \oplus
E_{-\lambda}$ gives
\be
\left.\widetilde{B}\right|_{E_{-\lambda}} = -
\left(\widetilde{B}_{\lambda}\right)^{T}.
\ee
As $\mu= \tilde\mu$ was an eigenvalue of $A$, 
\be
\left. \exp (-JF) \right|_{\tilde{E}_{\tilde\mu} \oplus \tilde{E}_{\tilde\mu^{-1}}} = \id.
\ee

{\bf Case 2: $\tilde\mu$ is complex, $|\tilde\mu|>1$, and $\tilde\mu$ is an eigenvalue
  of $dS(0,0)$.}  Then $E_\lambda \oplus
E_{-\lambda} \oplus E_{\bar{\lambda}} \oplus E_{-\bar{\lambda}}$ is
the complexification of a real symplectic vector space which is equal
  to $\tilde{E}_{\tilde\mu} \oplus \tilde{E}_{\tilde\mu^{-1}} \oplus
  \tilde{E}_{\overline{\tilde\mu}} \oplus \tilde{E}_{\overline{\tilde\mu}^{-1}}$.  Changing
variables as in \cite{Ch} \S $6$, we see 
\be
\left. \widetilde{B} \right|_{E_\lambda \oplus
E_{-\lambda} \oplus E_{\bar{\lambda}} \oplus E_{-\bar{\lambda}}} =
\left( \begin{array}{cc} \widetilde{B}_\lambda & 0 \\ 0 & - \left(\widetilde{B}_\lambda\right)^T
\end{array} \right),
\ee
where
\be
\widetilde{B}_\lambda = \left( \begin{array}{ccccc} \Lambda & I & 0 &
  \multicolumn{2}{c}\dotfill \\
0 & \Lambda & I & 0 & \ldots \\
\vdots & \ddots & \ddots  & \multicolumn{2}{c} \dotfill \\
0 & \ldots & 0 & \Lambda & I \\
0 & \multicolumn{3}{c}\dotfill & \Lambda \end{array} \right),
\ee
with $I$ the $2 \times 2$ identity matrix and 
\be
\Lambda = \left( \begin{array}{cc}
\Re \lambda & - \Im \lambda \\
\Im \lambda & \Re \lambda
\end{array} \right).
\ee
Further,
\be
\left. \exp (-JF) \right|_{\tilde{E}_{\tilde\mu} \oplus \tilde{E}_{\tilde\mu^{-1}} \oplus \tilde{E}_{\overline{\tilde\mu}}
  \oplus \tilde{E}_{\overline{\tilde\mu}^{-1}}} = \id.
\ee

{\bf Case 3: $\mu>1$ is real, and $\tilde\mu = -\mu$ is an eigenvalue of
  $dS(0,0)$.}  Then $E_\lambda \oplus E_{-\lambda}$ is a real
  symplectic vector space, equal to $\tilde{E}_{\tilde\mu} \oplus
  \tilde{E}_{\tilde\mu^{-1}}$ and $\widetilde{B}$ is handled as in Case
  $1$, with the important difference: 
\be
\left. \exp( -J F) \right|_{\tilde{E}_{\tilde\mu} \oplus \tilde{E}_{\tilde\mu^{-1}}} = -\id.
\ee

{\bf Case 4: $|\tilde\mu| = 1$, $\Im \tilde\mu >0$ is an eigenvalue of $dS(0,0)$.}  Then
$E_\lambda \oplus E_{ -\lambda}$ is a complex symplectic vector space
which is the complexification of a real symplectic vector space which
is equal to $\tilde{E}_{\tilde\mu} \oplus \tilde{E}_{\tilde{\mu}^{-1}}$.
Since we have assumed in particular that $\tilde\mu$ occurs with multiplicity $1$, so does
$\lambda$.  Write $\lambda = i \alpha$, $\alpha >0$, in which case we
observe
\be
\left. F \right|_{E_{i\alpha} \oplus E_{-i\alpha}} =
\left(\begin{array}{cc} \alpha  & 0 \\ 0 &  \alpha \end{array}
\right)
\ee
is diagonal since $\tilde\mu$ is distinct.  

We have proved the following proposition, which we record in detail to
fix our notation.
\begin{proposition}
\label{normal-prop-5}
Let $S:W_1 \to W_2$, $W_1, W_2 \subset \reals^{2n-2}$ be a local
symplectic map, $S(0,0) = (0,0)$, and
let $n_{hc}$ be the number of Jordan blocks of complex eigenvalues $\mu$ of
$dS(0,0)$ satisfying $|\mu| > 1$, $\Re \mu > 1$, and $\Im \mu >0$; $n_{hr+}$ be the
number of Jordan blocks of real positive eigenvalues $\mu$ of $dS(0,0)$
satisfying $\mu >1$; $n_{hr-}$ be the number of Jordan blocks of negative
real eigenvalues $- \mu$ of $dS(0,0)$ satisfying $-\mu < -1$; and
$n_e$ be the number of eigenvalues $\mu$ of modulus $1$ satisfying
$\Im \mu >0$.  For 
\ben
&& j \in ( 1, \ldots, n_{hc}; 2n_{hc}+1, \ldots,
2n_{hc} + n_{hr+}; \label{j-1} \\
&& \quad \quad \quad \quad 2n_{hc} + n_{hr+} + 1, \ldots, 2n_{hc} +
n_{hr+} + n_{hr-}; \label{j-2a} \\ && \quad \quad \quad \quad 2n_{hc} +
n_{hr+} + n_{hr-} + 1, \ldots ,  2n_{hc} +
n_{hr+} + n_{hr-} + n_e ) \label{j-2}
\een
let $k_j$ denote the multiplicity of $\mu_j$ so that
\be
\lefteqn{2n-2 = } \\
&& 4 \left(\sum_{j=1}^{n_{hc}} k_j\right) + 2 \left(\sum_{j=2n_{hc}+1}^{2n_{hc} + n_{hr+}}
k_j \right) + 2\left( \sum_{2n_{hc} + n_{hr+} + 1}^{2n_{hc} + n_{hr+}
  + n_{hr-}} k_j \right) \\
&& \quad \quad \quad + 2 \left( \sum_{2n_{hc} + n_{hr+}
  + n_{hr-}+1}^{2n_{hc} + n_{hr+}
  + n_{hr-} + n_e} k_j \right).
\ee
Choose $\lambda_j
(\mu_j) = \log \mu_j$ 
satisfying (\ref{lambda-1}-\ref{lambda-2}) for $j$ in the range
\eqref{j-1} and \eqref{j-2}, and for $j$ in the range \eqref{j-2a}, choose $\lambda_j(
\mu_j ) =\log (- \mu_j)$ satisfying (\ref{lambda-1}-\ref{lambda-2}).  Then there are real
matrices $B$ and $F$ satisfying $\,^\omega B = -B$, $F^* = F$, and a symplectic choice of
coordinates such that 
\be
dS(0,0) = \exp(-JF)\exp (B),
\ee
and $B$ is of the form
\be
B = \diag  \left( B_j; -B_j^T \right),
\ee
for $j$ in the range (\ref{j-1}-\ref{j-2}).  For $j \in ( 1, \ldots n_{hc})$, $B_j$ is the $2k_j \times 2k_j$
matrix 
\ben
\label{A-2}
B_j = \left( \begin{array}{ccccc} \Lambda_j & I & 0 &
  \multicolumn{2}{c}\dotfill \\
0 & \Lambda_j & I & 0 & \ldots \\
\vdots & \ddots & \ddots  & \multicolumn{2}{c} \dotfill \\
0 & \ldots & 0 & \Lambda_j & I \\
0 & \multicolumn{3}{c}\dotfill & \Lambda_j \end{array} \right),
\een
with $I$ the $2 \times 2$ identity matrix and 
\be
\Lambda_j = \left( \begin{array}{cc}
\Re \lambda_j & - \Im \lambda_j \\
\Im \lambda_j & \Re \lambda_j
\end{array} \right).
\ee
For $j \in (2n_{hc}+1, \ldots
2n_{hc} + n_{hr+} + n_{hr-})$, $B_j$ is the $k_j \times k_j$ matrix
\ben
\label{A-1}
B_j = \left( \begin{array}{ccccc} \lambda_j & 1 & 0 &
  \multicolumn{2}{c}\dotfill \\
0 & \lambda_j & 1 & 0 & \ldots \\
\vdots & \ddots & \ddots  & \multicolumn{2}{c} \dotfill \\
0 & \ldots & 0 & \lambda_j & 1 \\
0 & \multicolumn{3}{c}\dotfill & \lambda_j \end{array} \right),
\een
and for $j \in (2n_{hc} + n_{hr+} + n_{hr-} + 1, \ldots, 2n_{hc} +
n_{hr+} + n_{hr-} + n_e)$, $B_j$ is the $1 \times 1$ matrix $0$.
Here 
\be
F= \diag \left( F_j; F_j \right),
\ee
for $j$ in the range (\ref{j-1}-\ref{j-2}), where for $j \in (1, \ldots, 2n_{hc})$, $F$ is the $2k_j \times 2k_j$
zero matrix, for $j \in ( 2n_{hc}+1, \ldots, 2n_{hc}+
n_{hr+})$,
$F_j$ is the $k_j \times k_j$ zero matrix, for $j \in (2n_{hc} + n_{hr+}+1, \ldots
2n_{hc} + n_{hr+} + n_{hr-})$, 
\be
F_j = \pi I,
\ee
where $I$ is the $k_j \times k_j$ identity matrix, and for $j \in
(2n_{hc} + n_{hr+} + n_{hr-}+1, \ldots, 2n_{hc} + n_{hr+} + n_{hr-} +
n_e)$, $F_j = \Im \lambda_j$.
\end{proposition}

Following the proof of Lemma \ref{deform-lemma}, we set
\be
K_t^1 = \exp(-tJF) \text{ and } K_t = \exp (t B ),
\ee
which we observe is the same as 
\be
K_t^1 = \exp(t H_{q^1}) \text{ and } K_t = \exp(t H_q )
\ee
for 
\ben
\lefteqn{q(x,\xi) = } \nonumber \\
& = &  \sum_{j=1}^{n_{hc}} \sum_{l=1}^{k_j}\left( \Re \lambda_j \left( x_{2l-1} \xi_{2l-1} + x_{2l} \xi_{2l} \right) 
- \Im \lambda_j \left( x_{2l-1} \xi_{2l} - x_{2l}\xi_{2l-1} \right)
\right)  \label{quad-1} \\
&& \quad \quad + \sum_{j=1}^{n_{hc}} \sum_{l=1}^{k_j-1} \left( x_{2l+1}\xi_{2l-1} + x_{2l+2}
\xi_{2l} \right) \label{quad-2}\\
 & & \quad \quad +  \sum_{j = 2n_{hc}+1}^{2n_{hc} + n_{hr+} + n_{hr-}} \left( \sum_{l=1}^{k_j}
\lambda_j x_l \xi_l  + \sum_{l=1}^{k_j - 1} x_{l+1}\xi_l \right), \label{quad-3}
\een
and
\ben
\lefteqn{ q^1(x, \xi) =} \nonumber \\
&&  \sum_{j= 2n_{hc} + n_{hr+} + 1}^{ 2n_{hc} + n_{hr+} + n_{hr-}}
\frac{ \pi}{2} (x_j^2 + \xi_j^2) + \sum_{j= 2n_{hc} + n_{hr+} +
  n_{hr-}+1}^{2n_{hc} + n_{hr+} + n_{hr-}+ n_e} \frac{\Im \lambda_j}{2}
( x_j^2 + \xi_j^2). \label{q1-ell}
\een

\subsection{Geometry of the Poincar\'e Section}
The previous section motivates the next proposition.  First we need
the following lemma, which follows from the more general \cite[Lemma 4.2]{Ch}.  Recall under the assumption that
$S$ is hyperbolic, the stable and unstable manifolds $\Lambda_\mp \subset N$ for $S$ are
$n-1$-dimensional locally embedded transversal Lagrangian submanifolds
(see \cite[Theorem 6.2.3]{kaha}).

\begin{lemma}
\label{zero-cov}
Let $S: W_1 \to W_2$, $W_1, W_2 \subset \reals^{2n-2}$, $S(0,0) = (0,0)$, be a local
hyperbolic symplectic map with unstable/stable manifolds $\Lambda_\pm$.  Then there exists a local symplectic coordinate system $(x, \xi)$ near $\gamma$ such that 
$\Lambda_+ = \{  \xi = 0 \}$ and $\Lambda_- = \{  x = 0 \}$.   
\end{lemma}

For the following proposition, we assume there are no negative real
eigenvalues and no eigenvalues of modulus $1$ to the linearized
Poincar\'e map.  Later we will modify the general Poincar\'e map to be
of this form.  This follows from the proof of \cite[Proposition 4.3]{Ch}.

\begin{proposition}
\label{normal-prop-6}
Let $S : W_1 \to W_2$, $W_1, W_2 \subset \reals^{2n-2}$, be a local
hyperbolic symplectic map, $S(0,0) = (0,0)$, and assume $dS(0,0)$ has
no negative real eigenvalues.  There is a smooth
family of local symplectomorphisms $\kappa_t$, a smooth, 
real-valued matrix function $B_t(x,\xi)$, and a symplectic choice of
coordinates in which
\ben
&& \text{(i)} \,\, \kappa_0 = \id, \,\,\, \kappa_1(x,\xi) =
S(x,\xi); \nonumber \\
&& \text{(ii)} \,\, \frac{d}{dt} \kappa_t = \left( \kappa_t \right)_*
H_{q_t}, \label{quad-th-0}
\een
where 
\ben
\label{quad-th-1}
q_t(x,\xi) =  \lll B_t(x, \xi)x,\xi \rrr .
\een
Here
\ben
\label{quad-th-2}
\lll B_t(0,0)x, \xi \rrr = q(x, \xi),
\een
for $q(x, \xi)$ of the form (\ref{quad-1}-\ref{quad-3}).
\end{proposition}


\section{The Proof of Theorem \ref{main-theorem-5}}
\label{theorem-proof-chapter}

\numberwithin{equation}{section}
\subsection{Motivation}
We recall from Theorem \ref{main-theorem-4} that if $u \in L^2(X)$ has wavefront
set sufficiently close to $\gamma$ and $B \in \Psi^{0,0}(V)$ is a
microlocal cutoff near $(0,0)$, we have for $z \in [-\epsilon_0 ,
  \epsilon_0] + i(-c_0h, c_0h)$, 
\be
\left\| (P-z) u \right\|_{L^2(X)} \geq C^{-1} h \left\|B(I - M(z)) R_+
u \right\|_{L^2(V)}.
\ee
Hence we want to show $M(z)$ has spectrum away from $1$.  This is the
content of the following Theorem, which we state in its general
form for reference.

\begin{theorem}
\label{FIO-def-thm}
Let $\tV \subset \reals^{2m}$ be an open neighbourhood of $(0,0)$, and
assume 
$\kappa_z: \neigh (\tV) \to \kappa_z(\neigh (\tV))$, $\kappa_z(0,0) = (0,0)$, $z \in (-
\delta, \delta)$, $\delta>0$ is a smooth
family of symplectomorphisms such that $d
\kappa_z(0,0)$ is semi-hyperbolic and the nonresonance condition \eqref{non-res} holds for $d \kappa_z(0,0)$.  Let $M(z)$ be the microlocally unitary
$h$-FIO which quantizes $\kappa_z$ as in Proposition \ref{AF=FB}.
Then for $z \in (-\delta', \delta')$, $\delta'>0$ sufficiently small
and 
$s \in \reals$ sufficiently close to $0$, there exist self-adjoint, semiclassically
tempered operators $\exp(\pm s K^w)$ so that 
for $v \in L^2(\reals^m)$ with $h$-wavefront set sufficiently close to
$(0,0)$,
\ben
\left\| e^{-sK^w}M(z)e^{sK^w} v \right\|_{L^2}
\leq \frac{1}{R} \| v \|_{L^2}. \label{W-M}
\een
\end{theorem}

From \S \ref{mono}, we know $M(z)$ is an $h$-FIO associated to the
graph of $S(z)$, where $S(z)$ is the Poincar{\'e} map for $\gamma_z$,
the periodic orbit in the energy level $z$.  Suppose for the moment
that $S(z)$ satisfies the hypotheses of Proposition
\ref{normal-prop-6}, and let $q_{z,t}$ be
$q_t$ as in the conclusion of the Proposition, where now $q_{z,t}$ varies
over energy levels $z$ near $0$.  Setting $Q_{z,t} =
\Op_h^w (q_{z,t})$, from Corollary \ref{F-cor} there exists $M_{z,0} \in
\Psi_h^{0,0}$ microlocally unitary so that $M(z) = M^z(1)$ for $M^z(t)$ a
family of operators satisfying the
evolution equations
\be
hD_t M^z + M^zQ_{z,t} & =&  0, \,\,\, 0 \leq t \leq 1,  \\
M^z(0) & = & M_{z,0}. 
\ee

In order to prove Theorem \ref{FIO-def-thm}, we observe if $W(z):L^2(V) \to L^2(V)$ is the microlocal inverse for $M(z)$, we
have also $W_{z,0} \in \Psi_h^{0,0}$ microlocally unitary so that $W(z) = W^z(1)$ for $W^z(t)$ satisfying the following evolution equation:
\ben
hD_tW^z - Q_{z,t} W^z & = & 0, \,\,\, 0 \leq t \leq 1, \label{w-ode-1} \\
W^z(0) & = & W_{z,0}. \label{w-ode-2}
\een
The rest of this section is devoted to proving there exist
semiclassically tempered operators $\exp(\pm s K^w)$ as in the
statement of the Theorem so that 
\ben
\label{W-est-1001}
\left\| e^{-sK^w}W(z)e^{sK^w} v \right\|_{L^2(V)} \geq R \| v \|_{L^2(V)}
\een
for some $R>1$.  Then 
\be
 \|v\|_{L^2(V)}& = & \left\| e^{-sK^w}W(z)e^{sK^w}e^{-sK^w}
 M(z)e^{sK^w} v \right\|_{L^2(V)} \\
& \geq & R \left\|
e^{-sK^w}M(z)e^{sK^w} v \right\|_{L^2(V)},
\ee
which gives the Theorem once we prove \eqref{W-est-1001}.

In order to get Theorem \ref{main-theorem-5} from Theorem \ref{FIO-def-thm}, we
observe by \eqref{W-M} we have also
\be
\left\| \left( I - e^{-sK^w} M(z) e^{sK^w} \right) v \right\|
\geq C^{-1} \| v \|.
\ee
Thus
\be
\Re \lll e^{-sK^w} (I - M) e^{sK^w} v, v \rrr & = & \| v \|^2 - \Re
\lll e^{-sK^w} M(z) e^{sK^w} v, v \rrr \\
& \geq & C^{-1} \|v\|_{L^2(V)}^2.
\ee
Since $\|\exp(\pm sK^w)\|_{L^2 \to L^2} = \O(h^{-N})$ for some $N$, we
have 
\be
\Re \lll (I - M(z)) v, v \rrr \geq Ch^N \|v\|^2.
\ee

Now let $u \in L^2(X)$ have wavefront set close to $\gamma$.  Set $v =
R_+ u$ so that $\WF v$ is close to $(0,0)$, and observe with $B$ as in Theorem \ref{main-theorem-4} and $b
= \sigma_h(B)$, $1-b$ has support away from $(0,0) \in T^* \reals^{n-1}$.  Then 
\be
\Op_h^w( 1-b ) M(z) v = M(z) \Op_h^w ( S(z)^*(1-b)) v = \O(h^\infty),
\ee
so that if $\WF u$ is sufficiently small,
\be
B(I - M(z)) R_+ u = (I-M(z))R_+ u
\ee
microlocally, and \eqref{inj-2a} gives the theorem.

\qed

Our biggest tool so far is the normal form deformation in Proposition
\ref{normal-prop-6}, however we cannot immediately apply it to $S(z)$
satisfying the assumptions of the introduction.  To get by this we will transform
$S(z)$ into a hyperbolic map satisfying the assumptions of Proposition
\ref{normal-prop-6} and then later deal with the errors which
come up when transforming back. 

The proof of Theorem \ref{FIO-def-thm} will proceed in $4$ basic steps.  First, we deform the
effective Hamiltonian into a sum of two Hamiltonians with disjoint support in $t$, one
hyperbolic and one elliptic.  The summed Hamiltonian will be called
$q_{z,t}$.   
We then modify the evolution equation defining $W^z$ to an equation 
involving a conjugated version of $W^z$, $\widetilde{W}(t)$.
This evolution equation will be given in terms of a conjugated quantization of $q_{z,t}$,
$\widetilde{Q}_{z,t}$, that we will then need to estimate from below.
This step is a variation on the classical idea of a ``positive
commutator''.  That is, $\Op_h^w(q_{z,t})$ is self-adjoint, but if we
conjugate it with an operator of the form $e^{G^w}$, we get
$\Op_h^w(q_{z,t})$ plus a lower order {\it skew}-adjoint commutator.
The principal symbol of the commutator $[G, \Op(q_{z,t})]$ is $hi
H_{q_{z,t}}G$.  The linear part of $H_{q_{z,t}}$ is block diagonal in
the hyperbolic and elliptic variables, but the nonlinear part
potentially forces interaction between the hyperbolic and elliptic
variables.  Hence we will be forced to introduce a complex weight $G$
to gain some orthogonality between the hyperbolic and elliptic
variables.  This is accomplished in Step 3.  
Finally we will estimate $M$, whose inverse is related to 
$\widetilde{W}$ by conjugation.

\subsection{Step 1: Deform $q_{z,t}$}

We construct a rescaled deformation of identity into $S(z)$ in which
the elliptic part of the effective Hamiltonian has
disjoint support in $t$ from the support of the non-elliptic part.

We will be using four cutoff functions,
\be
\psi_1(t), \psi_2(t), \psi(t),\text{ and }\chi(t): [0,1] \to [0,1]
\ee
satisfying the following
properties (see Figure \ref{fig:fig7}):
\be
&& \, \,\text{  (i) } \psi_1(0) = \psi_2(0) = \psi(0) = \chi(0) = 0,
\,\,\, \psi_1(1) = \psi_2(1) = 
\psi(1) = \chi(1) = 1; \\
&& \, \text{ (ii) } \psi_1', \,\, \psi', \text{ and } \chi' \text{ are
  all
  non-negative,} \\
&& \text{ (iii) } \supp \psi_1' \subset [0,1/4], \,\, \supp \chi'
\subset [1/4, 1/2], \\ 
&& \quad  \supp \psi' \subset [1/2,3/4], \text{ and } \supp \psi'
\subset [3/4,1]. \\
\ee
\begin{figure}
\centerline{
\input{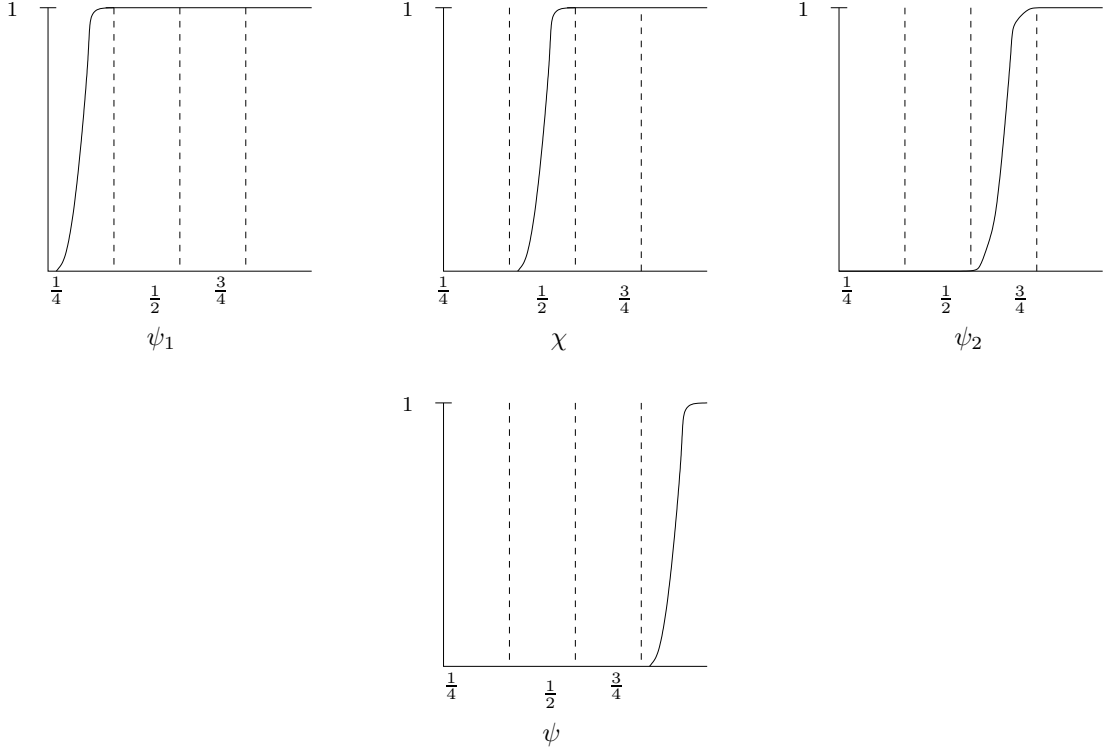}}
\caption{\label{fig:fig7} The cutoff functions $\psi_1$, $\psi_2$, $\psi$, and $\chi$.}
\end{figure} 

Motivated by Lemma \ref{ODE-deform-lemma} and Proposition
\ref{normal-prop-6}, we construct a family of symplectomorphisms, ${\kappa}_{z,t}$, satisfying ${\kappa}_{z,0} = \id$ and
${\kappa}_{z,1} = S(z)$, but the elliptic part has disjoint
support in $t$ from the hyperbolic part.  That is, let $F$ be given as in Proposition
\ref{normal-prop-5}, and let 
\be
E = \exp ( - J F), \,\,\, K^1_t = \exp ( - t J F),
\ee
so that $K_0 = \id$, $K_1 = E$, and 
\be
\frac{d}{dt} K^1_t = (K^1_t)_* H_{q^1},
\ee
where $q^1$ is given by \eqref{q1-ell}.  Let $\widetilde{K}^1_t$ be
defined by
\be
\widetilde{K}^1_t = K^1_{\psi_1(t)},
\ee
so that $\widetilde{K}^1_0 = \id$, $\widetilde{K}^1_1 = E$, and the chain rule then gives
\be
\frac{d}{dt} \widetilde{K}^1_t & = &  \psi_1'(t) \frac{d}{d \tau} K^1_\tau
|_{ \tau = \psi_1(t)} \\
& = & \psi_1'(t) (K^1_\tau)_* H_{q^1} |_{ \tau = \psi_1(t)} \\
& = & (\widetilde{K}^1_t)_* H_{\psi_1'(t) q^1}.
\ee

We introduce an ``artificial hyperbolic'' transformation which will
temporarily replace the elliptic part by setting
\be
q_{ah} = \sum_{j=2h_{hc} + n_{hr+} + n_{hr-} + 1}^{2h_{hc} + n_{hr+} +
  n_{hr-} + n_e} 2 x_j \xi_j,
\ee
defining $K_{ah} = \exp (H_{q_{ah}})$, and
\be
\widetilde{S}(z) = K_{ah}^{-1} \circ E^{-1} \circ S(z) , 
\ee 
so that $ \widetilde{S}(z)$ satisfies the assumptions of Proposition
\ref{normal-prop-6} near $z=0$.
From Proposition \ref{normal-prop-6}, there is a family $\kappa^1_{z,t}$
satisfying $\kappa^1_{z,0} = \id$, $\kappa^1_{z,1} = \widetilde{S}(z)$, and
\be
\frac{d}{dt} \kappa^1_{z,t} = (\kappa^1_{z,t})_* H_{\tilde{q}_{z,t}}, 
\ee
where now 
\be
\tilde{q}_{z,t} = \lll B_{z,t}(x, \xi)x,\xi \rrr
\ee
for $B_{z,t}$ satisfying \eqref{quad-th-2}.  Let 
\be
\tkappa_{z,t} = \kappa^1_{z,\psi(t)},
\ee
so that $\tkappa_{z,0} = \id$, $\tkappa_{z,1} = \widetilde{S}(z)$, and
\be
\frac{d}{dt} \tkappa_{z,t} & = & \psi'(t) \frac{d}{d \tau} \kappa^1_{z,\tau}
|_{\tau = \psi(t)} \\
& = & \psi'(t) ( \kappa^1_{z,\tau})_* H_{\tilde{q}_{z,\tau}} |_{\tau = \psi(t)} \\
 & = & ( \tkappa_{z,t} )_* H_{ \psi'(t) \tilde{q}_{z, \psi(t)}}.
\ee
Let $K_t^2 = \exp(t H_{q_{ah}})$ and $\widetilde{K}^2_t =
K^2_{\psi_2(t)}$, so that $\widetilde{K}^2_0 = \id$,
$\widetilde{K}^2_1 = K_{ah}$, and
\be
\frac{d}{dt} \widetilde{K}^2_t = (\widetilde{K}^2_t)_* H_{\psi_2'(t)
  q_{ah}}.
\ee
Finally, let 
\be
{\kappa}_{z,t} = \widetilde{K}^1_t \circ \widetilde{K}^2_t \circ
\tkappa_{z,t} .
\ee
Unraveling the definitions,
we have ${\kappa}_{z,t}$ satisfying
\be
&& \text{(i)  } {\kappa}_{z,0} = \id, \,\,\, {\kappa}_{z,1} 
= S(z); \\
&& \text{(ii) } {\kappa}_{z,t} = \left\{ \begin{array}{l}
    \widetilde{K}^1_t, \,\,\, 0 \leq t \leq 1/4 ; \\
E, \,\,\, 1/4 \leq t \leq 1/2 ; \\
E \circ \widetilde{K}^2_t, \,\,\, 1/2 \leq t \leq 3/4; \\
E \circ K_{ah} \circ \tkappa_{z,t}.
\end{array} \right.
\ee 
If we compose a smooth function $a$ with ${\kappa}_{z,t}$, using Lemma
\ref{H-push-lemma}, we have 
\be
\frac{d}{dt} {\kappa}_{z,t}^* a  & = & \left\{ \begin{array}{l}
\frac{d}{dt} a( \widetilde{K}^1_t ), \,\,\, 0 \leq t \leq 1/4; \\
\frac{d}{dt} a(E), \,\,\, 1/4 \leq t \leq 1/2; \\
\frac{d}{dt} a(E) \circ \widetilde{K}^2_t, \,\,\, 1/2 \leq t \leq 3/4;
\\
\frac{d}{dt} a(EK_{ah}) \circ \tkappa_{z,t}, \,\,\, 3/4 \leq t \leq
1; \end{array} \right. \\
& = & \left\{ \begin{array}{l}
(H_{\psi_1'(t) q^1} a ) \circ \widetilde{K}^1_t, \,\,\, 0 \leq t \leq
1/4; \\
0 , \,\,\, 1/4 \leq 1/2; \\
\left( H_{(E^{-1})^* \psi_2'(t) q_{ah}} a \right) \circ E \circ
\widetilde{K}^2_t, \,\,\, 1/2 \leq t \leq 3/4; \\
\left( H_{(K_{ah}^{-1})^* (E^{-1})^* \psi'(t) \tilde{q}_{z, \psi(t)}} a
\right) \circ E \circ K_{ah} \circ \tkappa_{z,t}, \,\,\, 3/4 \leq t \leq
1. \end{array} \right.
\ee
Summing up and using the support properties of $\psi$, $\psi_1$,
and $\psi_2$, we have
\be
\frac{d}{dt} \kappa_{z,t} = (\kappa_{z,t})_* H_{\tilde{q}_{z,t}^2},\ee
where
\ben
\label{q-tilde-2}
\tilde{q}_{z,t}^2 = (E^{-1} K_{ah}^{-1})^*\psi'(t)\tilde{q}_{z,\psi(t)} +
(E^{-1})^*\psi_1'(t) q^1 + \psi_2'(t) q_{ah}.
\een

We record for later use that since $\psi = \psi' = 0$ and $\psi_2 = \psi_2'=0$ on the support of
$\chi'$, we have for $t \in \supp \chi'$,
\ben
\label{kappa-1a}
{\kappa}_{z,t} = E = \left( \begin{array}{ccc} \id_{h+} & 0  & 0 \\ 0 &
    -\id_{h-}  & 0 \\ 0 & 0 & \tilde{E}  \end{array} \right),
\een
where $\id_{h+}$ is identity in $x_j$ and $\xi_j$ for 
\be
1 \leq j \leq 2n_{hc} + n_{hr+},
\ee
$\id_{h-}$ is the identity in $x_j$ and $\xi_j$ for 
\be
2n_{hc} + n_{hr+} + 1 \leq j \leq 2n_{hc} + n_{hr+} + n_{hr-},
\ee
and $\tilde{E}$ is an elliptic symplectic transformation in the
variables $x_j$ and $\xi_j$ for
\be
2n_{hc} + n_{hr+} + n_{hr-} + 1 \leq j \leq 2n_{hc} + n_{hr+} +
n_{hr-} + n_e.
\ee

\subsection{Step 2: Conjugation of Evolution Equations}

For Step 2, we introduce the following notation.  By $(X_\hyp, \Xi_\hyp)$ and
$(X_\el, \Xi_\el)$ we mean the symplectic variables in the subspace
associated to the hyperbolic and elliptic parts of $dS(0,0)$
respectively.  In our notation,
\be
X_\hyp = (X_1, \ldots, X_{n-n_e-1}), \,\,\, \Xi_\hyp = ( \Xi_1,
\ldots, \Xi_{n-n_e-1}),
\ee
and
\be
X_\el = (X_{n-n_e}, \ldots, X_{n-1}), \,\,\, \Xi_\el = (\Xi_{n-n_e},
\ldots, \Xi_{n-1}).
\ee
For a vector $Y \in \reals^{n-1}$, we define also 
\be
|Y|_\hyp^2 &= & \sum_{j=1}^{n-n_e-1} Y_j^2 \text{ and} \\
|Y|_\el^2 & = & \sum_{j=n-n_e}^{n-1} Y_j^2,
\ee
where as usual $n_e = n-1 - 2n_{hc} - n_{hr+} - n_{hr-}$.  
If 
\be
\lll B \cdot, \cdot \rrr: \reals^{n-1} \times \reals^{n-1} \to \cx 
\ee
is a bilinear form, we will also use
the notation
\be
\lll B Y, (Z_\hyp, iZ_{\el}) \rrr = \sum_{j=1}^{n-1}
\sum_{k=1}^{n-n_e-1} B^{jk} Y_j Z_k + i \sum_{j=1}^{n-1}
\sum_{k=n-n_e}^{n-1} B^{jk} Y_j Z_k.
\ee

Let $W(z) = M(z)^{-1}$ as above and let $Q_{z,t} = \Op_h^w( \tilde{q}_{z,t}^2)$ for
$\tilde{q}_{z,t}^2$ in the form \eqref{q-tilde-2}.
Applying Corollary \ref{F-cor}, there is $W^z (t)$ and $W_{z,0}$
unitary 
satisfying (\ref{w-ode-1}-\ref{w-ode-2}) with this choice of $Q_{z,t}$
so that $W^z(1) = W(z)$.  As in \S \ref{model-chapter}, but with
$W^z$ instead of $M^z$, if we
conjugate $W^z(t)$ satisfying (\ref{w-ode-1}-\ref{w-ode-2}) in a way
which is independent of $t$, we get a new 
equation with a conjugated $Q_{z,t}$.  That is, with $T_{h,
  \tilde{h}}$ defined 
in \eqref{U-hsc}, let
\be
W^{z,1}(t) = T_{h, \tilde{h}}W^z(t) T_{h,
  \tilde{h}}^{-1}
\ee
and observe $W^{z,1}(t)$ satisfies
\be
hD_t W^{z,1} - Q_{z,t}^1 W^{z,1} &=& 0, \,\,\, 0 \leq t \leq 1\\
W^{z,1}(0) & = & T_{h, \tilde{h}}W_{z,0} T_{h,
  \tilde{h}}^{-1}
\ee
for $Q_{z,t}^1 = T_{h, \tilde{h}} Q_{z,t} T_{h, \tilde{h}}^{-1}$.

We define the escape function $G$ in the new coordinates by
\ben
\label{G-esc-def}
G(X,\Xi) & = & \half \log \left( \frac{1 + |X|_\hyp^2}{1 + |\Xi|_\hyp^2}
\right) + i \half (|X_\el|^2 - | \Xi_\el|^2 ) \\
& =:& G_1 + i G_2. \nonumber
\een
Here we have added an imaginary term to the definition of $G$.
Observe 
\be
\exp (i \Op_{\tilde{h}}^w (G_2))
\ee
is unitary.  As mentioned in the introduction to this section, this is
used to control the nonlinear interactions between the hyperbolic and
elliptic variables in a Poisson bracket later in the proof.

The real part of $G$, $G_1$, satisfies
\begin{eqnarray*}
\left| \partial_X^\alpha \partial_\Xi^\beta G_1(X, \Xi) \right| \leq C_{\alpha \beta} \langle X \rangle^{-|\alpha|}\langle 
\Xi \rangle^{-|\beta|}, \,\,\,\, \,\,\, \mathrm{for} \,\,\, (\alpha, \beta) \neq (0,0),
\end{eqnarray*}
and since $\langle X \rangle^2 \langle \Xi \rangle^{-2} $ is an order
function, $\Re G$ satisfies the assumptions of Lemma \ref{etG-lemma}.
Thus we can construct the operators $e^{\pm s \chi(t)
  G^w}$, where $G^w$ is the $\tilde{h}$-Weyl quantization of $G$, and doing so we may define
\ben
\label{tilde-W-def}
\widetilde{W}(t) = e^{-s\chi G^w} W^{z,1}(t) e^{s\chi G^w}.
\een
Similar to \S \ref{model-chapter}, $\widetilde{W}$ satisfies the
evolution equation
\ben
hD_t \widetilde{W} - \widetilde{Q}_{z,t} \widetilde{W} & =& 
\frac{h}{i}s\chi'(t) e^{-s\chi G^w}\left[W^{z,1}, G^w \right] e^{s \chi G^w}, \,\,\, 0 \leq
t \leq 1 \label{w-ode-2-a}\\
\widetilde{W}(0) &=& e^{-s\chi(0) G^w} T_{h, \tilde{h}}W_{z,0} T_{h,
  \tilde{h}}^{-1} e^{s\chi(0) G^w}, \label{w-ode-2-b}
\een
where
\be
\widetilde{Q}_{z,t} = e^{-s\chi G^w}Q_{z,t}^1   e^{s\chi G^w}.
\ee
The definition of $W^{z,1}$ together with Proposition
\ref{2-param-egorov} means 
\be
\chi'(t)\left[  W^{z,1}, G^w \right]  & = & \chi'(t) T_{h, \tilde{h}}
\left[ W^z, T_{h, \tilde{h}}^{-1} G^w T_{h, \tilde{h}} \right] T_{h,
  \tilde{h}}^{-1} \\
& = & \chi'(t) T_{h, \tilde{h}} \Op_h^w \left( 
{\kappa}_{z,t}^* \widetilde{G} - \widetilde{G} + \O(h^{1/2} \tilde{h}^{3/2})
\right) W^z T_{h, \tilde{h}}^{-1}  ,
\ee
where
\be
\widetilde{G}(x, \xi) = G\left( (\tilde{h} / h )^\half (x, \xi)
\right) \in \s_{\half}^{-\infty,
  0,0} \text{ microlocally}.
\ee  
From \eqref{kappa-1a} and the definition of $G$, 
\be
\Re {\kappa}_{z,t}^* \widetilde{G} = \Re \widetilde{G}
\ee
on $\supp \chi'$.
Hence, using Lemma \ref{2-param-lemma} and Proposition
\ref{2-param-egorov}, there is a symbol $e_t \in \s_{0}^{-\infty,
  -1/2, -3/2}$ such that
\be
\Im \lefteqn{ \frac{h}{i}s\chi'(t) e^{-s\chi G^w}\left[W^{z,1}, G^w \right] e^{s \chi
  G^w} =} \\ & = & \Im \frac{h}{i} s\chi'(t) \left( \Op_{\tilde{h}} (e_t) +
  \frac{\tilde{h}}{i} s \chi'(t) G^w \Op_{\tilde{h}} (\{e_t, G \} )  +
\O(h^{1/2}\tilde{h}^{7/2})\right) \\
& = & \O ( h^{3/2}\tilde{h}^{3/2}).
\ee

\subsection{Step 3: Estimation of $\widetilde{Q}_{z,t}$}

We want to gain some knowledge of $\widetilde{Q}_{z,t}$.  For that we use
the techniques from the proof of Theorem 1 in \cite{Ch} together with
the necessary modifications discussed in the introduction.  We summarize
the content of this Step in the following Lemma:

\begin{lemma}
For $\widetilde{Q}_{z,t}$ as defined above, we have the estimate
\ben
-\Im \langle \widetilde{Q}_{z,t} u, u \rangle  \geq
\psi'(t)  \frac{h \tilde{h}}{C} \| u \|^2, \label{hsc-est}
\een
for any $u \in L^2( \reals^{n-1})$.
\end{lemma}

The idea is
that the conjugated $\widetilde{Q}_{z,t}$ is $Q_{z,t}^1$ to leading order, which
is self-adjoint, and the second order term is roughly the quantization of 
\be
\frac{\th}{i} H_{q} G
\ee
for a quadratic form $q$.  The following Proposition is from
\cite[Theorem 4]{Ch}
and says that for the quadratic forms in which we are interested we can make $H_q G$ into a positive definite
quadratic form, and there are linear symplectic coordinates in which $H_qG$
is almost the harmonic oscillator $\sum_j x_j^2 + \xi_j^2$.

\begin{proposition}
\label{non-distinct-theorem}
Suppose $q \in \Ci (\reals^{2m})$ is quadratic of the form
\begin{eqnarray}
\lefteqn{ q(x,\xi)  = } \nonumber \\ 
&&  =\sum_{j=1}^{n_{hc}} \sum_{l=1}^{k_j}\left( \Re \lambda_j \left( x_{2l-1} \xi_{2l-1} + x_{2l} \xi_{2l} \right) 
- \Im \lambda_j \left( x_{2l-1} \xi_{2l} - x_{2l}\xi_{2l-1} \right)
\right)  \label{p-1a}\\
&& \quad \quad + \sum_{j=1}^{n_{hc}} \sum_{l=1}^{k_j-1} \left( x_{2l+1}\xi_{2l-1} + x_{2l+2}
\xi_{2l} \right) \label{p-1b}\\
 & & \quad \quad +  \sum_{j = 2n_{hc}+1}^{2n_{hc} + n_{hr}} \left( \sum_{l=1}^{k_j}
\lambda_j x_l \xi_l  + \sum_{l=1}^{k_j - 1} x_{l+1}\xi_l \right), \label{p-1c}
\end{eqnarray}
and 
\begin{eqnarray*}
G(x, \xi) = \half \left( \log (1 + |x|^2) - \log(1 + |\xi|^2 ) \right).
\end{eqnarray*}
Then there exist $m \times m$ positive definite matrices $M$
and $M'$, positive real numbers $0 < r_1 \leq r_2, \leq \cdots \leq
r_{m}< \infty$, and linear symplectic coordinates $(x, \xi)$ such that
\begin{eqnarray}
\label{non-distinct-theorem-statement}
H_q(G) = \frac{\sum_{j=1}^m r_j^{-2} x_j^2 }{1 + |M x|^2} +
\frac{\sum_{j=1}^m r_j^{-2} \xi_j^2 }{1 + |M' \xi|^2}.
\end{eqnarray} 
\end{proposition}

Let $U$ be 
a neighbourhood of $(0,0)$, $U \subset T^* \reals^{n-1}$, and assume 
\begin{eqnarray*}
U \subset U_{\epsilon/2} := \left\{ ( x, \xi):| (x, \xi)| < \frac{\epsilon}{2} \right\}
\end{eqnarray*}
for $\epsilon >0$.  We assume throughout that we are working
microlocally in $U_\epsilon$.  With $\tilde{h}$ small (fixed later 
in the proof), we have done the following rescaling:
\begin{eqnarray}
\label{rescaling1}
X : = \hsc^\half x,  \, \Xi := \hsc^\half \xi 
\end{eqnarray}
We assume for the remainder of the proof that $|(X, \Xi)| \leq \hsc^\half \epsilon$.  We used the unitary operator 
$T_{h, \tilde{h}}$ defined in \eqref{U-hsc} to introduce the second
parameter into $Q_{z,t}$ to get
\begin{eqnarray*}
Q_{z,t}^1 = T_{h, \tilde{h}} Q_{z,t} T_{h, \tilde{h}}^{-1}
\end{eqnarray*}
as above.  On the support of $\chi(t)$, after a linear symplectic
change of variables, we write 
\be
Q_{z,t}^2 = T_{h, \tilde{h}} \Op_h (\psi'(t)\tilde{q}_{z,\psi(t)} +
(K_{ah}E)^* \psi_2'(t) q_{ah}) T_{h,
  \tilde{h}}^{-1},
\ee
where $\tilde{q}_{z,t} = \lll B_{z,t} x, \xi \rrr$ defined in Step 1.
The principal symbol of $Q_{z,t}^2$ on $\supp \chi'$ is 
\begin{eqnarray}
\label{q1-expression}
\lefteqn{ q_{z,t}^2( X, \Xi )} \\ & = &  q_{z,t}^3(X, \Xi) +
q_{ah}((h/\th)^\half(X,\Xi)) \nonumber \\
& = & \psi'(t)\left\langle B_{z,\psi(t)}\left(
\csh^{\half}( X, \Xi) \right) \csh^\half X, \csh^\half \Xi
\right\rangle \nonumber \\
&& \quad + \psi_2'(t) \sum_{j=2h_{hc} + n_{hr+} + n_{hr-} + 1}^{2h_{hc} + n_{hr+} +
  n_{hr-} + n_e} (h/ \th) 2 X_j \Xi_j, \nonumber
\end{eqnarray}
and $q_{z,t}^2 \in \s_{-\half}^{-\infty,0,0}$ microlocally.  We have 
\begin{eqnarray}
\label{q1-est}
\left|\partial_{X, \Xi}^\alpha q_{z,t}^2\right| \leq C_\alpha\csh^{|\alpha|/2}
\end{eqnarray}
for $(X, \Xi) \in U_{\hsc^\half \epsilon}$ by Lemma
\ref{U-hsc-lemma}. 

Now $\Re \{ G, q_{ah} ((h/\th)^{1/2}(X, \Xi)) \} =0$, so to find the
real part of $H_{q_{z,t}^2} G$, we need only calculate
$\Re H_{q_{z,t}^3} G$.  
For $|(X, \Xi)| \leq \hsc^{\half} \epsilon$ we have with $G$
as above in \eqref{G-esc-def}
\begin{eqnarray}
\lefteqn{ H_{q_{z,t}^3} G(X, \Xi) = } \nonumber \\
& = & \csh \psi'(t) \left[ \left\langle B_{z,\psi(t)} X, \frac{\partial}{\partial X} \right\rangle
  - \left\langle B_{z,\psi(t)} \frac{\partial}{\partial \Xi}, \Xi \right\rangle
  \right]  G(X, \Xi) \label{hqg-main-term} \\
& &  + \csh^{\frac{3}{2}} \psi'(t) \left[ \sum_{j = 1}^{n-1} \left\langle \frac{\partial}{\partial \Xi_j}B_{z,\psi(t)}( \cdot,\cdot ) 
X, \Xi \right\rangle \frac{\partial}{\partial X_j}G(X, \Xi)\right] \label{hqg-err1} \\
& &  - \csh^{\frac{3}{2}} \psi'(t) \left[ \sum_{j = 1}^{n-1} \left\langle \frac{\partial}{\partial X_j} B_{z,\psi(t)}( \cdot,\cdot ) X, 
\Xi \right\rangle \frac{\partial}{\partial \Xi_j}  G(X, \Xi )\right]. \label{hqg-err2}
\end{eqnarray}

Now owing to 
Lemma \ref{2-param-lemma} and
\eqref{q1-est} we have microlocally to leading order in $h$:
\begin{eqnarray*}
\Re \ad_{G^w}^k \left( Q_{z,t}^2
 \right) = \O_{L^2 \to L^2}\left( h \tilde{h}^{k-1} \right),
\end{eqnarray*}
and in particular,
\begin{eqnarray}
\label{GQ-comm}
i \Im \left[ Q_{z,t}^2, G^w \right]  =  -i \tilde{h} \Re \Op_{\tilde{h}}^w \left( H_{q_{z,t}^2} G\right) 
+ \O (h^{3/2} \tilde{h}^{3/2}).
\end{eqnarray}

Estimating the real part of the errors (\ref{hqg-err1}-\ref{hqg-err2}), we get
\begin{eqnarray}
\Re \lefteqn{ \csh^{\frac{3}{2}}\left[ \sum_{j = 1}^{n-1} \left\langle
  \frac{\partial}{\partial \Xi_j}B_{z,\psi(t)}( \cdot, \cdot ) X, 
\Xi \right\rangle \frac{\partial}{\partial X_j}G(X, \Xi)\right] } \nonumber \\
&& 
=\csh^{\frac{3}{2}}\frac{1}{1 + |X_\hyp|^2}\O(|\Xi||X||X_\hyp|), \label{hpg-err1a}
\end{eqnarray}
and analogously for \eqref{hqg-err2}.  At $(0,0)$, $B_{z,\psi(t)}$
is positive definite and block diagonal of the form \eqref{quad-th-2}, so we compute:
\begin{eqnarray*}
\left| \left\langle B_{\psi(t)}(0,0) X, \left(\frac{X_\hyp}{1 + |X_\hyp|^2},
    iX_\el \right) \right\rangle \right| & \geq &
C^{-1} \left( \frac{ |X_\hyp|^2}{1 + |X_\hyp|^2} + |X_\el|^2 \right) \\
& = & C^{-1} \frac{|X|^2 + |X_\hyp|^2 |X_\el|^2}{1+|X_\hyp|^2}.
\end{eqnarray*}
Hence
\ben
\label{error-est-1001} \lefteqn{ \Re \frac{1}{1 + |X_\hyp|^2}\O(|\Xi||X||X_\hyp|)  \frac{\left\langle B_{\psi(t)}(0,0) X, \left(\frac{X_\hyp}{1 + |X_\hyp|^2},
    iX_\el \right) \right\rangle}{ \left| \left\langle B_{\psi(t)}(0,0) X, \left(\frac{X_\hyp}{1 + |X_\hyp|^2},
    iX_\el \right) \right\rangle \right|} } \\
&& = \Re \left\langle B_{\psi(t)}(0,0) X, \left(\frac{X_\hyp}{1 + |X_\hyp|^2},
    iX_\el \right) \right\rangle \O(|\Xi|), \nonumber
\een
and analogously for \eqref{hqg-err2}.  Now we expand $B_{z,\psi(t)}$ in a Taylor approximation about $(0,0)$ to get
\begin{eqnarray*}
\Re \lefteqn{H_{q_{z,t}^3}G   = } \\
& = &  \Re \csh \psi'(t) \Bigg[  \left\langle B_{z,\psi(t)}(0,0) X,
  \left(\frac{X_\hyp}{1+|X|_\hyp^2} , i X_\el
  \right) \right\rangle \\
&& \quad \quad \quad + \Re \csh^{\half} \O\left(\frac{|X||X_\hyp|}{1+|X_\hyp|^2}|(X, \Xi)|\right)\Bigg]  \\ 
& & + \Re \csh \psi'(t) \Bigg[ \left\langle B_{z,\psi(t)}(0,0) \Xi,
  \left(\frac{\Xi_\hyp}{1+|\Xi|_\hyp^2} , i \Xi_\el
  \right) \right\rangle \\
&& \quad \quad \quad + \Re \csh^{\half} \O \left( \frac{|\Xi||\Xi_\hyp|}{1+|\Xi_\hyp|^2} |(X, \Xi)|\right)\Bigg] , 
\end{eqnarray*}
which, from \eqref{error-est-1001}, is


\be
\Re \lefteqn{H_{q_{z,t}^3}G } \\
& = &  \Re \csh \psi'(t) \left\langle B_{z,\psi(t)}(0,0) X,
  \left(\frac{X_\hyp}{1+|X|_\hyp^2} , i X_\el
  \right) \right\rangle \\
&& \quad \quad \quad \cdot \left( 1 + \csh^{\frac{1}{2}}
\O ( |\Xi|) \right)\\ 
& & \quad +\Re \csh \psi'(t) \left\langle B_{z,\psi(t)}(0,0) \Xi,
  \left(\frac{\Xi_\hyp}{1+|\Xi|_\hyp^2} , i \Xi_\el
  \right) \right\rangle \\
&& \quad \quad \quad \cdot \left( 1 +\csh^{\frac{1}{2}} 
\O ( |X|)\right). 
\ee

Now since $B_{z,\psi(t)}(0,0)$ is block 
diagonal of the form \eqref{quad-th-2}, Proposition
\ref{non-distinct-theorem} yields a linear symplectomorphism $\kappa_1$ such that 
\begin{eqnarray*}
\Re \lefteqn{ \kappa_1^* (H_{q_{z,t}^2}(G)) = } \\
& = &  \csh \psi'(t) \Bigg[ \frac{\sum_{j=1}^{n-n_e-1}
    r_j^{-2} X_j^2 }{1 + |M X|^2}  \left( 1 + \csh^{\frac{1}{2}}
\O ( |\Xi|) \right) \\
&& +
\frac{\sum_{j=1}^{n-n_e-1} r_j^{-2} \Xi_j^2 }{1 + |M' \Xi|^2} \left( 1 +\csh^{\frac{1}{2}} 
\O ( |X|)\right)
\Bigg] ,
\end{eqnarray*}
where $M$ and $M'$ are nonsingular.
Thus, since $\chi(t) \psi_1'(t) = 0$,
\begin{eqnarray}
\lefteqn{ \Im \widetilde{Q}_{z,t}} \nonumber\\
 & = & \Im s \chi(t) [ Q_{z,t}^2, G^w ] + s\chi(t)E_1^w   +
s^2\chi(t)^2 E_2^w \nonumber \\
& = &  - sh\chi(t) (A_1( 1 + E_0) + A_2(1 + E_0'))^w  \nonumber \\
&& \quad + s\chi(t)E_1^w   +
s^2\chi(t)^2 E_2^w , \label{Qth2}
\end{eqnarray}
with $E_0, E_0' = \O( \epsilon)$, $E_1=\O
(h^{3/2} \tilde{h}^{3/2})$, $E_2 = \O(h \tilde{h})$, and 
$(A_1 + A_2)^w =: A^w = \Op_{\tilde{h}}^w(A)$ for 
\begin{eqnarray}
A(X, \Xi) = \psi'(t) (\kappa_1^{-1})^* \left(\frac{\sum_{j=1}^{n-n_e-1} r_j^{-2} X_j^2 }{1 + |M X|^2} +
\frac{\sum_{j=1}^{n-n_e-1} r_j^{-2} \Xi_j^2 }{1 + |M' \Xi|^2} \right).
\end{eqnarray}

From Proposition \ref{AF=FB} there is a unitary $h$-FIO $F_1$ quantizing
$\kappa_1^{-1}$ so that 
\be
\tilde{A} : = F_1 \Op_{\tilde{h}}^w(A) F_1^{-1} = \Op_{\tilde{h}}^w(
\kappa_1^* A ) + \O( \tilde{h}^2).
\ee

We claim that for $\tilde{h}$ sufficiently small and $\tilde{v}$ smooth, 
\begin{eqnarray*}
\langle \tilde{A}^w\tilde{v}, \tilde{v} \rangle \geq \frac{\tilde{h}}{C} \|\tilde{v} \|^2
\end{eqnarray*}
for some constant $C>0$, which is essentially the lower bound for the harmonic oscillator 
$\tilde{h}^2 D_X^2 + X^2$.  It suffices to prove this
inequality for individual $j$, which is the content of Lemma
\ref{harm-osc}.  As $F_1$ is unitary, setting $\tilde{v} = F_1
\tilde{u}$ for $\tilde{u}$ smooth gives
\ben
\langle A^w\tilde{u}, \tilde{u} \rangle & \geq & \frac{\tilde{h}}{C}
\|\tilde{u} \|^2  - \O( \tilde{h}^2) \| \tilde{u} \|^2 \nonumber \\
& \geq &   \frac{\tilde{h}}{C'}
\|\tilde{u} \|^2, \label{A-lower-bound} 
\een
for $\tilde{h}>0$ sufficiently small.

Now fix $\tilde{h}>0$ and $|s|>0$ sufficiently small so that the
 estimate \eqref{A-lower-bound} holds and the errors $E_1$ and $E_2$ satisfy
\begin{eqnarray*}
\|shA^w \tilde{u} \|_{L^2} \gg \| s E_1^w \tilde{u} \|_{L^2} + \|s^2
E_2^w \tilde{u} \|_{L^2},
\end{eqnarray*}
and fix $\epsilon>0$ sufficiently small that the errors $|E_0|, |E_0|' \ll 1$,
independent of $h>0$. 

For $\tilde{u}$ a smooth function with wavefront set contained in $U$,
we now have
\begin{eqnarray*}
-\Im \langle \widetilde{Q}_{z,t} \tilde{u}, \tilde{u} \rangle & \geq &
\psi'(t) \chi(t) \frac{h \tilde{h}}{C} \| \tilde{u} \|^2 \\
& = & \psi'(t)  \frac{h \tilde{h}}{C} \| \tilde{u} \|^2, 
\end{eqnarray*}
since $\chi(t) \equiv 1$ on the support of $\psi_1'(t)$.  This is
\eqref{hsc-est}, the
crucial estimate needed for Step 4.

If $\tilde{q}_{z,t}$ is not in the form \eqref{quad-th-1}, by Proposition \ref{normal-prop-6} there is a symplectomorphism $\kappa_2$ so that $\kappa_2^*\tilde{q}_{z,t}$ 
is of the form \eqref{quad-th-1}.  Using Proposition \ref{AF=FB} to quantize $\kappa_2$ as an $h$-FIO
$F_2$, we get
\begin{eqnarray*}
\Op_h^w \left( \kappa^* \tilde{q}_{z,t} + E_1 \right) = F^{-1} \widetilde{Q}_{z,t} F ,
\end{eqnarray*}
where $E_1 = \O(h^2)$ is the error arising from Proposition
\ref{AF=FB}.  We may then 
use the previous argument for $\kappa_2^*q_{z,t}$ getting an additional
error of $\O(h^2)$ from Proposition \ref{AF=FB} in 
\eqref{hsc-est}.

The following Lemma comes from \cite[Lemma 5.1]{Ch}.
\begin{lemma}
\label{harm-osc}
Let
\begin{eqnarray*}
a_0(y, \eta):= \frac{y_j^2}{ \langle y \rangle^2} + \frac{\eta_j^2}{ \langle \eta \rangle^2},
\end{eqnarray*}
for $(y, \eta) \in \reals^{2n-2}$, and $\langle y \rangle = ( 1 +
|y|^2)^{1/2}$, and let 
\begin{eqnarray*}
a_1(y, \eta) := \frac{y_{2j}^2 + y_{2j-1}^2}{\langle y \rangle^2} +
\frac{\eta_{2j}^2 + \eta_{2j-1}^2}{\langle \eta \rangle^2}.
\end{eqnarray*}
Then $a_i$, $i=0,1$ satisfies
\begin{eqnarray}
\label{ai-ineq}
\langle \Op_{\tilde{h}}^w(a_i)\tilde{u}, \tilde{u} \rangle \geq \frac{\tilde{h}}{C} \|\tilde{u} \|^2
\end{eqnarray}
for $\tilde{h}>0$ sufficiently small and a constant $0<C<\infty$.
\end{lemma}

\subsection{Step 4: Estimation of $\widetilde{W}$}

Let $v \in L^2(V)$ with
wavefront set sufficiently close to $(0,0)$, and set $\tilde{v} =
T_{h, \tilde{h}} v$.  Now $\widetilde{W}(t)$ is no longer unitary, so we calculate
\be
\partial_t \lll \widetilde{W}(t) \tilde{v}, \widetilde{W}(t) \tilde{v} \rrr & = & 2 \lll
\partial_t \widetilde{W}(t) \tilde{v}, \widetilde{W}(t) \tilde{v} \rrr \\
& = & \frac{2i}{h} \lll \left( \widetilde{Q}_{z,t}  + \O ( h^{3/2} \tilde{h}^{3/2}) 
 \right)
\widetilde{W}(t) \tilde{v}, \widetilde{W}(t) \tilde{v}
\rrr \\
& = & -\frac{2}{h} \lll \left( \Im \widetilde{Q}_{z,t} + \O ( h^{3/2} \tilde{h}^{3/2})
\right) \widetilde{W}(t) \tilde{v}, \widetilde{W}(t) \tilde{v} \rrr \\ 
& \geq & C^{-1} \left( \psi'(t)\tilde{h} - \O ( h^{1/2} \tilde{h}^{3/2}) \right) \lll
\widetilde{W} (t)\tilde{v}, \widetilde{W}(t) \tilde{v} \rrr.
\ee 
Thus there is a positive constant $C$ such that 
\be
\partial_t \left( \lll \widetilde{W}(t) \tilde{v} , \widetilde{W}(t) \tilde{v} \rrr
e^{-(\psi(t)\tilde{h} - \O ( h^{1/2} \tilde{h}^{3/2}) )/C} \right) \geq 0,
\ee
so
\be
\left\| \widetilde{W}(t) \tilde{v} \right\|^2 \geq e^{\psi(t)(\tilde{h} - \O ( h^{1/2} \tilde{h}^{3/2}) )/C} \| \widetilde{W}(0) \tilde{v} \|^2
\ee
and since $\psi(1) = 1$, shrinking $\tilde{h}>0$ if
necessary, we have for $0<h\leq h_0$ sufficiently small,
\be
\left\| \widetilde{W}(1) \tilde{v} \right\| \geq R \| \widetilde{W}(0)
\tilde{v} \|, \,\,\, R>1 \text{ independent of } 0 < h \leq h_0.
\ee
Now
\be
\widetilde{W}(0) & = & e^{-s \chi(0) G^w} T_{h, \tilde{h}} W^z(0) T_{h,
  \tilde{h}}^{-1} e^{ s \chi(0) G^w} \\
& = &  T_{h, \tilde{h}} W^z(0) T_{h,
  \tilde{h}}^{-1} 
\ee
is unitary, so
\ben
\label{t-W-est}
 \left\| \widetilde{W}(1) \tilde{v} \right\| \geq R \| \tilde{v} \|,
\een
independent of $0<h \leq h_0$.

As in \S \ref{model-chapter}, let the operators $K^w$ be defined by 
\be
e^{sK^w} = T_{h, \tilde{h}}^{-1} e^{s\chi(1)G^w} T_{h, \tilde{h}} = T_{h, \tilde{h}}^{-1} e^{sG^w} T_{h, \tilde{h}},
\ee
so that 
\be
\widetilde{W}(1) = e^{-sK^w} M(z)^{-1} e^{sK^w},
\ee
and Theorem \ref{FIO-def-thm} is proved.
\qed

\begin{remark}
The error arising at the end of the proof of Theorem 
\ref{FIO-def-thm} from the use of Theorem \ref{AF=FB} is of
order $\O(h^2)$ and hence negligible compared to our lower bound of
$h$ for $A$.  However, the estimate of $A$ is used for the imaginary
part of $\widetilde{Q}_{z,t}$, and the error in Theorem \ref{AF=FB}
is real, so $\O(h)$ would
have been sufficient.  This means the analysis above does not strictly
depend on using the Weyl calculus.
\end{remark}

\begin{remark}
It is interesting to note that the estimate (\ref{main-theorem-5-est}) depends only on the real parts of the eigenvalues 
$\lambda_j$ above.  Unraveling the definitions, the eigenvalues $\lambda_j$ are logarithms of the eigenvalues of the 
linearized Poincar\'{e} map $dS(0)$ from above.  Then (\ref{main-theorem-5-est}) depends only on the modulis of the 
eigenvalues of $dS(0)$ which lie off the unit circle.  We interpret
this as a quantum analogue of the fact that $dS(0,0)$ is semi-hyperbolic. 
\end{remark}


\section{Proof of Theorem \ref{main-theorem-2} and the Main Theorem}
\label{main-theorem-proof-chapter}
\numberwithin{equation}{section}

\subsection{Proof of Theorem \ref{main-theorem-2}}
In this section we show how to use Theorem \ref{main-theorem-5} with a few other results to deduce 
Theorem \ref{main-theorem-2}.  

\begin{proposition}
\label{Q(z)u-prop}
Suppose $\psi_0 \in \s^{0,0}(T^*X) \cap \Ci_c(T^*X)$ is a microlocal
cutoff function to a sufficiently small neighbourhood 
of $\gamma \subset \{p^{-1}(0) \}$.  For $Q(z) = P(h) - z - iCha^w$ as above with $z \in [-1,1] + i(-c_0h, \infty)$, 
$c_0 >0$ and $C>0$ sufficiently large, we have
\begin{eqnarray}
\label{Q(z)u-prop-est}
Q(z)u = f \Longrightarrow \left\| (1 - \psi_0 )^w u \right\| \leq C h^{-1} \|f \| .
\end{eqnarray}
\end{proposition}
For this proposition and the proof, we use the convenient shorthand notation: for a symbol $b$, $b^w := \Op_h^w(b)$.
\begin{remark}
Note that Proposition \ref{Q(z)u-prop} is the best possible situation.  It says roughly that away from $\gamma$, 
$Q^{-1}$ is bounded by $Ch^{-1}$.  Thus the global statement in Theorem \ref{main-theorem-2} represents a loss of 
$\sqrt{\log (1/h)}$.
\end{remark}
\begin{proof}
Choose $c_0 >0$ from Theorem \ref{main-theorem-5}, and assume $\supp
\psi_0 \cap \supp a = \emptyset$,
Choose $C>0$ sufficiently large so that  
\begin{eqnarray*}
(Ca-c_0)^w (1 - \psi_0)^w \geq 
c_0 (1 - \psi_0)^w /2 \geq
c_0 ((1 - \psi_0)^w)^*(1-\psi_0)^w/2.
\end{eqnarray*}
Then we calculate
\begin{eqnarray*}
\frac{1}{2} c_0 h \int_X \left| (1 - \psi_0)^w u \right|^2 dx & \leq & h \int_X \left( Ca^w + h^{-1} \Im z \right) 
u \overline{ (1- \psi_0)^w u } dx \\
& = & -\Im \int_X Q(z) u \overline{ (1-\psi_0)^w u} dx \\
& = & -\Im \int_X f \overline{ (1-\psi_0)^w u} dx \\
& \leq & \|f\|  \left\| (1 - \psi_0)^w u \right\|  \\
& \leq & (4\epsilon h)^{-1} \|f \|^2 + \epsilon h \left\| (1 -\psi_0)^w u \right\|^2
\end{eqnarray*}
for any $\epsilon>0$ fixed.  Taking $\epsilon \ll c_0$ yields \eqref{Q(z)u-prop-est}.
\end{proof}

We need the following lemma, which
follows with little modification from \cite[Lemma 6.1]{Ch}.  
\begin{lemma}
\label{wf-lemma}
Suppose $V_0 \subset T^*X$, $P \in \Psi^{k,0}_h$ (or
$P \in \Psi_{h,db}^{2,0}$ differential with homogeneous principal
symbol if $\gamma \cap \partial X  \neq \emptyset$), $T >0$, $A$ an operator, and $V \subset T^*X$ a neighbourhood of $\gamma$ satisfying
\begin{eqnarray}
\left\{ \begin{array}{l}
\forall \rho \in \{ p^{-1}(0) \} \setminus V, \,\,\, \exists \, 0 <t<T \,\, \text{and} \,\, \epsilon = \pm 1 \,\,\, 
\text{such that} \\
\exp(\epsilon s H_p)(\rho) \subset \{ p^{-1}(0) \} \setminus V \,\,
\text{for} \,\, 0 < s < t, \,\,  \\
\exp(\epsilon s H_p)(\rho) \text{ is non-glancing for } 0 \leq s \leq
t, \text{ and} \\
\exp(\epsilon t H_p)(\rho) \in V_0;  \\
\end{array} \right.
\end{eqnarray}
and $A$ is microlocally elliptic in $V_0 \times V_0 $.  If $B \in \Psi^{0,0}(X, \Omega_X^\half)$ and $\WF (B) \subset T^*X 
\setminus V$, then
\begin{eqnarray*}
\left\| Bu \right\| \leq C \left( h^{-1} \left\| Pu \right\| + \| Au \| \right) + \O (h^\infty) \|u\|.
\end{eqnarray*}
\end{lemma}

We will need the next lemma, which is essentially an operator version of the classical Three-Line Theorem from complex 
analysis.  The proof can be found in \cite[Lemma 6.3]{Ch}, which is collected from \cite[Lemma A.2]{BZ}, \cite[Lemma 4.7]{Bur2}), and
\cite[Lemma 2]{TaZw}.
\begin{lemma}
\label{BZ-lemma}
Let $\mathcal{H}$ be a Hilbert space, and assume $A,B: \mathcal{H} \to \mathcal{H}$ are bounded, self-adjoint operators 
satisfying $A^2 = A$ and $BA = AB = A$.  Suppose $F(z)$ is a family of bounded operators satisfying 
$F(z)^* = F(\bar{z})$, $\Re F \geq C^{-1} \Im z$ for $\Im z >0$, and further assume
\begin{eqnarray*}
BF^{-1}(z) B \,\, \text{is holomorphic in}\,\, \Omega:= [-\epsilon, \epsilon] + i[-\delta, \delta], \,\, 
\text{for} \,\, \frac{\delta}{\epsilon} \ll M^{-\frac{1}{N_1}} < 1
\end{eqnarray*}
for some $N_1>0$, where $\|B F^{-1}(z)B \| \leq M$.  Then for $|z| < \epsilon/2$, 
$\Im z = 0$,
\begin{eqnarray*}
(a) \quad \left\| B F^{-1}(z) B \right\| & \leq & C \frac{\log M}{\delta}, \\
(b) \quad \left\| B F^{-1}(z) A \right\| & \leq & C \sqrt{ \frac{\log M}{\delta}}.
\end{eqnarray*}
\end{lemma}

\begin{proof}[Proof of Theorem \ref{main-theorem-2}]
Let $\psi_0$ satisfy the assumptions of Proposition \ref{Q(z)u-prop}.  Then
\begin{eqnarray*}
\| (1 -\psi_0)^w u \| \leq Ch^{-1} \| Q(z) u\| .
\end{eqnarray*}
Further, since 
\begin{eqnarray*}
\left\| \left[ Q, \psi_0^w \right] u \right\| \leq \left\| \left[ Q, \psi_0^w \right] (1 - \tilde{\psi}_0^w)u \right\| 
,
\end{eqnarray*}
for some $\tilde{\psi}_0$ satisfying the assumptions of Proposition \ref{Q(z)u-prop} and $\WF \tilde{\psi}_0 \subset 
\{ \psi_0 = 1\}$, so using Theorem \ref{main-theorem-5}, the fact that
$[Q, \psi_0^w]$ is compactly supported (in the hyperbolic regions if
$\gamma \cap \partial X \neq \emptyset$) and 
of order $h$, we have
\begin{eqnarray*}
\| \psi_0^w u \| & \leq & Ch^{-N_0} \left( \| \psi_0^w Q u \| +
  \left\| \left[ Q, \psi_0^w \right] u \right\| \right) + \O(h^\infty)
\| u \|\\
& \leq & C h^{-N_0} \left(  \|\psi_0^w Q u\| + h^{-1} \| h Q u\|
\right) + \O(h^\infty) \|u \| \\
& \leq & C h^{-N_0} \| Qu \| + \O(h^\infty)\| u \|.
\end{eqnarray*}
This follows immediately from Lemma \ref{wf-lemma} with $A = h^{-1}[Q,
\psi_0^w](1 - \tilde{\psi}_0)^w$ and $B = \psi_0^w$.

Now let $F(w)$ be the family of operators $F(w) = ih^{-1} Q(z_0 + hw)$, $A = \chi_{\supp \phi}^w$, $B = \id$. 
 Fix $\delta >0$ independent of $h$, $\epsilon = (Ch)^{-1}$, $M = h^{-N_0}$, and apply Lemma \ref{BZ-lemma} to get
\begin{eqnarray*}
\|B F^{-1} B \| & \leq & C \log (h^{-N_0}); \\
\|B F^{-1} A \| & \leq & C \sqrt{ \log (h^{-N_0}) },
\end{eqnarray*}
and (\ref{main-theorem-2-est-1}-\ref{main-theorem-2-est-2}) follows.
\end{proof}
\subsection{Proof of the Main Theorem}

The proof of the Main Theorem now will follow as in \cite{Ch} by
a commutator argument, although we will need to take some care at the
boundary.  Let $a^w$ be a symbol which is microlocally $1$ away from
$\gamma$, and for $z \in [-\epsilon_0, \epsilon_0] + i(-c_0h, c_0h)$,
define as in \cite{BZ} and \cite{Ch} and the introduction
\ben
Q(z) := P(h) - z - iCha^w; \,\,\, C>0 \,\, \text{fixed}.
\een
For the analysis near the boundary, choose also $\psi \in \Ci ( \reals
)$ satisfying (\ref{psi-assump-100}-\ref{psi-assump-101}).

Let $m_j^\pm \in T^*X$, for $j = 1, \ldots, K$
denote the points where $\gamma$ reflects off the boundary, with
$m_j^\pm$ denoting the point of intersection with the boundary of the
outgoing and incoming bicharacteristics respectively, and let $m_j$ be
the projection of $m_j^\pm$ onto $T^*(\partial X)$.  Let $U_j \subset
T^*(\partial X)$
denote a neighbourhood of $m_j$ which is small enough so that a
factorization of $P$ as in Lemma \ref{p-fact-lemma} is possible in a
neighbourhood of $U_j$.  Shrinking $U_j$ if necessary, we assume also that the construction
in Lemma \ref{comm-lemma} is valid in a neighbourhood of $U_j$.
That is, if $P$ is factorized as in Lemma \ref{p-fact-lemma} near
$U_j$, we write 
\be
P = (hD_1 - A_-(x, hD'))(hD_1 - A_+(x, hD')) \,\,\, \text{near} \,\,
m_j,
\ee
and there is an operator $A_{b,j}(x, hD')$ which is $1$ microlocally
near $U_j$, zero away from $U_j$ and commutes with $(hD_1-A_+(x,
hD'))$ microlocally near $U_j$.  

Let $\gamma_\pm^j$ be a small interval on the outgoing/incoming
bicharacteristic near $m_j^\pm$, and let $\widetilde{U}_j \subset
T^*X$ be a neighbourhood of $\gamma_\pm^j$ such that 
\be
(\WF (\psi(P) A_{b,j}))|_{\partial X} \subset \widetilde{U}_j,
\ee
and $\psi(P) A_{b,j} \equiv 1$ on $\gamma_\pm^j \cap \widetilde{U}_j$.  Choose
$\chi_j \in \Ci_c(T^*X)$, $\chi_j \equiv 1$ on $\widetilde{U}_j$ with
sufficiently small support that
\ben
\label{chij-grad}
\psi(P) A_{b,j} \equiv 1 \text{ on } \supp \nabla \chi_j \cap \gamma.
\een
Finally, set 
\be
\chi_0 = 1 - \sum_j \chi_j.
\ee

Now for $A \in \Psi_{h,db}^{0,0}$ as in the statement of the Main Theorem
with wavefront set sufficiently close to $\gamma$, let $A_0 \in
\Psi_h^{0,0}$ have wavefront set close to $\gamma$ and satisfy 
\ben
&& A_0 \equiv 1 \text{ on } \left\{\WF A \setminus \left\{ \bigcup_j \widetilde{U}_j \right\}
\right\}, \nonumber \\
&& A_0 \equiv 1 \text{ on } \left(\bigcup_{j=0}^K \supp \nabla \chi_j\right) \cap
\gamma, \label{chi0-grad} \\
&& A_0 \equiv 0 \text{ elsewhere } .  \nonumber
\een
We define $\tilde{A} \in \Psi_{h,db}^{0,0}$
satisfying 
\ben
\label{tilde-A-WF}
\tilde{A} \equiv 1 \text{ on } \WF A 
\een
by
\be
\tilde{A} = \chi_0 A_0 + \sum_j \chi_j \psi(P) A_{b,j},
\ee
where $\psi$ satisfies (\ref{psi-assump-100}-\ref{psi-assump-101}).
Observe if $\WF A$ is sufficiently close to $\gamma$, $\tilde{A}$
satisfies \eqref{tilde-A-WF}.  
We have $Q(0)\tilde{A}u = P(h)\tilde{A}u$ since $\WF a^w \cap \WF \tilde{A}
= \emptyset$.  But
\ben
\label{P-comm-id}
P(h) \tilde{A} u  =  \left[ P(h), \tilde{A} \right] u + \tilde{A}
P(h) u
\een
and we claim
\ben
\left\| \left[ P, \tilde{A} \right] u\right\| & = & \left\| \left[
  P, \chi_0  A_0 \right]u + \sum_j \left( \left[ P, \chi_j \psi(P)
  A_{b,j} \right] \right) u  \right\| \nonumber \\
& = & \O(h) \left\| (I - A) u \right\|. \label{comm-est-1}
\een
To see this, we observe for $ u \in \Ci(X) \cap L^2(X)$,
\be
\lefteqn{ \left[
  P, \chi_0 A_0 \right]u + \sum_j  \left[ P, 
  \chi_j A_{b.j} \right] =}  \\
 & = &  \chi_0 \left[ P,A_0
  \right]u + [P, \chi_0] A_0 u \\
&& \quad \quad + \sum_j \left( \chi_j \psi(P) \left[ P, A_{b,j} \right]  +
  \left[ P , \chi_j \right] \psi(P) A_{b,j}\right) u 
\ee
We have
\be
 \| ( [P, \chi_0] A_0 + \sum_j \left[ P , \chi_j \right] \psi(P) A_{b,j}) u
   \| \leq Ch \|(I -A) u \|
\ee
from \eqref{chij-grad} and \eqref{chi0-grad}.  These two conditions
also imply
\be
&& \WF \chi_0 \left[ P,A_0
  \right] \cap \gamma = \emptyset \text{ and } \\
&& \WF \chi_j \psi(P) \left[ P, A_{b,j} \right] \cap \gamma = \emptyset,
\ee
and the symbol of $A_0$ is compactly suppported away from the boundary, so
\be
\left\| \chi_0 \left[
  P,  A_0 \right]u \right\| \leq C h \| (I -A ) u \|.
\ee
For each $j$, it suffices to consider the
remaining terms in local coordinates at the boundary.  Fix $j$ and
assume we are in the coordinates used in Lemma \ref{comm-lemma} in
$U_j$:
\be
\lefteqn{ \chi_j \psi(P)\left[ P, A_{b,j}(x, hD') \right] =} \\
&=& \chi_j \psi(P) \left[ (hD_1 - A_-(x, hD'))(hD_1 -
  A_+(x, hD')), A_{b,j} \right] \\
& = & \chi_j \psi(P) \left[ (hD_1 - A_-(x, hD')), A_{b,j}(x,hD') \right] ( hD_1 -
A_+(x, hD')),
\ee
since $A_{b,j}$ commutes with $(hD_1 - A_+(x, hD'))$.  
The principal symbol of 
\be
\chi_j \psi(P) \left[ (hD_1 - A_-(x, hD')), A_{b,j}(x,hD')
  \right] 
\ee
is 
\be
\frac{h}{i} \chi_j \psi((\xi_1 - r^\half(x, \xi'))(\xi_1 + r^\half(x, \xi')))\left\{ (\xi_1 + r^\half(x, \xi')), \sigma_h (
A_{b,j}^+)(x, \xi') \right\} ,
\ee
which is $\O(h)$ and has $h$-wavefront set away from $\gamma$.
Summing over $j$ gives \eqref{comm-est-1}.

Combining \eqref{main-theorem-2-est-2}, \eqref{P-comm-id}, and \eqref{comm-est-1}, we have
\be
 \|u \|_{L^2(X)} & \leq & \left( \left\|\tilde{A} u \right\|_{L^2(X)} + \left\| (I - A) u \right\|_{L^2(X)} 
\right) \\
& \leq & C \left( h^{-1} \sqrt{\log (1/h)}  \left\| P \tilde{A} u \right\|_{L^2(X)} +
\left\|(I-A) u \right\|_{L^2(X)} \right) \\
& \leq & C \left( \sqrt{\log (1/h)} + C^{-1} \right) \left\| (I - A ) u
\right\|_{L^2(X)} \\
&& \quad \quad \quad \quad + C \frac{ \sqrt{ \log(1/h)}}{h} \| P u
\|_{L^2(X)},
\ee
which for $0 < h \leq h_0$ is the statement of the Main Theorem.



\section[Quasimodes]{An Application: Quasimodes near Elliptic Orbits}
\label{quasi}
\numberwithin{equation}{section}

In this section, we show how the techniques of reducing microlocal
estimates near a periodic orbit to estimates on an $h$-Fourier
integral operator acting microlocally on the Poincar\'e section via
the {\it Quantum Monodromy operator} from \cite{SjZw} and \S\ref{grushin} can be used
with the quasimode construction in \cite{ISZ} to produce
well-localized quasimodes near an elliptic periodic orbit.  We also
give estimates on the number and location of approximate eigenvalues
associated to the quasimodes.

Let $X$ be a smooth, compact manifold, $\dim X = n$, and suppose $P \in \Psi^{k,0}(X)$, $k
\geq 1$, be a
semiclassical pseudodifferential operator of real principal type which
is semiclassically elliptic outside a
compact subset of $T^*X$ as in the introduction.  
Let $\Phi_t = \exp tH_p$ be the classical flow of $p$ and assume there is
a closed {\it elliptic} orbit $\gamma \subset \{ p = 0 \}$.  That $\gamma$
is elliptic means if $N \subset \{ p = 0 \}$ is a Poincar'e section
for $\gamma$ and $S:N \to S(N)$ is the Poincar\'e map, then $dS(0,0)$ has
eigenvalues all of modulus $1$.  We will also
need the following non-resonance assumption:
\ben
\label{nonres}
\left\{ \begin{array}{l} \text{if } e^{\pm i \alpha_1}, e^{\pm i \alpha_2}, \ldots, e^{\pm i \alpha_k}
 \text{ are eigenvalues of }
dS(0,0), \text{ then } \\ \alpha_1 , \alpha_2, \ldots,
 \alpha_k 
\text{ are independent over }  \pi \ZZ.
\end{array} \right.
\een

Finally, we assume if $\gamma \cap \partial X \neq \emptyset$ then
$\gamma$ reflects only transversally off $\partial X$, $\partial X$
is noncharacteristic with respect to $P$, and $P \in
\Diff_{h,db}^{2,0}$.

Under these assumptions, it is well known that there is a family of
elliptic closed orbits $\gamma_z \subset \{ p = z \}$ for $z$ near
$0$, with $\gamma_0 = \gamma$.  In this work we consider the following eigenvalue problem for $z$ in a
neighbourhood of $z=0$:
\ben
\label{ev-prob-1}
\left\{ \begin{array}{l} (P - z) u = 0; \\ \| u \|_{L^2(X)} =
  1. \end{array} \right.
\een
We prove the following Theorem.
\begin{theorem}
\label{main-theorem-1}
For each $m \in \ZZ$, $m >1$, and each $c_0>0$ sufficiently small, there is a finite, distinct family of values
\be
\{z_j \}_{j = 1}^{N(h)} \subset [-c_0 h^{1/m}, c_0 h^{1/m} ]
\ee
and a family of quasimodes $\{ u_j\} = \{u_j(h)\}$ with
\be
\WF u_j = \gamma_{z_j},
\ee
satisfying
\ben
\label{ev-prob-2}
\left\{ \begin{array}{l} (P - z_j) u_j = \O(h^\infty) \| u_j \|_{L^2(X)}; \\ \| u_j \|_{L^2(X)} =
  1. \end{array} \right.
\een
Further, for each $m \in \ZZ$, $m >1$, there is a constant $C = C(c_0,{1/m})$ such that
\ben
\label{ev-est-1}
C^{-1} h^{-n(1-1/m)} \leq N(h) \leq C h^{-n}.
\een
\end{theorem}

\subsection{The Model Case}
\label{model-case-section}
We consider the case $n = 2$, the first nontrivial dimension.  Recall the model for $p$ near an elliptic periodic orbit is $p \in \Ci
( T^*( \SS^1 \times \reals))$, 
\be
p = \tau +  \frac{\alpha}{2} ( x^2 + \xi^2),
\ee
with $\alpha >0$ satisfying $\alpha \notin \pi \ZZ$.  Then we study
\eqref{ev-prob-1} for 
\be
P = hD_t +  \frac{\alpha}{2} ( x^2 + h^2 D_x^2).
\ee
Let 
\be
Q & = &  \frac{\alpha}{2} ( x^2 + h^2 D_x^2) \\
& = & \Op_h^w \left(  \frac{\alpha}{2} ( x^2 +
\xi^2) \right).
\ee
$Q$ is just $\alpha/2$ times the harmonic osciallator, so we have 
\be
Q v_k = h \frac{ \alpha}{2} (2k+1) v_k
\ee
for
\be
v_k & := &  h^{-1/4}H_k(x/h^\half) e^{-x^2/2h}, \\
\| v_k \|_{L^2} & = & 1,
\ee 
where $H_k$ are the (normalized) Hermite polynomials of degree $k$ (see, for
example, \cite{EvZw}).  Note $\WF v_k = (0,0)$.  Now we make an ansatz of
\be
u =  g_k(t) v_k(x) 
\ee
for $g_k(t)$ to be determined.  Plugging $u$ into \eqref{ev-prob-1}
yields
\be
hD_t g_k + \frac{\alpha}{2}h(2k+1) g_k = z g_k,
\ee
which implies
\be
g_k(t) = \exp \left(\frac{it}{h} \left(z -
    \frac{\alpha}{2}(2k+1)h\right)\right). 
\ee
Since the spectrum of $hD_t$ on $\SS^1$ is $\{ 2 \pi m h \}_{m \in
  \ZZ}$, we have
\ben
\label{z-range-1}
z = \frac{ \alpha}{2} (2k+1) h + 2 \pi m h.
\een
In the model case, since there is no microlocalization necessary (and,
in particular, $p$ is not elliptic at infinity), we actually have
dense spectrum in any interval.

In order to motivate our general construction, we present the same
example from the point of view of the monodromy operator.  Here we
think of $Q-z$ as a $z$-dependent family of operators on $L^2( V)$, where $V \subset \reals$ is
an open neighbourhood of $0$.  Then the monodromy operator $M(z)$ is defined
microlocally as the time $t=1$ solution to the ordinary differential
equation
\ben
\label{mono-inv-eq-1}
\left\{ \begin{array}{l} hD_t M(z, t) + (Q-z) M(z,t) = 0, \\ M(z,0) =
    \id_{L^2(V) \to L^2(V)}. \end{array} \right.
\een
Our general technique will be to find eigenfunctions of $M(z) =
M(z,1)$ with eigenvalue
$1$.  Using again $v_k$ as in the
previous paragraph, we try
\be
M(z,t) v_k = e^{-i  2 \pi m t} v_k,
\ee
with $m \in \ZZ$ so that $M(z,1) v_k = v_k$.  This yields from \eqref{mono-inv-eq-1} 
\be
\left( -h 2 \pi m +\frac{\alpha}{2} h (2k+1) -z \right) v_k = 0
\ee
which is the same as \eqref{z-range-1}.

\subsection{Quasimodes on the Poincar\'e section}
\label{qm-section}
Theorem \ref{ml-equiv-thm} and the definition of the monodromy
operator $M(z)$ motivate us to study the normal form for a family of
elliptic 
symplectomorphisms
\be
S_z:  W_1 \to W_2 
\ee
under the nonresonance condition \eqref{nonres} on $dS(0)$, where $W_1$
and $W_2$ are neighbourhoods of $0 \in \reals^{2n-2}$.  We use the
standard 
notation of \cite{ISZ} and write
\be
\imath_j & = & x_j^2 + \xi_j^2, \text{ and}\\
I_j & = &  \imath_j^w = x_j^2 + h^2 D_{x_j}^2.
\ee
According to the results of \cite{IaSj} and \cite{ISZ}, there is a symplectic choice
of coordinates near $(x, \xi; z) = (0,0;0)$ such that
\ben
\label{S-z-n-form}
S_z = \exp H_{q_z} + \O( (x, \xi;z)^\infty),
\een
for
\be
q_z = \sum_{j=1}^{n-1} \lambda_j(z) \imath_j  + R(z, \imath_1, \ldots,
\imath_{n-1}).
\ee
Here the remainder $R(z, \imath) = \O( \imath^2)$ and the
$\lambda_j(z)$ are positive and depend smoothly on $z$.

Further, if $M(z)$ is the monodromy operator quantizing $S_z$ and
\ben
\label{quant-cond-1}
&&\text{(i)  } z \in [-\epsilon_0 h^{1/m}, \epsilon_0 h^{1/m}] +
i(-c_0h, c_0h), \\
&&\text{(ii) }\imath_j \leq h^{1/m} \label{quant-cond-2}
\een
for $m \in \ZZ$, $m >1$, then
there is a family of unitary $h$-FIOs $V(z)$ such that
\ben
\label{Mz-n-form}
e^{iz/h} M(z) = V(z)^{-1}e^{-i (Q(z,h)-z)/h} V(z) + \O_{L^2 \to L^2}(h^\infty),
\een
where
\ben
\label{Qzh-n-form}
Q(z,h) & = & \sum_{j=0}^\infty h^j q_j(z, I), \text{ with}\\
q_j(z,I) & = & \O(I) \nonumber
\een
and
\be
q_0(z,\imath) = q_z (\imath ).
\ee

Now let $\beta \in \NN^{n-1}$ be a multi-index and define
\be
v_\beta = c_\beta h^{-(n-1)/4} e^{-|x|^2/2h}\prod_{j=1}^{n-1}
H_{\beta_j}(x_j/h^\half ),
\ee
with $H_{\beta_j}$ the Hermite polynomials as in \S \ref{model-case-section} and $c_\beta$ chosen independent of $h$ to normalize $v_\beta$ in
$L^2( \reals^{n-1})$.  The functions $v_\beta$ satisfy
\be
I_j v_\beta = h ( 2 \beta_j + 1 ) v_\beta,
\ee
and with $\11= (1, \ldots, 1) \in \NN^{n-1}$ we write
\be
I v_\beta = h (2 \beta + \11 ) v_\beta.
\ee
Hence we have
\ben
Q(z,h) v_\beta & = & \left(\sum_{j=0}^\infty h^j q_j(z, h(2 \beta + \11))
\right) v_\beta \nonumber \\
& =: & \zeta_\beta(z) v_\beta, \label{zeta-beta-def}
\een
where
\be
\zeta_\beta(z) = h \sum_{j=1}^{n-1} \lambda_j(z) (2 \beta_j + 1) +
\O(h^2).
\ee
The quantization condition \eqref{quant-cond-2} implies we have the
restriction on $\zeta_\beta$:
\be
| h \sum_{j=1}^{n-1} \lambda_j(z) (2 \beta_j + 1)  | \leq C h^{1/m},
\ee
for $0 < {1/m} <1$, giving
\ben
\# \{ \zeta_\beta(z) \} & = & \#\left\{ \left|h \sum_{j=1}^{n-1} \lambda_j(z) (2 \beta_j +
1) \right| \leq C h^{1/m} \right\} \nonumber \\
& \simeq & \# \{ |\beta| \leq h^{{1/m} -1} \} \nonumber \\
& \simeq & h^{({1/m} -1)(n-1)} + o(1). \label{zeta-cond}
\een

\subsection{The proof of Theorem \ref{main-theorem-1}}
Observe the functions $v_\beta$ constructed above satisfy
\be
\WF v_\beta = (0,0) \in \reals^{2n-2}.
\ee
Beginning with $v_\beta$ we want to construct $\tilde v_{\beta,k}$ and
find values of $z$, $\beta$, and $k \in \ZZ$ so that
\be
(\id-M(z)) \tilde v_\beta = \O(h^\infty).
\ee
Let
\be
\tM(z) = V(z) M(z) V(z)^{-1} = e^{-i (Q(z,h)-z)/h} 
\ee
with $V(z)$ and $Q(z,h)$ as in \eqref{Mz-n-form}, and observe $\tM(z)
= \tM(z,1)$ for 
\be
\tM(z,t) = \exp (-it (Q(z,h)-z)/h)
\ee
satisfying
\ben
\label{tM-eq}
\left\{ \begin{array}{l}
hD_t \tM(z,t) + Q(z,h) \tM(z,t) = z \tM(z,t) \\ \tM(z,0) = \id. \end{array}
\right. 
\een
The spectrum of $hD_t$ on $\reals/ \ZZ$ is $\{ h2 \pi k \}$ for $k \in
\ZZ$, so we want the solution space to \eqref{tM-eq} intersected with the
solution space to 
\be
(e^{iz/h} - \tM(z,1)) v = v
\ee
to contain the ``ansatz'' space  
\ben
\label{v-k-b}
 v_{k, \beta}(t, x) :=  e^{-i t2 \pi k} v_\beta(x) .
\een
More precisely, $v_{k, \beta}(1, x) = v_\beta{x}$, so we want to solve
\be
\left\{ \begin{array}{l}
hD_t \tM(z,t) v_{\beta,k}  + (Q(z,h)-z) \tM(z,t) v_{\beta,k} = - z
\tM(z,t) v_{\beta,k} \\ \tM(z,0) v_{\beta,k} = v_{\beta,k}. \end{array}
\right. 
\ee
That is, we want to find $z$ satisfying
\be
2z - \zeta_\beta(z) = 2 \pi k h,
\ee
where $\zeta_\beta (z)$ is given by \eqref{zeta-beta-def}.

Expanding $Q(z,h)$ in a formal series in $z$ as we may do according to
the quantization condition \eqref{quant-cond-2}, we write
\ben
\label{Q-z-series}
Q(z,h) = \sum_{l=0}^\infty z^l Q_l(h, I) 
\een 
microlocally, with 
\be
Q_0 = \sum_{j=1}^{n-1} \lambda_j(0) I_j + \O(I^2),
\ee
and
\be
Q_l = \O(I).
\ee
Hence we will seek 
\ben
\label{z-h-series}
z_{k, \beta}=  \sum_{j=0}^\infty z_{k,\beta}^{(j)},
\een
with $z_{k,\beta}^{(j)} =
\O(h^{ (j+1)/m})$.  For $z_{k, \beta}^{(0)}$, we solve
\be
2 z_{k, \beta}^{(0)} = h\sum_{j=1}^{n-1} \lambda_j(0) ( 2 \beta_j + 1 ) +
2k \pi h
\ee
which is $\O(h^{1/m})$ if
\ben
\label{k-cond}
|k| \leq C h^{1/m -1}.
\een
For $z_{k,\beta}^{(1)}$ we plug $z_{k, \beta}^{(0)}+ z_{k,
  \beta}^{(1)}$ into \eqref{Q-z-series} to get the equation
\be
2z_{k, \beta}^{(0)}+ 2z_{k,
  \beta}^{(1)} & = & h\sum_{j=1}^{n-1} \lambda_j(0) ( 2 \beta_j + 1 ) +
2k \pi h + \sum_{l=1}^\infty (z_{k, \beta}^{(0)}+ z_{k,
  \beta}^{(1)})^l Q_l(h, h(2\beta + \11)) \\
& = & 2 z_{k, \beta}^{(0)} + z_{k, \beta}^{(0)} Q_l(h, h(2\beta + \11)) +
\O(h^{ 3/m}),
\ee
provided $z_{k,\beta}^{(1)} = \O(h^{2/m})$.  Hence we choose
\be
2 z_{k,\beta}^{(1)} =  z_{k, \beta}^{(0)} Q_l(h, h(2\beta + \11)) .
\ee
Continuing in this fashion, we select $z_{k, \beta}^{(j)}$ for $j \geq
2$ using the following equation:
\be
2 \sum_{r=0}^j z_{k, \beta}^{(r)} = \sum_{r=0}^{j-1} \left( \sum_{l=0}^{j-r-1} z_{k, \beta}^{(l)} \right)^r Q_r(h, h(2 \beta + \11)),
\ee
modulo $\O(h^{(j+2)/m})$, hence $z_{k,\beta}^{(j)} = \O(h^{(j+1)/m})$.

Now there is no reason why \eqref{z-h-series} should converge in any
sense, so we want to find a convergent series
\be
\tz_{k,\beta} =\sum_{j=0}^\infty \tz_{k, \beta}^{(j)}
\ee
with $\tz_{k, \beta}^{(j)} = \O(h^{(j+1)/m})$, satisfying
\ben
\label{z-asymp}
\tz_{k,\beta} - \sum_{0}^{mN} z_{k, \beta}^{(j)} = \O(h^N)
\een
for every $N>0$.
For this, we follow the proof of
Borel's Lemma from \cite{EvZw}.  Choose $\chi \in \Ci_c([-1,2])$
satisfying $\chi \equiv 1$ on $[0,1]$.  Set
\be
\tz_{k,\beta} = \sum_{j=0}^\infty \chi( \lambda_j h ) z_{k,
  \beta}^{(j)},
\ee
where $\lambda_j \to \infty$, $\lambda_j < \lambda_{j+1}$ has yet to be selected.  Observe for each
$h>0$, this is a finite sum, hence converges.  We calculate:
\be
\tz_{k,\beta} - \sum_{j=0}^{mN+m} z_{k, \beta}^{(j)} & = &
\sum_{mN+m+1}^\infty \chi( \lambda_j h ) z_{k,
  \beta}^{(j)} + \sum_0^{mN+m} z_{k, \beta}^{(j)} ( \chi( \lambda_j h) -1)
\\
& =:& A + B.
\ee
But since $x \chi(x)$ is uniformly bounded, we have
\be
|A| & \leq & \sum_{mN+m+1}^\infty C_j h^{(j+1)/m} \frac{ \lambda_j
  h}{\lambda_jh} \chi(\lambda_j h) \\
& \leq & \sum_{mN+m+1}^\infty C_j' h^{(j-m+1)/m} \lambda_j^{-1} \\
& \leq & h^N \sum_{mN+1}^\infty 2^{-j}
\ee
if $\lambda_j$ is sufficiently large.  

To estimate $B$, we observe for $0 < \lambda_{mN+m} h \leq 1$, $B = 0$
since $\chi \equiv 1$ on $[0,1]$.  If $\lambda_{mN+m} < h$, we calculate
\be
|B| & \leq &  \sum_0^{mN+m} C_j
h^{(j+1)/m}  ( \chi( \lambda_j h) -1) \\
& \leq & C_N h^{1/m} \lambda_{mN+m}^{N} h^N,
\ee
which is \eqref{z-asymp}.

Now for fixed $(\beta, k)$ and $N>0$, 
 we have the crude estimate
\be
\tz_{\beta,k}^l - \left(\sum_{j=0}^{mN+m} (z_{k,\beta}^{(j)}) \right)^l
= \left(\tz_{\beta,k} - \sum_{j=0}^{mN+m} (z_{k,\beta}^{(j)})\right)( l \O(1)),
\ee
which from the definitions of $z_{k,\beta}$, $\tz_{k,\beta}$, and $Q_l(h,I)$ gives:
\be
\lefteqn{ hD_t \tM(\tz_{k,\beta},t) v_\beta + (Q(\tz_{k,\beta},t)-\tz_{k,\beta})
\tM(\tz_{k,\beta},t) v_\beta = } \\
& = &  hD_t \tM(\tz_{k,\beta},t) v_\beta \\
&& \quad + \left( \sum_{l=0}^{mN+m}
  \left(\sum_{j=0}^{mN} (z_{k,\beta}^{(j)}) \right)^l Q_l(h,
  h(2\beta + \11)) - \sum_{j=0}^{mN+m} (z_{k,\beta}^{(j)}) \right) \tM(
\tz_{k, \beta},t) v_\beta \\
&& \quad + \O(h^N) \| v_\beta \|_{L^2( \reals^{n-1})}
  \\
& = &  (2k \pi h - \tz_{k,\beta})
\tM(\tz_{k,\beta},t) v_\beta + \O(h^N) \| v_\beta \|_{L^2(
  \reals^{n-1})}.
\ee
Hence 
\be
\tM(\tz_{k,\beta},t) v_\beta = e^{i t (2 \pi k -\tz_{k,\beta}/h)}  v_\beta + t\O(h^{N-1})
\| v_\beta \|_{L^2( \reals^{n-1})},
\ee
so
\be
(e^{i\tz_{k ,\beta}/h} - \tM (\tz_{k,\beta})) v_\beta = \O(h^{N-1})\| v_\beta\|_{L^2(\reals^{n-1})}
\ee
for any $N$, or
\be
(e^{i\tz_{k,\beta}/h} - \tM (\tz_{k,\beta})) v_\beta = \O(h^{\infty})\| v_\beta\|_{L^2(\reals^{n-1})}.
\ee

Now the definition of $\tM$ implies
\be
M(\tz_{k,\beta}) V(\tz_{k,\beta})^{-1} v_\beta =  V(\tz_{k,\beta})^{-1}
v_\beta + \O(h^\infty) \| V(\tz_{k,\beta})^{-1} v_\beta\|_{L^2(\reals^{n-1})},
\ee
so
\be
u_{\tz_{k,\beta}}:= E_+ V(\tz_{k,\beta})^{-1} v_\beta,
\ee
with $E_+$ defined in \eqref{Eplus-def} satisfies \eqref{ev-prob-2}.

Finally, the quantization conditions
(\ref{quant-cond-1}-\ref{quant-cond-2}) and the estimates \eqref{zeta-cond} and
\eqref{k-cond} give
\be
\# \{ z : (Q(z,h)-z)v = \O(h^\infty) \} \geq C^{-1} h^{-n(1-1/m)},
\ee
which is \eqref{ev-est-1}.


\begin{center}
\end{center}
\newpage

\appendix

\addcontentsline{toc}{chapter}{Appendices}

\numberwithin{equation}{section}

\section{Tools from ODE Theory}
Next we consider the ordinary differential equation problem
\ben
\label{ODE-1}
\left\{ \begin{array}{l} \phi ' (t) = A(t) \phi(t), \\ \phi(0) =
  \id , \end{array} \right.
\een
where $A(t): \reals \to M_n$ is an $n \times n$ matrix-valued
function.  It is well known that if $A(t)$ is smooth and grows at most
polynomially in $t$ we can solve \eqref{ODE-1} for all $t$ using an
iterated integral.  More precisely, let
\be
A_0 & = & \id, \\
A_1(t) & = & \int_0^t A(t_1) d t_1, \text{ and} \\
A_k(t) & = & \int_0^t A(t_1)A_{k-1}(t_1) dt_1, \text{ for } k \geq 2.
\ee
Then for $k \geq 1$, 
\be
\frac{d}{dt} A_k(t) = A(t) A_{k-1}(t),
\ee
and 
\be
\phi(t) = \sum_{k=0}^\infty A_k(t)
\ee
solves \eqref{ODE-1}.  To see the series converges for any $t$,
observe that since
\be
\|A(t)\|_\infty \leq C t^N
\ee
for some $N$, by induction we have the estimate
\be
\|A_k(t)\|_\infty \leq C \frac{N!t^{N+k}}{(N+k)!} \leq C_N \frac{t^{N+k}}{k!}.
\ee
This is the classical ``time-ordered integral'' solution to \eqref{ODE-1}.

Motivated by the requirements of \S \ref{theorem-proof-chapter} we
want to consider the following problem: Find a smooth matrix-valued
function $B(t): [0,1] \to M_n$ such that $B(t)$ has compact support in
$(0,1)$ and we can solve
\ben
\label{ODE-2}
\left\{ \begin{array}{l} \psi'(t) = B(t) \psi(t) \\
\psi(0) = \id, \,\, \psi(1) = \phi(1), \end{array} \right.
\een
where $\phi(t)$ solves \eqref{ODE-1}.  
\begin{lemma}
\label{ODE-deform-lemma}
There exists $B(t) \in \Ci_c((0,1);M_n)$ and $ \psi(t) \in \Ci([0,1];
M_n)$ satisfying \eqref{ODE-2}.
\end{lemma}

\begin{proof}
Let $\chi(t) \in \Ci([0,1]; [0,1])$
satisfy

(i)  $\chi \equiv 0$ for $t \in [0,1/3]$, 

(ii)  $\chi \equiv 1$ for $t \in [2/3, 1]$,

(iii)  $\chi' >0$ for $t \in (1/3, 2/3)$.

Then $\psi(t) = \phi( \chi(t))$ satisfies \ref{ODE-2} with 
\be
B(t) = \chi'(t) A(\chi(t)).
\ee
\end{proof}

\section{Tools from Symplectic Geometry}
\numberwithin{equation}{section}

We will need some facts about Hamiltonian vector fields.  The
following Lemma is known as Jacobi's identity.
\begin{lemma}
\label{H-push-lemma}
Suppose $\kappa : U \to V$ is a symplectomorphism, $q \in \Ci(U)$, and
$H_q$ is the Hamiltonian vector field of $q$.  Then
\be
\kappa_* H_q = H_{(\kappa^{-1})^* q}.
\ee
\end{lemma}
\begin{proof}
Let $\omega$ be the symplectic structure on $V$.  That $\kappa$ is a
symplectomorphism means for $p \in U$, 
\be
(\kappa^*  \omega )|_p = \omega|_{\kappa(p)}.
\ee
Let $Y$ be a vector field on $V$.  We need to determine how the
$1$-form $(\kappa_* H_q) \lrcorner \omega$ acts on $Y$.  We calculate
for $p' \in V$:
\be
\omega_{p'} \left( [\kappa_* H_q]|_{p'}, Y|_{p'} \right) & =
& \omega|_{p'} \left( [\kappa_* H_q]|_{p'}, \kappa_*
  [\kappa^{-1}_*Y |_{\kappa^{-1}(p')}] \right) \\
& = & (\kappa^* \omega)|_{\kappa^{-1}(p')} \left( H_q|_{\kappa^{-1}(p')}, (\kappa^{-1}_* Y)|_{\kappa^{-1}(p')} \right)
\\ 
& = & Y (q(\kappa^{-1}))|_{p'}.
\ee
\end{proof}


\section{Semi-hyperbolic geodesics in $3$ dimensions}

In this appendix, we modify the example of Colin de Verdi\`ere-Parisse
\cite{CVP} to extend to three dimensions and have a semi-hyperbolic geodesic.  

Consider the Riemannian manifold
\be
M = \reals_x / \ZZ \times \reals_y \times \reals_z
\ee
equipped with the metric
\be
ds^2 = \cosh^2y (2z^4 - z^2 +1)^2 dz^2 + dy^2 + dz^2.
\ee
Thus the matrix for the metric
\be
g_{ij} = \left\{ \begin{array}{l} \cosh^2y (2z^4 - z^2 +1)^2,\,\,
    i=j=1, \\ 1, \,\, i=j=2,3, \\
    0, \,\, i \neq j. \end{array} \right. ,
\ee
and we calculate the Christoffel symbols:
\be
&& \Gamma_{2,1}^1 = \Gamma_{1,2}^1 = \tanh y, \\
&& \Gamma_{3,1}^1 = \Gamma_{1,3}^1 = (8z^3 - 2z)(2z^4-z^2+1)^{-1}, \\
&& \Gamma_{1,1}^2 = -\sinh y \cosh y (2z^4 -z^2 +1)^2, \\
&& \Gamma_{1,1}^3 = -(8z^3-2z)(2z^4 -z^2 +1) \cosh^2y,
\ee
with all other Christoffel symbols equal to zero.  The geodesic
equtions are
\be
\ddot{x} & = & -2 (\tanh y) \dot{y} \dot{x} -2 ((8z^3 - 2z)(2z^4-z^2+1)^{-1})
\dot{z} \dot{x} \\
\ddot{y} & = & \sinh y \cosh y (2z^4 -z^2 +1)^2 \dot{x}^2 \\
\ddot{z} & = &  (8z^3-2z)(2z^4 -z^2 +1) \cosh^2y \dot{x}^2.
\ee
Setting $v_x = \dot{x}$, $v_y = \dot{y}$, and $v_z = \dot{z}$, we get
the first order system
\be
\dot x & = & v_x, \\
\dot v_x & = & -2 (\tanh y) v_y v_x -2 ((8z^3 - 2z)(2z^4-z^2+1)^{-1})
v_z v_x, \\
\dot y & = & v_y, \\
\dot v_y & = & \sinh y \cosh y (2z^4 -z^2 +1)^2 v_x^2, \\
\dot z & = & v_z, \\
\dot v_z & = & (8z^3-2z)(2z^4 -z^2 +1) \cosh^2y v_x^2.
\ee
There are trivially three periodic geodesics, given by the solutions
\be
x(t) & = & v_x(0) t + x(0), \\
y(t) & = & 0, \\
z(t) & = & 0, \pm 1/2.
\ee

\begin{figure}
\centerline{
\input{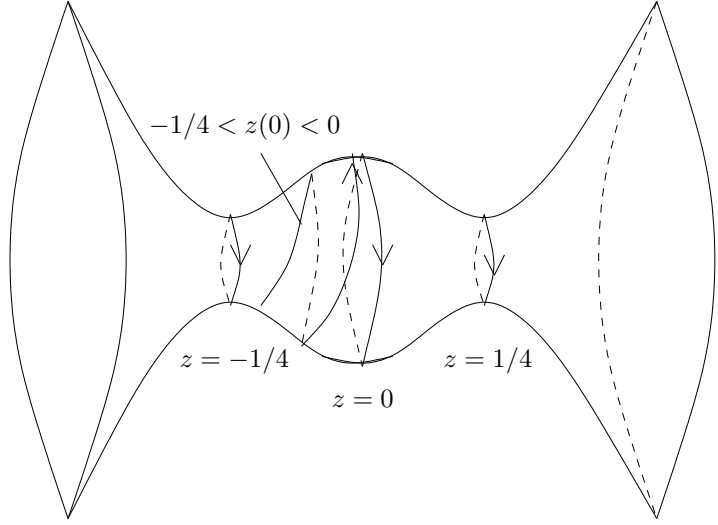}}
\caption{\label{fig:fig20} The $x\text{-}z$ hypersurface in $M$ with some
  representative orbits.}
\end{figure}

Next we examine the Laplace-Beltrami operator on $M$.  We compute
\be
\Delta & = & |g|^{-1/2} \partial_i |g|^{1/2} g^{ij} \partial_j \\
& = & \cosh^{-2} y (2z^4 -z^2 +1)^{-2} \partial_x^2 + \partial_y^2
+ \partial_z^2 \\
&& \quad + \tanh y \partial_y + (8z^3 -2z) (z^4 - z^2
+1)^{-1} \partial_z.
\ee
The isometry $T: L^2(M, d\text{Vol}_g) \to L^2(M, dx \,
dy \, dz)$ given by
\be
Tu(x,y,z) = \cosh^{1/2} y (2z^4 - z^2 +1)^{1/2} u(x,y,z)
\ee
conjugates $\Delta$ into a self-adjoint operator $\tDelta$.  A
computation yields
\be
\tDelta & = & T \Delta T^{-1} \\
& = & \cosh^{-2} y (2z^4 - z^2 +1)^{-2} \partial_x^2 + \partial_y^2
+ \partial_z^2 -\frac{1}{4} ( 1 + \sech^2 y) \\
&& \quad + \frac{1}{4}
(8z^3-2z)^2(2z^4 -z^2 +1)^{-2} - \frac{1}{2} (24z^2-2)(2z^4 -z^2
+1)^{-1}.
\ee
In keeping with the theme of this work, we want to examine asymptotic
behaviour of eigenfunctions for this operator.  In order to separate
variables, let
\be
\phi_{k, \lambda}(x,y,z) = e^{ikx} \psi_{k, \lambda}(y,z), 
\ee
and compute:
\be
\lefteqn{ - \tDelta \phi_{k, \lambda} = } \\
& = & \Big( - \Delta_{y,z}
  +k^2\cosh^{-2} y (2z^4 - z^2 +1)^{-2} +\frac{1}{4} ( 1 + \sech^2 y)
  \\
&&  - \frac{1}{4}
(8z^3-2z)^2(2z^4 -z^2 +1)^{-2} + \frac{1}{2} (24z^2-2)(2z^4 -z^2
+1)^{-1} \Big) \phi_{k, \lambda}.
\ee
Rearranging, we have the following equation for $\psi_{k,\lambda}$:
\be
&& \Big( - \Delta_{y,z}  +k^2 (\cosh^{-2} y (2z^4 - z^2 +1)^{-2}-1) \\
&& \quad  +\frac{1}{4} ( 1 + \sech^2 y) - \frac{1}{4}
(8z^3-2z)^2(2z^4 -z^2 +1)^{-2}  \\
&& \quad + \frac{1}{2} (24z^2-2)(2z^4 -z^2
+1)^{-1} \Big) \psi_{k, \lambda} \\
&& \quad \quad \quad \quad \quad \quad \quad \quad \quad \quad \quad
\quad \quad  = ( \lambda -k^2) \psi_{k, \lambda}.
\ee
We divide by $k^2$ and use $h = 1/k$ as the semiclassical parameter, giving
\be
P(h) \psi_h & = & (-h^2 \Delta_{yz} + V(y,z) ) \psi_h \\
& = & ( h^2 \lambda -1) \psi_h,
\ee
with
\be
V(y,z) & = &   \cosh^{-2} y (2z^4 - z^2 +1)^{-2}-1 + \\
&& h^2\frac{1}{4} ( 1 + \sech^2 y) - h^2\frac{1}{4}
(8z^3-2z)^2(2z^4 -z^2 +1)^{-2}  \\
&& + h^2\frac{1}{2} (24z^2-2)(2z^4 -z^2
+1)^{-1}.
\ee
The semiclassical principal symbol is
\be
\sigma_h(P) & = & \eta^2 + \zeta^2 + \cosh^{-2} y (2z^4 - z^2
+1)^{-2}-1 \\
& =: & \eta^2 + \zeta^2 + \tV.
\ee
Observe $\tV$ has nondegenerate critical points at $y=0$, $z = 0, \pm
1/4$.  The signatures of $\partial^2 \tV$ are $(-,+)$,
$(-,-)$, and $(-,-)$, respectively.  Thus the quadratic part of the
normal forms for $\sigma_h(P)$ takes the form 
\be
&& \lambda_1 y \eta + \frac{\lambda_2}{2}(z^2
+ \zeta^2), \text{ near }y=0, z=0, \text{ and} \\
&& \lambda_1 y \eta + \lambda_2 z \zeta, \text{ near } y =0, z = \pm
1/4.
\ee


\end{document}